\documentclass[12pt,fullpage,doublespace]{article}
\usepackage{graphicx} 
\usepackage{setspace} 
\usepackage{latexsym} 
\usepackage{amsfonts} 
\usepackage{amsmath} 
\usepackage{amssymb}
\usepackage{accents}
\usepackage{textcomp}
\usepackage{undertilde}
\usepackage{enumerate}
\usepackage{bm}
\usepackage{stmaryrd}
\usepackage[margin=1in]{geometry}
\usepackage{amsthm}
\usepackage{ifsym}
\usepackage{amssymb,latexsym,amsmath}
\usepackage{graphics}
\usepackage{tikz}
\usetikzlibrary{matrix,arrows,positioning,scopes}

\newtheorem{theorem}{Theorem}[section]
\newtheorem{definition}[theorem]{Definition}
\newtheorem{proposition}[theorem]{Proposition}
\newtheorem{lemma}[theorem]{Lemma}

\newtheorem{remark}[theorem]{Remark}

\numberwithin{equation}{section}

\renewcommand{\gg}{\gamma}

\newcommand{\rest}{\restriction}
\newcommand{\la}{\langle}
\newcommand{\ra}{\rangle}

\newcommand{\cp}{{\rm cp }}

\newcommand{\lh}{{\rm lh}}

\def\k{\kappa}
\def\a{\alpha}
\def\b{\beta}
\def\d{\delta}

\def\l{\lambda}

\renewcommand{\models}{\vDash}
\newcommand{\powerset}{{\cal P}}

\def\P{{\mathcal{P} }}
\def\W{{\mathcal{W} }}
\def\Q{{\mathcal{ Q}}}

\def\K{{\mathcal{ K}}}

\def\R{{\mathcal R}}

\def\M{{\mathcal{M}}}
\def\N{{\mathcal{N}}}
\def\F{{\mathcal{F}}}
\def\T {{\mathcal{T}}}
\def\U{{\mathcal{U}}}
\def\S{{\mathcal{S}}}
\def\F{{\mathcal{F}}}

\def\VT{{\vec{\mathcal{T}}}}

\newcommand{\rthm}[1]{Theorem~\ref{#1}}

\newcommand{\rdef}[1]{Definition~\ref{#1}}
\newcommand{\rfig}[1]{Figure~\ref{#1}}

 \input xy
 \xyoption{all}
\begin{document}
\title{HOD in Natural Models of $\textsf{AD}^+$}
\date{\today}
\author{Nam Trang\\
Department of Mathematical Sciences\\
Carnegie Mellon University\\
namtrang@andrew.cmu.edu.}
\maketitle
\begin{abstract}
The goal of this paper is to compute the full HOD of models of $\textsf{AD}^+$ of the form $L(\powerset(\mathbb{R}))$ below $``\textsf{AD}_\mathbb{R} +\Theta$ is regular". As part of this computation, we give a computation of HOD$|\Theta$ left open in \cite{ATHM} for $\Theta$ a successor in the Solovay sequence.
\end{abstract}
\thispagestyle{empty}
\normalsize

Throughout this paper, unless stated otherwise, we assume $V=L(\powerset(\mathbb{R}))+ \textsf{AD}^+$ + no $\textsf{AD}^+$ models $M$ containing $\mathbb{R}\cup \textrm{OR}$ satisfying $``\textsf{AD}_\mathbb{R} + \Theta$ is regular"\footnote{Under our hypothesis, ``Strong Mouse Capturing" ($\textsf{SMC}$) holds. This notion will be introduced in Section 1.}. We call this assumption (\textasteriskcentered). Under (\textasteriskcentered), we analyze full \textrm{HOD}, extending the analysis in \cite{ATHM}. Our smallness assumption is made because of the fact that for our computation, we'll rely heavily on the theory of hod mice, which is developed in \cite{ATHM} for models satisfying the assumption.
\\
\indent To put this work in a proper context, we recall a bit of history on the computation of \textrm{HOD}. In $L(\mathbb{R})$ under $\textsf{AD}$\footnote{It's known that if $L(\mathbb{R})\vDash \textsf{AD}$ then $L(\mathbb{R})\vDash \textsf{AD}^+$.}, Harrington and Kechris show that \textrm{HOD} $\vDash \textsf{CH}$. Let $\kappa = \omega_1^{L(\mathbb{R})}$. Solovay shows that \textrm{HOD} $\vDash \kappa$ is measurable and Becker shows $\kappa$ is the least measurable in \textrm{HOD}. These were shown using descriptive set theory. Then Steel in \cite{steel2010outline} or \cite{steel03} using inner model theory shows $V^{\textrm{HOD}}_\Theta$ is a fine-structural mouse, which in particular implies $V^{\textrm{HOD}}_\Theta\vDash \textsf{GCH}$. Woodin (see \cite{WoodinHod}), building on Steel's work, completes the full \textrm{HOD} analysis in $L(\mathbb{R})$ and shows \textrm{HOD} $\vDash \textsf{GCH}$ and furthermore shows that the full \textrm{HOD} of $L(\mathbb{R})$ is a hybrid mouse that contains some information about a certain iteration strategy of its initial segments. A key fact used in the computation of \textrm{HOD} in $L(\mathbb{R})$ is that if $L(\mathbb{R})\vDash \textsf{AD}$ then $L(\mathbb{R})\vDash \textsf{MC}\footnote{\textsf{MC} stands for Mouse Capturing, which is the statement that if $x,y\in \mathbb{R}$, then $x\in OD(y) \Leftrightarrow x$ is in a mouse over $y$.}$. It's natural to ask whether analogous results hold in the context of $\textsf{AD}^+ + V=L(\powerset(\mathbb{R}))$ as the \textrm{HOD} computation is an integral part of the structural analysis of $\textsf{AD}^+$ models and plays an important role in applications such as the core model induction. Woodin has shown that under this assumption \textrm{HOD} $\vDash \textsf{CH}$. Recently, Grigor Sargsyan in \cite{ATHM}, assuming (\textasteriskcentered), proves Strong Mouse Capturing (S\textsf{MC}) (a generalization of \textsf{MC}) and computes $V^{\textrm{HOD}}_\Theta$ for $\Theta$ being limit in the Solovay sequence and $V^{\textrm{HOD}}_{\theta_\alpha}$ for $\Theta = \theta_{\alpha+1}$ in a similar sense as above under the assumption (\textasteriskcentered). 
\\
\indent This paper extends Sargsyan's work to the computation of full \textrm{HOD} under (\textasteriskcentered). There are two main cases. We show that if $\Theta$ is $\theta_0$ or is a successor in the Solovay sequence, \textrm{HOD} is of the form $L[\mathcal{M}_\infty][\Sigma_\infty]$, where $\mathcal{M}_\infty$ is a fine structural premouse (hybrid premouse if $\Theta > \theta_0$) extending \textrm{HOD}$|\Theta$. The definition of $\M_\infty$ will be spelled out in detail during the course of the paper. $\Sigma_\infty$ is a fragment of the strategy for $\mathcal{M}_\infty|\Theta$ on (finte stacks of) normal trees in $\mathcal{M}_\infty$. For clarity, we devote the entire Section 2 to the computation of \textrm{HOD} for $\textsf{AD}^+$ models satisfying $\Theta = \theta_0$. In Section 3, we bring in the machinery of hod mice developed in \cite{ATHM} and combine it with techniques of Section 2 to compute \textrm{HOD} for $\textsf{AD}^+$ models satisfying $\Theta = \theta_{\alpha+1}$ for some $\alpha$. Though in addition to (\textasteriskcentered), we need an additional assumption; this assumption is explained in Section 3. The case $\Theta$ is a limit in the Solovay sequence, i.e. $\Theta = \theta_\alpha$ for some limit $\alpha$, is dealt with in Section 4. There the $\textrm{HOD}$ computation is split into two cases depending on whether or not $\textrm{HOD} \vDash \textrm{cof}(\Theta)$ is measurable. 
\\
\indent This work is done when the author is a graduate student at UC Berkeley under the supervision of Professor John Steel. The author would like to thank him for suggesting this topic, his patience, and numerous helpful advice during the course of this project. The extent to which this paper is in debt to Grigor Sargsyan's work on hod mice will be apparent in Chapters 3 and 4. The author would also like to thank him for numerous suggestions and corrections on an older version of this paper.
\section{Backgrounds}
\subsection{Basic facts about $\textsf{AD}^+$ and hod mice}
\noindent \index{\textsf{AD}$^+$} We start with the definition of Woodin's theory of $\textsf{AD}^+$. In this paper, we identify $\mathbb{R}$ with $\omega^{\omega}$. We use $\Theta$ to denote the sup of ordinals $\alpha$ such that there is a surjection $\pi: \mathbb{R} \rightarrow \alpha$.
\begin{definition}
\label{AD+}
$\textsf{AD}^+$ is the theory $\textsf{ZF} + \textsf{AD} + \textsf{DC}_{\mathbb{R}}$ and 
\begin{enumerate}
\item for every set of reals $A$, there are a set of ordinals $S$ and a formula $\varphi$ such that $x\in A \Leftrightarrow L[S,x] \vDash \varphi[S,x]$. $(S,\varphi)$ is called an $\infty$-Borel code\index{$\infty$-Borel code} for $A$;
\item for every $\lambda < \Theta$, for every continuous $\pi: \lambda^\omega \rightarrow \omega^{\omega}$, for every $A \subseteq \mathbb{R}$, the set $\pi^{-1}[A]$ is determined.
\end{enumerate}
\end{definition}
\noindent \textsf{AD}$^+$ is arguably the right structural strengthening of \textsf{AD}. In fact, \textsf{AD}$^+$ is equivalent to ``$\textsf{AD}\ +\ $the set of Suslin cardinals is closed" (see \cite{ketchersid2011more}). Another, perhaps more useful, equivalence of $\textsf{AD}^+$ is ``$\textsf{AD} + \Sigma_1$ statements reflect to Suslin-co-Suslin" (see \cite{steelderived} for a more precise statement).
\begin{definition}[\textsf{AD}$^+$]
\label{Solovaysequence}
The \textbf{Solovay sequence} is the sequence $\langle\theta_\alpha \ | \ \alpha \leq \Omega\rangle$ where
\begin{enumerate}
\item $\theta_0$ is the sup of ordinals $\beta$ such that there is an $OD$ surjection from $\mathbb{R}$ onto $\beta$;
\item if $\alpha>0$ is limit, then $\theta_\alpha = \sup\{\theta_\beta \ | \ \beta<\alpha\}$;
\item if $\alpha =\beta + 1$ and $\theta_\beta < \Theta$ (i.e. $\beta < \Omega$), fixing a set $A\subseteq \mathbb{R}$ of Wadge rank $\theta_\beta$, $\theta_\alpha$ is the sup of ordinals $\gamma$ such that there is an $OD(A)$ surjection from $\mathbb{R}$ onto $\gamma$, i.e. $\theta_\alpha = \theta_A$.
\end{enumerate}
\end{definition}
Note that the definition of $\theta_\alpha$ for $\alpha = \beta+1$ in Definition \ref{Solovaysequence} does not depend on the choice of $A$. We recall some basic notions from descriptive set theory.
\\
\indent Suppose $A\subseteq \mathbb{R}$ and $(N, \Sigma)$ is such that $N$ is a transitive model of ``$\textsf{\textsf{ZF}C}-Replacement$" and $\Sigma$ is an $(\omega_1, \omega_1)$-iteration strategy or just $\omega_1$-iteration strategy for $N$. We use $o(N)$, $\textrm{OR}^N$, $\textrm{ORD}^N$ interchangably to denote the ordinal height of $N$. Suppose that $\d$ is countable in $V$ but is an  uncountable cardinal of $N$ and suppose that $T, U\in N$ are trees on $\omega\times (\d^+)^N$. We say $(T,U)$ \textit{locally Suslin captures  $A$ at $\delta$} over $N$ if for any $\a\leq\d$ and for $N$-generic $g\subseteq Coll(\omega, \a)$,
\begin{center}
  $A\cap N[g]=p[T]^{N[g]} = \mathbb{R}^{N[g]} \backslash p[U]^{N[g]}$.
\end{center}
We also say that $N$ locally Suslin captures $A$ at $\d$.
We say that $N$ locally captures $A$ if $N$ locally captures $A$ at any uncountable cardinal of $N$. We say $(N, \Sigma)$ \textit{Suslin captures}\index{Suslin capturing} $A$ at $\delta$, or $(N, \d, \Sigma)$  \textit{Suslin captures}\index{Suslin capturing} $A$, if there are trees $T, U\in N$ on $\omega\times (\d^+)^N$ such that whenever $i: N\rightarrow M$ comes from an iteration via $\Sigma$, $(i(T), i(U))$ locally Suslin captures $A$ over $M$ at $i(\d)$. In this case we also say that $(N, \d, \Sigma, T, U)$ Suslin captures $A$. We say $(N, \Sigma)$ Suslin captures $A$ if for every countable $\d$ which is an uncountable cardinal of $N$, $(N, \Sigma)$ Suslin captures $A$ at $\d$. When $\d$ is Woodin in $N$, one can perform genericity iterations on $N$ to make various objects generic over an iterate of $N$. This is where the concept of Suslin capturing becomes interesting and useful. We'll exploit this fact on several occasions.

We say that $\Gamma$ is a \textbf{good pointclass}\index{good pointclass} if it is closed under recursive preimages, closed under $\exists^{\mathbb{R}}$, is $\omega$-parametrized, and has the scale property. Furthermore, if $\Gamma$ is closed under $\forall^\mathbb{R}$, then we say that $\Gamma$ is \textbf{inductive-like}\index{inductive-like pointclass}.

\begin{theorem}[Woodin, Theorem 10.3 of \cite{DMATM}]\label{n*x} Assume $ \textsf{AD}^+$ and suppose $\Gamma$ is an inductive-like pointclass and is not the last inductive-like pointclass. There is then a function $F$ defined on $\mathbb{R}$ such that for a Turing cone of $x$, $F(x)=\la \N^*_x, \M_x, \d_x, \Sigma_x\ra$ such that
\begin{enumerate}
\item $\N^*_x|\d_x=\M_x| \d_x$,
\item $\N^*_x\models ``\textsf{ZF} + \d_x$ is the only Woodin cardinal",
\item $\Sigma_x$ is the unique iteration strategy of $\M_x$,
\item $\N^*_x= L(\M_x, \Lambda)$ where $\Lambda$ is the restriction of $\Sigma_x$ to stacks $\VT\in \M_x$ that have finite length and are based on $\M_x\rest \d_x$,
\item $(\N^*_x, \Sigma_x)$ Suslin captures $\Gamma$,
\item for any $\a<\d_x$ and for any $\N^*_x$-generic $g\subseteq Coll(\omega, \a)$, $(\N^*_x[g], \Sigma_x)$ Suslin captures $Code((\Sigma_x)_{\M_x\rest\a})$ and its complement at $\d_x^+$.
\end{enumerate}
\end{theorem}
\begin{theorem}[Woodin, unpublished but see \cite{steelderived}]\label{fundamental result of ad+} Assume $ \textsf{AD}^++V= L(\powerset(\mathbb{R}))$. Suppose $A$ is a set of reals such that there is a Suslin cardinal in the interval $(w(A), \theta_A)$. Then
\begin{enumerate}
\item The pointclass $\undertilde{\Sigma^2_1}(A)$ has the scale property.
\item $M_{\utilde{\Delta}^2_1(A)}\prec_{\Sigma_1} L(\powerset(\mathbb{R}))$.
\item $L_{\Theta}(\powerset(\mathbb{R}))\prec_{\Sigma_1} L(\powerset(\mathbb{R}))$.
\end{enumerate}
\end{theorem}

We quote another theorem of Woodin, which will be key in our HOD analysis.
\begin{theorem}[Woodin, see \cite{koellner2010large}]
\label{succThetaWoodin}
Assume $\textsf{AD}^+$. Let $\langle\theta_\alpha \ | \ \alpha \leq \Omega\rangle$ be the Solovay sequence. Suppose $\alpha = 0$ or $\alpha = \beta+1$ for some $\beta < \Omega$. Then $HOD \vDash \theta_\alpha$ is Woodin.
\end{theorem}
Next, we prove the following theorem of Woodin's which roughly states that \textrm{HOD} is coded into a subset of $\Theta$.
\begin{theorem}[Woodin]
\label{WoodinVopenka}
Assume $\textsf{AD}^+ + V = L(\powerset(\mathbb{R}))$. Then \textrm{HOD} $= L[P]$ for some $P \subseteq \Theta$ in \textrm{HOD}.
\end{theorem}
\begin{proof}
First, let
\begin{equation*}
\mathbb{P} = \{(\vec{\alpha}, \vec{a}) \ | \ \vec{\alpha} = \langle\alpha_0,\alpha_1,...,\alpha_n\rangle \in \Theta^{<\omega}, \vec{a} = \langle a_0,a_1,...,a_n\rangle, \forall i \leq n (a_i \subseteq \alpha_i)\}.
\end{equation*}
$\mathbb{P}$ is a poset with the (obvious) order by extension. If $g$ is a  $\mathbb{P}$-generic over $V$ then $g$ induces an enumeration of order type $\omega$ of $(\Theta, \cup_{\gamma<\Theta} \powerset(\gamma))$. Now let
\begin{equation*}
\mathbb{Q}^* = \{(\vec{\alpha}, A) \ | \ \vec{\alpha} = \langle\alpha_0,\alpha_1,...,\alpha_n\rangle \in \Theta^{<\omega}, A \subseteq \powerset(\alpha_0)\times \powerset(\alpha_1)\times...\times\powerset(\alpha_n), A \in OD\}.
\end{equation*}
The ordering on $\mathbb{Q}^*$ is defined as follows: 
\begin{eqnarray*}
(\vec{\alpha},A) \leq (\vec{\beta},B) \Leftrightarrow \forall i < \textrm{dom}(\vec{\alpha}) \vec{\alpha}(i) = \vec{\beta}(i), B|(\powerset(\vec{\alpha}(0))\times...\times \powerset(\vec{\alpha}(\textrm{dom}(\alpha)-1)))\footnote{Suppose $m\leq n$ and $B\subseteq \powerset(\alpha_0)\times\dots \times \powerset(\alpha_n)$. Then $B|\powerset(\alpha_0)\times \dots \times \powerset(\alpha_m)$ is the set of $t\rest (m+1)$ for $t\in B$.} \subseteq A.
\end{eqnarray*}
There is a poset $\mathbb{Q} \in \textrm{HOD} \cap \powerset(\Theta)$ that is isomorphic to $\mathbb{Q}^*$ via an OD map $\pi$. For our convenience, whenever $p \in \mathbb{Q}$, we will write $p^*$ for $\pi(p)$. Furthermore, we can define $\pi$ so that elements of $\mathbb{Q}$ have the form $(\vec{\alpha}, A)$ whenever $p^* = (\vec{\alpha}, A^*)$. In other words, we can think of $\pi$ as a bijection of $\Theta$ and the set of $OD$ subsets of $\powerset(\alpha_0)\times\powerset(\alpha_1)\times...\times\powerset(\alpha_n)$ for $\alpha_0,\alpha_1,...,\alpha_n < \Theta$. For notational simplicity, if $p^* = (\vec{\alpha}, A^*)$, we write $o(p^*)$ for $\vec{\alpha}$ and $s(p^*)$ for $A^*$.
\\
\\
\noindent \textbf{Claim.} \textit{ Let $g$ be $\mathbb{P}$-generic over $V$. Then $g$ induces a $\mathbb{Q}$-generic $G_g$ over \textrm{HOD}. In fact, for any condition $q\in \mathbb{Q}$, we can find a $\mathbb{P}$-generic $g$ over $V$ such that $q\in G_g$ and $G_g$ is a $\mathbb{Q}$-generic over \textrm{HOD}.}
\begin{proof}
As mentioned above, $g$ induces a generic enumeration $f$ of $(\Theta, \cup_{\gamma<\Theta}\powerset(\gamma))$ of order type $\omega$. Furthermore, for each $n < \omega$, $f(n)_0 < \Theta$ and $f(n)_1 \subseteq f(n)_0$. Let
\begin{equation*}
G = \cup_{n<\omega}\{(\langle f(0)_0,...,f(n)_0\rangle, A) \in \mathbb{Q} \ | \ \langle f(0)_1,...,f(n)_1\rangle \in A^*\}.
\end{equation*}
We claim that $G$ is $\mathbb{Q}$-generic over $\textrm{HOD}$. To see this, let $D \subseteq \mathbb{Q}$, $D\in \textrm{HOD}$ be a dense set. Let $p = f|(n+1)$ for some $n$. It's enough to find a $q=(\langle\alpha_0,...,\alpha_m\rangle,\langle a_0,...,a_m\rangle) \in \mathbb{P}$ extending $p$ such that $D_q \cap D \neq \emptyset$ where
\begin{equation*}
D_q = \{ (\langle\alpha_0,...,\alpha_m\rangle,A) \ | \ \langle a_0,...,a_m\rangle \in A^*\}.
\end{equation*}
If no such $q$ exists, let $r = (\langle f(0)_0,...,f(n)_0\rangle,B)$, where
\begin{equation*}
b\in B^* \Leftrightarrow \forall t \in D \forall c(b^\smallfrown c \notin s(t)).
\end{equation*}
Then $r$ is a condition in $\mathbb{Q}$ with no extension in $D$. Contradiction.
\end{proof}
For each $\alpha<\Theta$, $n < \omega$, and $\langle\alpha_0,...,\alpha_n\rangle$, let $A_{\alpha,\langle\alpha_0,...,\alpha_n\rangle} = (\langle\alpha_0,...,\alpha_n\rangle,A)$ such that $\forall a\in A^*(\alpha \in a(n))$. We can then define a canonical term in $\textrm{HOD}$ for a generic enumeration of $\cup_{\gamma<\Theta}\powerset(\gamma)$. For each $n < \omega$, let $\sigma_n = \{(p,\check{\alpha}) \ | \ p \in \mathbb{Q}, p \leq A_{\alpha,\langle\alpha_0,...,\alpha_n\rangle}$ for some $\langle\alpha_0,...,\alpha_n\rangle \in \Theta^{n+1}$$\}$; let $\tau = \{ (A, \sigma_n) \ | \ n<\omega, A \in \mathbb{Q}\}$. Then it's easy to see that whenever $G$ is $\mathbb{P}$-generic over $\textrm{HOD}$ induced by a $\mathbb{P}$-generic over $V$, $\tau_G$ enumerates $\cup_{\gamma<\Theta}\powerset(\gamma)$ in order type $\omega$. This means we can recover $\powerset(\mathbb{R})^V$ in the model $L[\mathbb{Q},\tau][G]$ by $\textsf{AD}^+$ (here we only use the fact that every set of reals has an $\infty$-Borel code which is a bounded subset of $\Theta$).
\\
\indent To sum up, we have $L[\mathbb{Q},\tau] \subseteq \textrm{HOD} \subseteq L[\mathbb{Q},\tau][G]$ for some $\mathbb{Q}$-generic $G$ over $\textrm{HOD}$. By a standard argument, this implies that $L[\mathbb{Q},\tau] = \textrm{HOD}$.
\end{proof}
We summarize some definitions and facts about hod mice that will be used in our computation. For basic definitions and notations that we omit, see \cite{ATHM}. The formal definition of a hod premouse $\P$\index{hod premouse} is given in Definition 2.12 of \cite{ATHM}. Let us mention some basic first-order properties of $\P$. There are an ordinal $\lambda^\P$ and sequences $\langle(\P(\alpha),\Sigma^\P_\alpha) \ | \ \alpha < \lambda^\P\rangle$ and $\langle \delta^\P_\alpha \ | \ \alpha \leq \lambda^\P  \rangle$ such that 
\begin{enumerate}
\item $\langle \delta^\P_\alpha \ | \ \alpha \leq \lambda^\P  \rangle$ is increasing and continuous and if $\alpha$ is a successor ordinal then $\P \vDash \delta^\P_\alpha$ is Woodin;
\item $\P(0) = Lp_\omega(\P|\delta_0)^\P$; for $\alpha < \lambda^\P$, $\P(\alpha+1) = (Lp_\omega^{\Sigma^\P_\alpha}(\P|\delta_\alpha))^\P$; for limit $\alpha\leq \lambda^\P$, $\P(\alpha) = (Lp_\omega^{\oplus_{\beta<\alpha}\Sigma^\P_\beta}(\P|\delta_\alpha))^\P$;
\item $\P \vDash \Sigma^\P_\alpha$ is a $(\omega,o(\P),o(\P))$\footnote{This just means $\Sigma^\P_\alpha$ acts on all stacks of $\omega$-maximal, normal trees in $\P$.}-strategy for $\P(\alpha)$ with hull condensation;
\item if $\alpha < \beta < \lambda^\P$ then $\Sigma^\P_\beta$ extends $\Sigma^\P_\alpha$.
\end{enumerate}
We will write $\delta^\P$ for $\delta^\P_{\lambda^\P}$ and $\Sigma^\P=\oplus_{\beta<\lambda^\P}\Sigma^\P_{\beta+1}$.

\begin{definition}\label{hod pair}
$(\P, \Sigma)$ is a hod pair\index{hod pair} if $\P$ is a countable hod premouse and $\Sigma$ is a $(\omega,\omega_1,\omega_1)$ iteration strategy for $\P$ with hull condensation such that $\Sigma^\P\subseteq\Sigma$ and this fact is preserved by $\Sigma$-iterations.
\end{definition}
Hod pairs typically arise in $\textsf{AD}^+$-models, where $\omega_1$-iterability implies $\omega_1+1$-iterability. In practice, we work with hod pairs $(\P,\Sigma)$ such that $\Sigma$ also has branch condensation. 
\begin{theorem}[Sargsyan]
\label{branch condensation's consequences}
Suppose $(\P, \Sigma)$ is a hod pair such that $\Sigma$ has branch condensation. Then $\Sigma$ is pullback consistent, positional and commuting.
\end{theorem}
The proof of \rthm{branch condensation's consequences} can be found in \cite{ATHM}. Such hod pairs are particularly important for our computation as they are points in the direct limit system giving rise to \textrm{HOD}. For hod pairs $(\M_\Sigma, \Sigma)$, if $\Sigma$ is a strategy with branch condensation and $\VT$ is a stack on $\M_\Sigma$ with last model $\N$, $\Sigma_{\N, \VT}$ is independent of $\VT$. Therefore, later on we will omit the subscript $\VT$ from $\Sigma_{N, \VT}$ whenever $\Sigma$ is a strategy with branch condensation and $\M_\Sigma$ is a hod mouse.

\begin{definition} Suppose $\P$ and $\Q$ are two hod premice. Then $\P\trianglelefteq_{hod}\Q$\index{$\trianglelefteq_{hod}$} if there is $\a\leq\l^\Q$ such that $\P=\Q(\a)$.
\end{definition}

If $\P$ and $\Q$ are hod premice such that $\P\trianglelefteq_{hod}\Q$ then we say $\P$ is a hod initial segment of $\Q$\index{hod initial segment}. If $(\P, \Sigma)$ is a hod pair, and $\Q\trianglelefteq_{hod} \P$, say $\Q = \P(\alpha)$, then we let $\Sigma_\Q$ be the strategy of $\Q$ given by $\Sigma$. Note that $\Sigma_\Q\cap \P = \Sigma^\P_\alpha \in \P$. \\
\indent All hod pairs $(\P, \Sigma)$ have the property that $\Sigma$ has hull condensation and therefore, mice relative to $\Sigma$ make sense. To state the Strong Mouse Capturing we need to introduce the notion of $\Gamma$-fullness preservation. We fix some reasonable coding (we call Code) of $(\omega,\omega_1,\omega_1)$-strategies by sets of reals. Suppose $(\P,\Sigma)$ is a hod pair. Let $I(\P,\Sigma)$\index{$I(\P,\Sigma)$} be the set $(\Q,\Sigma_\Q,\vec{\mathcal{T}})$ such that $\vec{\mathcal{T}}$ is according to $\Sigma$ such that $i^{\vec{\mathcal{T}}}$ exists and $\Q$ is the end model of $\vec{\mathcal{T}}$ and $\Sigma_\Q$ is the $\vec{\mathcal{T}}$-tail of $\Sigma$. Let $B(\P,\Sigma)$\index{$B(\P,\Sigma)$} be the set $(\Q,\Sigma_\Q,\vec{\mathcal{T}})$ such that there is some $\R$ such that $\Q = \R(\alpha)$, $\Sigma_\Q=\Sigma_{\R(\alpha)}$ for some $\alpha<\lambda^\R$ and $(\R,\Sigma_\R,\vec{\mathcal{T}})\in I(\P,\Sigma)$.
\begin{definition} 
\label{Gammafpr}
Suppose $\Sigma$ is an iteration strategy with hull-condensation, $a$ is a countable transitive set such that $\M_\Sigma\in a$\footnote{$\M_\Sigma$ is the structure that $\Sigma$-iterates.} and $\Gamma$ is a pointclass closed under boolean operations and continuous images and preimages. Then $Lp^{\Gamma, \Sigma}_{\omega_1}(a)=\cup_{\a<\omega_1} Lp^{\Gamma, \Sigma}_\a(a)$ where
\begin{enumerate}
\item $Lp^{\Gamma, \Sigma}_0(a)=a\cup \{a\}$
\item $Lp^{\Gamma,\Sigma}_{\a+1}(a)=\cup\{ \M :\M$ is a sound $\Sigma$-mouse\index{$\Sigma$-mouse} over $Lp^{\Gamma, \Sigma}_{\a}(a)$\footnote{By this we mean $\M$ has a unique $(\omega,\omega_1+1)$-iteration strategy $\Lambda$ above $Lp^{\Gamma, \Sigma}_{\a}(a)$ such that whenever $\N$ is a $\Lambda$-iterate of $\M$, then $\N$ is a $\Sigma$-premouse.} projecting to $Lp^{\Gamma, \Sigma}_{\a}(a)$ and having an iteration strategy in $\Gamma\}$.
\item $Lp^{\Gamma,\Sigma}_\l(a)=\cup_{\a<\l}Lp^{\Gamma,\Sigma}_\a(a)$ for limit $\lambda$.
\end{enumerate}
We let $Lp^{\Gamma,\Sigma}(a)=Lp^{\Gamma,\Sigma}_1(a)$\index{$Lp^{\Gamma,\Sigma}(a)$}.
\end{definition}
\begin{definition}[$\Gamma$-Fullness preservation]\label{gammafp}
\index{$\Gamma$-fullness preservation}Suppose $(\P, \Sigma)$ is a hod pair and $\Gamma$ is a pointclass closed under boolean operations and continuous images and preimages. Then $\Sigma$ is a $\Gamma$-fullness preserving\index{fullness preservation} if whenever $(\VT, \Q)\in I(\P, \Sigma)$, $\a+1\leq \l^\Q$ and $\eta>\d_\a$ is a strong cutpoint of $\Q(\a+1)$, then
\begin{center}
$\Q| (\eta^+)^{\Q(\a+1)}=Lp^{\Gamma, \Sigma_{\Q(\a), \VT}}(\Q| \eta)$.
\end{center}
and
\begin{center}
$\Q|(\d_\a^+)^\Q=Lp^{\Gamma, \oplus_{\b<\a}\Sigma_{\Q(\b+1), \VT}}(\Q|\delta^\Q_\alpha)$.
\end{center}
\end{definition}
When $\Gamma = \powerset(\mathbb{R})$, we simply say fullness preservation; in this case, we also write $Lp$ ($Lp^{\Sigma}$) instead of $Lp^\Gamma$ ($Lp^{\Gamma,\Sigma}$). A stronger notion of $\Gamma$-fullness preservation is super $\Gamma$-fullness preservation. Similarly, when $\Gamma = \powerset(\mathbb{R})$, we simply say super fullness preservation.
\begin{definition}[Super $\Gamma$-fullness preserving]
\label{SuperGammafpr}\index{super $\gamma$-fullness preserving}
Suppose $(\P,\Sigma)$ is a hod pair and $\Gamma$ is a pointclass closed under boolean operations and continuous images and preimages. $\Sigma$ is super $\Gamma$-fullness preserving if it is $\Gamma$-fullness preserving and whenever $(\mathcal{\vec{T}},\Q)\in I(\P,\Sigma)$, $\alpha < \lambda^\Q$ and $x\in HC$ is generic over $\Q$, then
\begin{center}
$Lp^{\Gamma,\Sigma_{\Q(\alpha)}}(x) = \{\M \ | \ \Q[x]\vDash ``\M \textrm{ is a sound } \Sigma_{\Q(\alpha)}\textrm{-mouse over } x \textrm{ and } \rho_{\omega}(\M) = x"\}$.
\end{center}
Moreover, for such an $\M$ as above, letting $\Lambda$ be the unique strategy for $\M$, then for any cardinal $\kappa$ of $\Q[x]$, $\Lambda\rest H_{\kappa}^{\Q[x]} \in \Q[x]$.
\end{definition}
Hod mice that go into the direct limit system that gives rise to HOD have strategies that are super fullness preserving. Here is the statement of the strong mouse capturing.
\begin{definition}[The Strong Mouse Capturing]\label{smc}
\index{Strong Mouse Capturing, SMC}The Strong Mouse Capturing (S\textsf{MC}) is the statement: Suppose $(\P, \Sigma)$ is a hod pair such that $\Sigma$ has branch condensation and is $\Gamma$-fullness preserving for some $\Gamma$. Then for any $x, y \in \mathbb{R}$, $x \in OD_\Sigma(y)$ iff $x$ is in some $\Sigma$-mouse over $\la \P, y\ra$.
\end{definition}

When $(\P, \Sigma)=\emptyset$ in the statement of \rdef{smc} we get the ordinary Mouse Capturing (\textsf{MC}). The Strong Mouse Set Conjecture (\textsf{SMSC}) just conjectures that \textsf{SMC} holds below a superstrong.
\begin{definition}[Strong Mouse Set Conjecture]\label{smsc} \index{Strong Mouse Set Conjecture, SMSC}Assume $\textsf{AD}^+$ and that there is no mouse with a superstrong cardinal. Then $S\textsf{MC}$ holds.
\end{definition}
Recall that by results of \cite{ATHM}, \textsf{SMSC} holds assuming (\textasteriskcentered). To prove that hod pairs exist in $\textsf{AD}^+$ models, we typically do a hod pair construction. For the details of this construction, see Definitions 2.1.8 and 2.2.5 in \cite{ATHM}. 
We recall the $\Gamma$-hod pair construction from \cite{ATHM} which is crucial for our \textrm{HOD} analysis. Suppose $\Gamma$ is a pointclass closed under complements and under continuous preimages. Suppose also that $\lambda^\P$ is limit. We let
\begin{center}
\index{$\Gamma(\P,\Sigma)$} $\Gamma(\P,\Sigma) = \{A \ | \ \exists (\Q,\Sigma_\Q,\vec{\mathcal{T}})\in B(\P,\Sigma) \ A <_w\footnote{Wadge reducible to}\index{$<_w$} Code(\Sigma_\Q)\}$.
\end{center}
\begin{equation*}
HP^\Gamma\index{HP$^\Gamma$} = \{(\P,\Lambda) \ | \ (\P,\Lambda) \textrm{ is a hod pair and } Code(\Lambda) \in \Gamma\},
\end{equation*}
and
\begin{eqnarray*}
Mice^\Gamma\index{Mice$^\Gamma$} = \{(a,\Lambda,\M) &&\ | \ a\in HC, \ a \textrm{ is self-wellordered transitive, } \Lambda \textrm{ is an iteration } \\ && \textrm{strategy } \textrm{ such that } (\M_\Lambda,\Lambda) \in HP^\Gamma, \ \M_\Lambda \in a, \ \textrm{and } \M \trianglelefteq Lp^{\Gamma,\Lambda}(a)\}.
\end{eqnarray*}
If $\Gamma = \powerset(\mathbb{R})$, we let $HP = HP^\Gamma$ and $Mice = Mice^\Gamma$. Suppose $(\M_\Sigma,\Sigma) \in HP^\Gamma$. Let
\begin{equation*}
Mice^\Gamma_\Sigma = \{ (a,\M) \ | \ (a,\Sigma,\M) \in Mice^\Gamma\}. 
\end{equation*}
\begin{definition}[$\Gamma$-hod pair construction]
\label{gamma hod pair construction}
\index{$\Gamma$-hod pair construction}Let $\Gamma$ be an inductive-like pointclass and $A_\Gamma$ be a universal $\Gamma$-set. Suppose $(M,\delta,\Sigma)$ is such that $M\vDash \textsf{\textsf{ZF}C}$ - Replacement, $(M,\delta)$ is countable, $\delta$ is an uncountable cardinal in $M$, $\Sigma$ is an $(\omega_1,\omega_1)$-iteration strategy for $M$, $\Sigma \cap (L_1(V_\delta))^M \in M$. Suppose $M$ locally Suslin captures $A_\Gamma$. Then the $\Gamma$-hod pair construction of $M$ below $\delta$ is a sequence $\langle\langle\N_\xi^\beta \ | \ \xi < \delta\rangle, \P_\beta,\Sigma_\beta,\delta_\beta \ | \ \beta \leq \Omega\rangle$ that satisfies the following properties.
\begin{enumerate}
\item $M \Vdash_{Col(\omega,<\delta)}$ ``for all $\beta<\Omega, (\P_\beta,\Sigma_\beta)$ is a hod pair such that $\Sigma_\beta \in \Gamma$"\footnote{This means there is a strategy $\Psi$ for $\P_\beta$ extending $\Sigma_\beta$ such that $Code(\Psi)\in \Gamma$ and $\Psi$ is locally Suslin captured by $M$ (at $\delta$).};
\item $\langle\N_\xi^0 \ | \ \xi < \delta\rangle$ are the models of the $L[\vec{E}]$-construction of $V_\delta^M$ and $\langle\N_\xi^{\beta} \ | \ \xi < \delta\rangle$ are the models of the $L[\vec{E},\Sigma_\beta]$-construction of $V_\delta^M$. $\delta_0$ is the least $\gamma$ such that $o(\N_\gamma^0) = \gamma$ and $Lp^\Gamma(\N^0_\gamma) \vDash``\gamma$ is Woodin" and $\delta_{\beta+1}$ is the least $\gamma$ such that $o(\N_\gamma^{\beta+1}) = \gamma$ and $Lp^{\Gamma,\Sigma_\beta}(\N^{\beta+1}_\gamma) \vDash``\gamma$ is Woodin".
\item $\P_0 = Lp_{\omega}^\Gamma(\N_{\delta_0}^0)$ and $\Sigma_0$ is the canonical strategy of $\P_0$ induced by $\Sigma$.
\item Suppose $\delta_{\beta+1}$ exists, $\N^{\beta+1}_{\delta_{\beta+1}}$ doesn't project across $\delta_\beta$. Furthermore, if $\beta = 0$ or is successor and $\N^{\beta+1}_{\delta_{\beta+1}} \vDash ``\delta_\beta$ is Woodin" and if $\beta$ is limit then $(\delta_\beta^+)^{\P_\beta} = (\delta^+_\beta)^{\N_{\delta_{\beta+1}}^{\beta+1}}$,
then $\P_{\beta+1} = Lp_{\omega}^{\Gamma,\Sigma_\beta}(\N_{\delta_{\beta+1}}^{\beta+1})$ and $\Sigma_{\beta+1}$ is the canonical strategy $\P_{\beta+1}$ induced by $\Sigma$.
\item For limit ordinals $\beta$, letting $\P_\beta^* = \cup_{\gamma<\beta}\P_\gamma, \ \Sigma_\beta^* = \varoplus_{\gamma<\beta}\Sigma_\gamma$, and $\delta_\beta$ = sup$_{\gamma<\beta}\delta_\gamma$, if $\delta_\beta < \delta$ then let $\langle\N^{*,\beta}_\xi \ | \ \xi < \delta\rangle$ be the models of the $L[\vec{E},\Sigma_\beta^*]$-construction of $V_\delta^M.$ If there isn't any $\gamma$ such that $o(\N^{*,\beta}_\gamma) = \gamma$ and $Lp^{\Gamma,\Sigma_\beta^*}(\N^{*,\beta}_\gamma) \vDash ``\gamma$ is Woodin" then we let $\P_\beta$ be undefined. Otherwise, let $\gamma$ be the least such that $o(\N^{*,\beta}_\gamma) = \gamma$ and  $Lp^{\Gamma,\Sigma_\beta^*}(\N^{*,\beta}_\gamma) \vDash ``\gamma$ is Woodin." If $\N^{*,\beta}_\gamma$ doesn't project across $\delta_\beta$ then $\P_\beta = \N^{*,\beta}_\gamma|(\delta_\beta^{+\omega})^{\N^{*,\beta}_\gamma}$, and $\Sigma_\beta$ is the canonical iteration strategy for $\P_\beta$ induced by $\Sigma$. Otherwise, let $\P_\beta$ be undefined.
\end{enumerate}
\end{definition}
\subsection{A definition of $K^\Sigma(\mathbb{R})$}
\begin{definition}
\label{model}\index{model operator}
Let $\mathcal{L}_0$ be the language of set theory expanded by unary predicate
symbols $\dot{E}, \dot{B}, \dot{S}$, and constant symbols $\dot{l}$ and
$\dot{a}$. Let $a$ be a given transitive set. A \textbf{model with paramemter a}
is an $\mathcal{L}_0$-structure of the form
\begin{center}
$\mathcal{M} = (M; \in, E, B, \mathcal{S}, l, a)$
\end{center}
such that $M$ is a transtive rud-closed set containing $a$, the structure
$\mathcal{M}$ is amenable, $\dot{a}^\M = a$, $\mathcal{S}$ is a sequence of
models with paramemter $a$ such that letting $S_\xi$ be the universe of
$\mathcal{S}_\xi$
\begin{itemize}
\item $\dot{S}^{\mathcal{S}_\xi} = \S\rest \xi$ for all $\xi\in
\textrm{dom}(\S)$ and $\dot{S}^{\S_\xi}\in S_\xi$ if $\xi$ is a successor
ordinal;
\item $S_\xi = \cup_{\alpha<\xi}S_\alpha$ for all limit $\xi \in
\textrm{dom}(\S)$;
\item if $\textrm{dom}(\S)$ is a limit ordinal then $M = \cup_{\alpha\in
\textrm{dom}(\S)} S_\alpha$ and $l=0$, and
\item if $\textrm{dom}(\S)$ is a successor ordinal, then $\textrm{dom}(\S) =
l$.
\end{itemize}
\end{definition}
The above definition is due to Steel and comes from \cite{wilson2012contributions}. Typically, the predicate $\dot{E}$ codes the top extender of the model;
$\dot{S}$ records the sequence of models being built so far. Next, we write down
some notations regarding the above definition.
\begin{definition}
\label{someNotations}
Let $\M$ be the model with parameter $a$. Then $|\M|$ denotes the universe of
$\M$. We let $l(\M) = dom(\dot{S}^\M)$\index{$l(\M)$} denote the \textbf{length of} $\M$ and
set $\M|\xi = 	\dot{S}^\M_\xi$ for all $\xi <l(\M)$. We set $\M|l(\M) = \M$. We
also let $\rho(\M)\leq l(\M)$ be the least such that there is some $A\subseteq
M$ definable (from parameters in $M$) over $\M$ such that $A\cap
|\M|\rho(\M)|\notin M$.
\end{definition}
\indent Suppose $J$ is a mouse operator that condenses well and relivizes well (in the sense of
\cite{CMI}). The definition of $\M_1^{J,\sharp}$ (more generally,
the definition of a $J$-premouse over a self-wellorderable set) has been given
in \cite{CMI} and \cite{wilson2012contributions}. Here we only re-stratify its
levels so as to suit our purposes.
\begin{definition}
\label{modelOpSoa}\index{$F_J(\M)$}
Let $\M$ be a model with parameter $a$, where $a$ is self-wellorderable. Suppose
$J$ is an iteration strategy for a mouse $\P$ coded in $a$. Let $A$ be a set of
ordinals coding the cofinal branch of $\mathcal{T}$ according to $J$, where $\T$
is the least (in the canonical well-ordering of $\M$) such that $J(\T)\notin
|\M|$ if such a tree exists; otherwise, let $A=\emptyset$. In the case $A\neq
\emptyset$, let $A^* = \{o(\M)+\alpha \ | \ \alpha\in A\}$ and $\xi$ be 
\begin{enumerate}
\item the least such that $\mathcal{J}_\xi(\M)[A^*]$ is a $\Q$-structure of
$\M|\rho(\M)$ if such a $\xi$ exists; or, 
\item $\xi$ is the least such that $\mathcal{J}_\xi(\M)[A^*]$ defines a set not
amenable to $\M|\rho(\M)$ if such a $\xi$ exists; or else,
\item $\xi = \textrm{sup}(A^*)$.
\end{enumerate}
For $\alpha\leq\xi$, we define $\M_\alpha$. For $\alpha=0$, let $\M_0 = \M$. For
$0 < \alpha < \xi$, suppose $\M_\alpha$ has been defined, we let 
\begin{center}
$\M_{\alpha+1} = (|\mathcal{J}(\M_\alpha)[A^*]|; \in, \emptyset,
A^*\cap|\mathcal{J}(\M_\alpha)[A^*]|,\dot{S}^\smallfrown \M_\alpha, l(\M_\alpha) +
1, a)$.
\end{center}
For limit $\alpha$, let $\M_\alpha = \cup_{\beta<\alpha}\M_\beta$. We then let 
$F_J(\M) = \M_\xi$. In the case $A=\emptyset$, we let
\begin{center}
$F_J(\M) = (|\mathcal{J}(\M)|;\in,\emptyset,\emptyset,\dot{S}^\smallfrown
\M,l(\M)+1,a)$.
\end{center}
In the case $J$ is a (hybrid) first-order mouse operator\footnote{This means
there is a (hybrid) mouse operator $J'$ that condenses well
such that there is a formula $\psi$ in the language
of $J'$-premice and some parameter $a$ such that for every $x\in
\textrm{dom}(J)$, $J(x)$ is the least $\M\lhd Lp^{J'}(x)$ that satisfies
$\psi[x,a]$.}, we let $J^*(\M)$ be the least level of $J(\M)$ that is a
$\Q$-structure or defines a set not amenable to $\M|\rho(\M)$ if it exists;
otherwise, $J^*(\M) = J(\M)$. We then define $F_J(\M)$ as follows. Let $\M_0 =
\M$. Suppose for $\alpha$ such that $\omega\alpha < o(J^*(\M))$, we've defined
$\M||\alpha$ and maintained that $|\M||\alpha| = |J^*(\M)||\alpha|$, let
$\M_{\alpha+1} = (|J^*(\M)||(\alpha+1)|;\in, \emptyset, \emptyset,
\dot{S}^\smallfrown \M_\alpha, l(\M_\alpha)+1,a)$, where $\dot{S} =
\dot{S}^{\M_\alpha}$. If $\alpha$ is limit and $J^*(\M)||\alpha$ is passive, let
$\M_\alpha = \cup_{\beta<\alpha} \M_\beta$; otherwise, let $\M_\alpha =
(\cup_{\beta<\alpha} |\M_\beta|;\in, E, \emptyset,
\cup_{\beta<\alpha}\dot{S}^{\M_\beta}, \textrm{sup}_{\beta<\alpha} l(\M_\beta),
a)$, where $E$ is $F_\alpha^{J^*(\M)}$. Finally,
 \begin{center}
$F_J(\M) = M_{\gamma}$, where $\omega\gamma = o(J^*(\M))$.
\end{center}
\end{definition}

The rest of the definition of a $J$-premouse over a self-wellorderable set $a$
is as in \cite{wilson2012contributions}. We now wish to extend this definition
to non self-wellorderable sets $a$, and in particular to $\mathbb{R}$. For this, we
need to assume that the following absoluteness property holds of the operator
$J$. As shown in \cite{trang2012scales}, if $J$ is a mouse strategy operator for a nice enough
strategy, then it does hold.

\begin{definition}
 We say $J$ \textbf{determines itself on generic extensions (relative to
$\N=\M_1^{J,\sharp}$)}\index{determines itself on generic extensions} iff there are formulas $\varphi,\psi$ in the language
of $J$-premice such that for any correct, non-dropping iterate $\P$ of
$\N$, via a countable iteration tree, any $\P$-cardinal $\delta$, any
$\gamma\in$ OR such that $\P|\gamma\models\varphi$\textup{+}``$\delta$ is
Woodin'', and any $g$ which is set-generic over $\P|\gamma$, then
$(\P|\gamma)[g]$
is closed under $J$ and $J\rest\P[g]$ is defined over $(\P|\gamma)[g]$ by
$\psi$. We say such a pair $(\varphi,\psi)$ \emph{generically determines $J$}.
\end{definition}
The model operators that we encounter in the core model induction condense well, relativize well, and determine themselves on generic extensions. 
\begin{definition}
We say a (hod) premouse $\M$ is \textbf{reasonable} iff under $\textsf{ZF}+\textsf{AD}$, $\M$
satisfies the
first-order properties which are consequences of
$(\omega,\omega_1,\omega_1)$-iterability, or under $\textsf{ZFC}$, $\M$ satisfies the first-order
properties which are consequences of $(\omega,\omega_1,\omega_1+1)$-iterability.
\end{definition}


The following lemma comes from \cite{trang2012scales}.
\begin{lemma}
\label{GenericInt}
Let $(\P,\Sigma)$ be such that either (a)
$\P$ is a reasonable premouse and $\Sigma$ is the unique normal OR-iteration
strategy for $\P$; or (b) $\P$ is a reasonable hod premouse, $(\P,\Sigma)$ is a
hod pair which is fullness preserving and has branch condensation. Assume that $\M_1^\Sigma$ exists
and is fully iterable. Then $\Sigma$ determines itself on generic
extensions.\end{lemma}

Let $M$ be a transitive model of
some fragment of set theory. Let $\dot{G}$ be the canonical $Col(\omega,M)$-name
for the generic $G\subseteq Col(\omega,M)$ and $\dot{x}_{\dot{G}}$ be the
canonical name for the real coding $\{(n,m) \ | \ G(n) \in G(m)\}$, where we
identify $G$ with the surjective function from $\omega$ onto $M$ that $G$
produces. Let $\Lambda$ be the strategy for $\N=\M_1^{J,\sharp}$. Using the
terminology of \cite{ATHM}, we say a tree $\mathcal{T}$ on
$\N$ via $\Lambda$ is the tree for making \textit{$M$ generically generic}
if the following holds:
\begin{enumerate}
\item $\mathcal{T}\rest (o(M)+1)$ is a linear iteration tree obtained by
iterating the first total measure of $\M$ and its images $o(M) + 1$ times.
\item For $\alpha \geq o(M) + 1$, $E^\mathcal{T}_\alpha$ is the extender with
least index in $\M^\mathcal{T}_\alpha$ such that there is a condition $p\in
Col(\omega,M)$ such that $p \Vdash \dot{x}_{\dot{G}}$ does not satisfy an axiom
involving $E^\mathcal{T}_\alpha$ from the extender algebra $\mathbb{B}_\delta$,
where $\delta$ is the Woodin cardinal of $\mathcal{M}^\mathcal{T}_\alpha$.
\end{enumerate}
We denote such a tree $\T_{M}$\index{$\T_{M}$}. Note that $\T_M\in V$, $\T$ is nowhere
dropping, and $lh(\mathcal{T}_{M}) < |M|^+$. Also note that $\T_M$ does
not include the last branch. Given a formula $\varphi$, let
$\T_M^\varphi=\T_M\rest\lambda$\index{$\T_M^\varphi$}, where $\lambda$ is least such that either
$\lambda=\lh(\T_M)$ or $\lambda$ is a limit ordinal and there is
$\P\trianglelefteq Q(\T_M\rest\lambda)$ such that
$M(\T_M\rest\lambda)\trianglelefteq\P$ and $\P\models\varphi$. Now suppose there is
$\P\lhd\N$ such that $\N|\delta^\N\trianglelefteq\P$ and $\P\models\varphi$. Let
$\lambda\leq\lh(\T_M^\varphi)$ be a limit. If
$\lambda<\lh(\T_M^\varphi)$ let
$Q^\varphi(\T_M\rest\lambda)=Q(M(\T_M\rest\lambda))$. Otherwise let
$Q^\varphi(\T_M\rest\lambda)=\P$, where $\P$ is least such that
$M(\T_M\rest\lambda)\trianglelefteq\P\trianglelefteq M^{\T_M}_{\Lambda(\T_M\rest\lambda)}$ and
$\P\models\varphi$. We should mention that in order for the definition of $\mathcal{T}_M$ to make sense, $\Lambda$ and $\Sigma$ need to be $(\omega,|M|^++1)$-iterable.

We're ready to define $J$-premice over an arbitrary transitive set $a$. The idea that to define a $\Sigma$-premouse (over an arbitrary set), it suffices to tell the model branches of trees that make certain levels of the model generically generic comes from \cite{ATHM}, where it's used to reorganize hod mice in such a way that $S$-constructions work.

\begin{definition}
\label{modelOpGenerala2}
Suppose $a$ is a transitive set coding $\M_1^{J,\sharp}$. Suppose $(\varphi,\psi)$ generically determines $J$. Let $\Lambda$ be the
strategy for $\M_1^{J,\sharp}$. We define $F_J^*(a)$ \index{$F_J^*(a)$}to be a level of a model $\M$ with parameter $a$ with the following properties. There is $\alpha < l(\M)$ such that $\M|\alpha\vDash \textsf{ZF}$. Let $\alpha$ be the least such and let $\xi$ be the largest cardinal of $\M|\alpha = \mathcal{J}_\alpha(a)$. Let $\lambda\leq\lh(\T_{\M|\alpha}^\varphi)$ be a
limit. Let
\begin{center}
$\P_{\alpha,\lambda}=Q^\varphi(\T_{\M|\alpha}\rest\lambda)$.
\end{center}
Let $B\subseteq o(\P_{\alpha,\lambda})$ be the standard set
coding $\P_{\alpha,\lambda}$. Let $\omega\gamma=o(\P_{\alpha,\lambda})$. Let for $\beta < l(\M)$,
\begin{center}
$A_\beta=\{o(\M|\beta)+\eta \ | \ \eta\in B\}\times\{(\alpha,\lambda)\}$. 
\end{center}
and define 
\begin{center}
$F_{J,\alpha,\lambda}(\M|\beta)=\mathcal{J}_\gamma^{A_\beta}(\M|\beta)$ 
\end{center}
if no levels of $\mathcal{J}_\gamma^A(\M|\beta)$ is a $\Q$-structure for $(\M|\beta)|\rho(\M|\beta)$ or projects across $\rho(\M|\beta)$; otherwise, let $F_{J,\alpha,\lambda}(\M|\beta)=\mathcal{J}(\M|\beta)$.\footnote{Technically, $F_{J,\alpha,\lambda}(\M|\beta)$ is stratified as a model over $a$ but we suppress the structure for brevity. See Definition \ref{modelOpSoa} for the stratification.}. 

Suppose $\M|\beta$ has been defined and there is a $\lambda$ such that $\P_{\alpha,\lambda}$ is defined, $\T^\varphi_{\M|\alpha}\rest\lambda\in \M|\beta$, but for no $\beta' < l(\M|\beta)$, $F_{J,\alpha,\lambda}(\M|\beta') \neq \mathcal{J}(\M|\beta')$, we let then $\M|\xi^* = F_{J,\alpha,\lambda}(\M|\beta)$, where $\xi^* = l(F_{J,\alpha,\lambda}(\M|\beta))$ for the least such $\lambda$.

We say that $\T^\varphi_{\M|\alpha}|\lambda$ is \textbf{taken care of in $\M$} if there is a $\beta < l(\M)$ such that $F_{J,\alpha,\lambda}(\M|\beta)$\\$\lhd \M$ and  $F_{J,\alpha,\lambda}(\M|\beta)\neq \mathcal{J}(\M|\beta)$. So $\M$ is the least such that for every limit $\lambda \leq lh(\T^\varphi_{\M|\alpha})$, $\T^\varphi_{\M|\alpha}\rest \lambda$ is taken care of in $\M$.

Finally, let $F_J^*(a) = \M$ if no levels of $\M$ projects across $\xi$. Otherwise, let $F_J^*(a) = \M|\beta$, where $\beta$ is the least such that $\rho_\omega(\M|\beta) < \xi$.

\end{definition}

\begin{definition}[Potential $J$-premouse over $a$]
\label{PotJPremice2}\index{potential $J$-premouse}
Let $a$ be a transitive structure such that $a$ contains a real coding $\N$. We
say that $\M$ is a \textbf{potential $J$-premouse over $a$} iff $\M$ is a model
with parameter $a$, and there is an ordinal $\lambda$ and a increasing, closed 
sequence $\left<\eta_\alpha\right>_{\alpha\leq\lambda}$ of ordinals,
such that for each $\alpha\leq\lambda$, we have:
\begin{enumerate}[(a)]
\item if $a$ is not a self-wellordered set, then $\eta_0=1$ and $\M|1=a$; otherwise, either $\lambda=0$ and $\M = \M|\eta_0\trianglelefteq \M_1^{J,\sharp}$ or else $\M|\eta_0 = \M_1^{J,\sharp}$ (in the sense of Definition \ref{modelOpSoa}),
\item $\eta_\alpha\leq l(\M)$,
\item if $\alpha+1<\lambda$, then $\M|\eta_{\alpha+1}=F^*_J(\M|\eta_\alpha)$,
\item if $\alpha+1=\lambda$, then $\M\trianglelefteq F^*_J(\M|\eta_\alpha)$,\footnote{We will
also
use $\M_\eta$ to 
denote $\M|\eta$.}
\item $\eta_\lambda=l(\M)$,
\item if $\eta=\eta_\alpha$ and $\dot{E}^{\M|\eta}\neq\emptyset$ (and therefore
$\alpha$ is a limit) then $\dot{E}^{\M|\eta}$ codes an extender $E$ that coheres
$\M|\eta$ and satisfies
the obvious modifications of the premouse axioms (in the sense of Definition
2.2.1 of \cite{wilson2012contributions}) and $E$ is
$a\times \gamma$-complete for all $\gamma < \textrm{crt}(E)$\footnote{This means
whenever $\langle X_x \ | \ x\in a\times \gamma\rangle\in \M|\lambda$ is such
that $X_x\in E_b$ for each $x\in a\times \gamma$, where $b$ is a finite subset
of $lh(E)$, then $\cap_{x\in a}X_x \in E_b$}.
\end{enumerate}
\end{definition}

We define projecta, standard parameters, solidity, soundness, cores as in
section 2.2 of \cite{wilson2012contributions}.
\begin{definition}
\label{JPremouse}\index{$J$-premouse}
Suppose $\M$ is a potential $J$-premouse over $a$. Then we say that $\M$ is a
$J$-premouse over $a$ if for all $\lambda < l(\M)$, $\M|\lambda$ is
$\omega$-sound.
\end{definition}

\begin{definition}
\label{Active}
Suppose $\M$ is a $J$-premouse over $a$. We say that $\M$ is active if
$\dot{E}^\M \neq \emptyset$ or $\dot{B}^\M \neq \emptyset$. Otherwise, we say
that $\M$ is passive.
\end{definition}
\begin{definition}[$J$-mouse]
\label{JMouse}\index{$J$-mouse}
Let $\M, a$ be as in Definition \ref{JPremouse}. We say that $\N$ is a $J$-mouse
over $a$ if $\rho_{\omega}(\N) = a$ and whenever $\N^*$ is a countable
transitive $J$-premouse over some $a^*$ and there is an elementary embedding
$\pi: \N^*\rightarrow \N$ such that $\pi(a^*) = a$, then $\N^*$ is
$(\omega,\omega_1+1)$-iterable\footnote{Sometimes we need more than just $\omega_1+1$-iterability.} and whenever $\R$ is an iterate of $\N^*$ via its unique
iteration strategy, $\R$ is a $J$-premouse over $a^*$.
\end{definition}
Suppose $\M$ is a $J$-premouse over $a$. We say that $\M$ is \textit{$J$-complete} if $\M$ is closed under the operator $F_J^*$. The following lemma is also from \cite{trang2012scales}.
\begin{lemma}
\label{ClosedUnderSigma}
Suppose $\M$ is a $J$-premouse over $a$ and $\M$ is $J$-complete. Then $\M$ is closed under $J$;
furthermore, for any set generic extension $g$ of $\N$, $\N[g]$ is closed under
$J$ and in fact, $J$ is uniformly definable over $N[g]$ (i.e. there is a
$\mathcal{L}_0$-formula $\phi$ that defines $J$ over any generic extension of
$N$).
\end{lemma}

If $a$ in Definition \ref{JMouse} is $H_{\omega_1}$, then we define
$Lp^J(\mathbb{R})$ to be the union of all $J$-mice $\N$ over $a$\footnote{We'll
be also saying $J$-premouse over $\mathbb{R}$ when $a = H_{\omega_1}$}. In core
model induction applications, we typically have a pair $(\P,\Sigma)$ where $\P$
is either a hod premouse and $\Sigma$ is $\P$'s $(\omega,\omega_1,\mathfrak{c}^++1)$-iteration strategy with branch
condensation and is fullness preserving (relative to mice in some pointclass) or
$\P$ is a sound (hybrid) premouse projecting to some countable set $a$ and
$\Sigma$ is the unique (normal) strategy for $\P$. Lemma \ref{GenericInt} shows that $\Sigma$
condenses well and determines itself on generic extension in the sense defined
above\footnote{Technically, the statement of Lemma \ref{GenericInt} assumes full $(\omega,\textrm{OR})$-iterability but the proof of the lemma is local enough that this holds.}. We then define $Lp^\Sigma(\mathbb{R})$\footnote{In this paper, we use $Lp^\Sigma(\mathbb{R})$ and $K^\Sigma(\mathbb{R})$ interchangably.} as above. 

We mention a theorem of Sargsyan and Steel that will be important for our computation. See \cite{DMATM} for a proof of the case $\Theta=\theta_0$ of the theorem, where $\Sigma = \emptyset$ and $K^\Sigma(\mathbb{R}) = K(\mathbb{R})$.
\begin{theorem}[Sargsyan, Steel]\label{main theorem}
Assume $\textsf{AD}^++\textsf{SMC} + \Theta=\theta_{\alpha+1}$ or $\Theta=\theta_0$. Suppose $(\P, \Sigma)$ is a hod pair such that $\Sigma$ has branch condensation and is fullness preserving and $K^\Sigma(\mathbb{R})$ is defined. Suppose also that $\M_\infty(\P,\Sigma)|\theta_\alpha= HOD|\theta_\alpha$, where $\M_\infty(\P,\Sigma)$ is the direct limit of all $\Sigma$-iterates of $\Sigma$. Then
\begin{center}
$\{ A\subseteq \mathbb{R} : A\in OD_{\Sigma}(y) \textrm{ for some real y} \}=\powerset(\mathbb{R})\cap K^\Sigma(\mathbb{R})$.
\end{center}
\end{theorem}

\subsection{A Prikry forcing}
Let $(\P,\Sigma)$ be a hod pair such that $\Sigma$ has branch condensation and $K^\Sigma(\mathbb{R})$ is defined. We briefly describe a notion of Prikry forcing that will be useful in our HOD computation. The forcing $\mathbb{P}$ described here is defined in $K^\Sigma(\mathbb{R})$ and is a modification of the forcing defined in Section 6.6 of \cite{steel2012hod}. All facts about this forcing are proved similarly as those in Section 6.6 of \cite{steel2012hod} so we omit all proofs.

First, let $T$ be the tree of a $\Sigma^2_1(\Sigma)$ scale on a universal $\Sigma^2_1$ set $U$. Write $\P_x$ for the $\Sigma$-premouse coded by the real $x$. Let $a$ be countable transitive, $x \in \mathbb{R}$ such that $a$ is coded by a real recursive in $x$. A normal iteration tree $\U$ on a 0-suitable $\Sigma$-premouse $\Q$ (see Definition \ref{suitability}, where $(\Q,\Sigma)$ is defined to be $0$-suitable) is short if for all limit $\xi \leq lh(\U)$, $Lp^\Sigma(\M(\U|\xi)) \vDash \delta(\U|\xi)$ is not Woodin. Otherwise, we say that $\U$ is maximal. We say that a 0-suitable $\P_z$ is short-tree iterable by $\Lambda$ if for any short tree $\T$ on $\P_z$, $b = \Lambda(\T)$ is such that $\M^{\T}_b$ is 0-suitable, and $b$ has a $Q$-structure $\Q$ such that $\Q \trianglelefteq \M^{\T}_b$. Put
\begin{center}
$\F^x_a$ = $\left\{ \P_z \ | \ z \leq_T x, \P_z \textrm{ is a short-tree iterable $0$-suitable $\Sigma$-premouse over } a\right\}$
\end{center}
For each $a$, for $x$ in the cone in the previous claim, working in $L[T,x]$, we can simultaneously compare all $\P_z \in \F^x_a$ (using their short-tree iteration strategy) while doing the genericity iterations to make all $y$ such that $y \leq_T x$ generic over the common part of the final model $\Q^{x,-}_a$. This process (hence $\Q^{x,-}_a$) depends only on the Turing degree of $x$. Put
\begin{center}
$\Q^x_a = Lp^\Sigma_\omega(\Q^{x,-}_a)$, and $\delta^x_a = o(\Q^{x,-}_a)$.
\end{center}
By the above discussion, $\Q^x_a, \delta^x_a$ depend only on the Turing degree of $x$. Here are some properties obtained from the above process.
\begin{enumerate}
\item $\F^x_a \neq \emptyset$ for $x$ of sufficiently large degree;
\item $\Q^{x,-}_a$ is full (no levels of $\Q^x_a$ project strictly below $\delta^x_a$);
\item $\Q^x_a \vDash \delta^x_a$ is Woodin;
\item $\powerset{(a)} \cap \Q^x_a = \powerset{(a)} \cap OD_T(a \cup \{a\})$ and $\powerset{(\delta^x_a)} \cap \Q^x_a = \powerset{(\delta^x_a)} \cap OD_T(Q^{x,-}_a \cup \{Q^{x,-}_a\})$;
\item $\delta^x_a = \omega_1^{L[T,x]}$.
\end{enumerate}
Now for an increasing sequence $\vec{d} = \langle d_0, ..., d_n\rangle$ of Turing degrees, and $a$ countable transitive, set
\begin{center}
$\Q_0(a) = \Q^{d_0}_a$ and $\Q_{i+1}(a) = \Q^{d_{i+1}}_{\Q_i(a)}$ for $i < n$
\end{center}

We assume from here on that the degrees $d_{i+1}$'s are such that $\Q^{d_{i+1}}_{\Q_i(a)}$ are defined. For $\vec{d}$ as above, write $\Q^{\vec{d}}_i(a) = \Q_i(a)$ even though $\Q_i(a)$ only depends on $\vec{d}|(i+1)$. Let $\mu$ be the cone measure on the Turing degrees. We can then define our Prikry forcing $\mathbb{P}$ (over $L(T,\mathbb{R})$) as follows. A condition $(p,S) \in \mathbb{P}$ just in case $p = \langle\Q^{\vec{d}}_0(a),..., \Q^{\vec{d}}_n(a)\rangle$ for some $\vec{d}$, $S \in L(T,\mathbb{R})$ is a ``measure-one tree" consisting of stems $q$ which either are initial segments or end-extensions of $p$ and such that $(\forall q = \langle\Q^{\vec{e}}_0(a),..., \Q^{\vec{e}}_k(a)\rangle \in S)(\forall^*_\mu d)$ let $\vec{f} = \langle\vec{e}(0),...,\vec{e}(k), d\rangle$, we have $\langle\Q^{\vec{f}}_0(a),...,\Q^{\vec{f}}_{(k+1)}(a)\rangle \in S$. The ordering on $\mathbb{P}$ is defined as follows.
\begin{center}
$(p,S) \preccurlyeq (q, W)$ iff $p$ end-extends $q$, $S \subseteq W$, and $\forall n \in $ dom$(p)$$\backslash$ dom$(q)$ $(p|(n+1) \in W)$.
\end{center}
$\mathbb{P}$ has the Prikry property in $K^\Sigma(\mathbb{R})$. Let $G$ be a $\mathbb{P}$-generic over $K^\Sigma(\mathbb{R})$, $\langle \Q_i \ | \ i < \omega\rangle = \cup \{p \ | \ \exists \vec{X}(p,\vec{X}) \in G\}$ and $\Q_{\infty} = \bigcup_i \Q_i$. Let $\delta_i$ be the largest Woodin cardinal of $\Q_i$. Then 
\begin{center}
$P(\delta_i) \cap L[T, \langle \Q_i \ | \ i < \omega\rangle] \subseteq \Q_i$,
\end{center} 
and
\begin{center}
$L[T, \Q_{\infty}] = L[T, \langle\Q_i \ | \ i < \omega\rangle] \vDash \delta_i$ is Woodin.
\end{center}
\begin{definition}[Derived models]
\label{DerivedModel}
Suppose $M\vDash \textsf{ZFC}$ and $\lambda\in M$ is a limit of Woodin cardinals in $M$. Let $G\subseteq Col(\omega,<\lambda)$ be generic over $M$. Let $\mathbb{R}^*_G$ (or just $\mathbb{R}^*$) be the symmetric reals of $M[G]$ and $Hom^*_G$ (or just $Hom^*$\index{$Hom^*$}) be the set of $A\subseteq\mathbb{R}^*$ in $M(\mathbb{R}^*)$ such that there is a tree $T$ such that $A = p[T]\cap \mathbb{R}^*$ and there is some $\alpha < \lambda$ such that 
\begin{center}
$M[G\rest \alpha] \vDash ``T$ has a $<$-$\lambda$-complement".
\end{center} 
By the \textbf{old derived model}\index{old derived model} of $M$ at $\lambda$, denoted by $D(M,\lambda)$\index{$D(M,\lambda)$}, we mean the model $L(\mathbb{R}^*,Hom^*)$. By the \textbf{new derived model}\index{new derived model} of $M$ at $\lambda$, denoted by $D^+(M,\lambda)$\index{$D^+(M,\lambda)$}, we mean the model $L(\Gamma,\mathbb{R}^*)$, where $\Gamma$ is the closure under Wadge reducibility of the set of $A\in M(\mathbb{R}^*)\cap \powerset(\mathbb{R}^*)$ such that $L(A,\mathbb{R}^*) \vDash \textsf{AD}^+$.
\end{definition}
\begin{theorem}[Woodin]
\label{DMT}
Let $M$ be a model of $\textsf{ZFC}$ and $\lambda\in M$ be a limit of Woodin cardinals of $M$. Then $D(M,\lambda)\vDash \textsf{AD}^+$, $D^+(M,\lambda)\vDash \textsf{AD}^+$. Furthermore, $Hom^*$ is the pointclass of Suslin co-Suslin sets of $D^+(M,\lambda)$.
\end{theorem}
Using the proof of Theorem 3.1 from \cite{steel08} and the definition of $K^\Sigma(\mathbb{R})$ defined above, we get that in $K^\Sigma(\mathbb{R})[G]$, there is a $\Sigma$-premouse $\Q_\infty^+$ extending $\Q_\infty$ such that $K^\Sigma(\mathbb{R})$ can be realized as a (new) derived model of $\Q_\infty^+$ at $\omega_1^V$, which is the limit of Woodin cardinals of $\Q_\infty^+$. Roughly speaking, the $\Sigma$-premouse $\Q_\infty^+$ is the union of $\Sigma$-premice $\R$ over $\Q_\infty$, where $\R$ is an S-translation of some $\M\lhd K^\Sigma(\mathbb{R})$ (see \cite{ATHM} for more on S-translations).

\section{The $\Theta = \theta_0$ case} 
\subsection{Definitions and notations}

\begin{definition}($k$-suitable premouse) 
\label{suitability}
Let $0 \leq k < \omega$ and $\Gamma$ be an inductive-like pointclass. A premouse $\N$ is $k$-suitable with respect to $\Gamma$ iff there is a strictly increasing sequence $\langle \delta_i \ | \ i \leq k \rangle$ such that
\begin{enumerate}
\item for all $\delta$, $\N$ $\vDash``\delta$ is Woodin" iff $\delta = \delta_i$ for some $i < 1 + k$;
\item $\textrm{OR}^\N = sup(\{(\delta^{+n}_k)^N | n < \omega\})$;
\item $Lp^{\Gamma}(\N|\xi) \unlhd \N$ for all cutpoints $\xi$ of $\N$ where $Lp^{\Gamma}(\N|\xi) = \cup\{\M \ | \ \N|\xi\unlhd \M \wedge \rho(\M) = \xi \wedge \M \textrm{ has iteration strategy in } \Gamma\}$;
\item if $\xi \in \textrm{OR} \cap \N$ and $\xi \neq \delta_i$ for all i, then $Lp^\Gamma(\N|\xi) \vDash " \xi$ is not Woodin."
\end{enumerate}
\end{definition}
\begin{definition} Let $\N$ be as above and $A \subseteq \mathbb{R}$. Then $\tau^\N_{A,\nu}$ is the unique standard term $\sigma \in \N$ such that $\sigma^g = A \cap \N[g]$ for all g generic over $\N$ for $Col(\omega, \nu)$, if such a term exists. We say that $\N$ term captures $A$ iff $\tau^\N_{A,\nu}$ exists for all cardinals $\nu$ of $\N$.
\end{definition}
\indent If $\N, \Gamma$ are as in Definition 2.1 and $A \in \Gamma$, then \cite{CMI} shows that $\N$ term captures $A$. Later on, if the context is clear, we'll simply say capture instead of term capture or Suslin capture. For a complete definition of ``$\N$ is $A$-iterable", see \cite{Scalesweakgap}. Roughly speaking, $\N$ is $A$-iterable if $\N$ term captures $A$ and 
\begin{enumerate}
\item for any maximal tree $\mathcal{T}$ (or stack $\vec{\mathcal{T}}$) on $\N$, there is a cofinal branch $b$ such that the branch embedding $i^{\mathcal{T}}_b=_{def} i$ moves the term relation for $A$ correctly i.e., for any $\kappa$ cardinal in $\N$, $i(\tau^\N_{A,\kappa}) = \tau^{\M^{\mathcal{T}}_b}_{A,i(\kappa)}$;
\item if $\mathcal{T}$ on $\N$ is short, then there is a branch $b$ such that $\Q(b,\mathcal{T})$\footnote{$\Q(b,\mathcal{T})$ is called the $\Q$-structure and is defined to be the least initial segment of $\M^{\mathcal{T}}_b$ that defines the failure of Woodinness of $\delta(\mathcal{T})$.} exists and $\Q(b,\mathcal{T})\unlhd Lp^{\Gamma}(\M(\mathcal{T}))$\footnote{This implicitly assumes that $\Q(b,\mathcal{T})$ has no extenders overlapping $\delta(\mathcal{T})$. We're only interested in trees $\mathcal{T}$ arising from comparisons between suitable mice and for such trees, $\Q$ structures have no extenders overlapping $\delta(\mathcal{T})$.}; we say that $\mathcal{T}^\smallfrown b$ is $\Gamma$-guided.
\end{enumerate}
This obviously generalizes to define $\vec{A}$-iterability for any finite sequence $\vec{A}$.
\begin{definition} Let $\N$ be $k$-suitable with respect to $\Sigma^2_1$ and $k < \omega$. Let $\vec{A} = \langle A_i \ | \ i \leq n\rangle$ be a sequence of OD sets of reals and $\nu = (\delta_k^{+\omega})^{\N}$. Then  
\begin{enumerate}
\item $\gamma^\N_{\vec{A}}$ = $\sup(\{\xi | \xi$ is definable over $(\N|\nu, \tau^\N_{A_0,\delta_k}, ... , \tau^\N_{A_n,\delta_k})\}\cap \delta_0)$;
\item $H^\N_{\vec{A}}$ = $Hull^{\N}(\gamma^\N_{\vec{A}} \cup \{\tau^\N_{A_0,\delta_k}, ... , \tau^\N_{A_n,\delta_k}\})$, where we take the full elementary hull without collapsing.
\end{enumerate}
\end{definition}
From now on, we will write $\tau^{\N}_A$ without further clarifying that this stands for $\tau^{\N}_{A,\delta}$ where $\delta$ is the largest Woodin cardinal of $\N$. We'll also write $\tau^\N_{A,l}$ for $\tau^{\N}_{A,\delta^\N_l}$ for $l\leq k$. Also, we'll occasionally say $k$-suitable without specifying the pointclass $\Gamma$.
\begin{definition}
\label{strongAiterability}
Let $\N$ be $k$-suitable with respect to some pointclass $\Gamma$ and $A\in \Gamma$. $\N$ is strongly $A$-iterable if $\N$ is $A$-iterable and for any suitable $\M$ such that if $i,j: \N\rightarrow \M$ are two $A$-iteration maps then $i\rest H^\N_A = j\rest H^\M_A$. 
\end{definition}
\begin{definition}
\label{guiding}
Let $\Gamma$ be an inductive-like pointclass and $\N$ be $k$-suitable with respect to $\Gamma$ for some $k$. Let $\mathcal{A}$ be a countable collection of sets of reals in $\Gamma \cup \breve{\Gamma}$. We say $\mathcal{A}$ guides a strategy for $\N$ below $\delta_0^\N$ if whenever $\mathcal{T}$ is a countable, normal iteration tree on $\N$ based on $\delta_0^\N$ of limit length, then
\begin{enumerate}
\item if $\mathcal{T}$ is short, then there is a unique cofinal branch $b$ such that $\mathcal{Q}(b,\mathcal{T})$ exists and $\mathcal{Q}(b,\mathcal{T}) \unlhd Lp^{\Gamma}(\mathcal{M}(\mathcal{T}))$\footnote{Again we disregard the case where $\Q$-structures have overlapping extenders.}, and
\item if $\mathcal{T}$ is maximal, then there is a unique nondropping branch $b$ such that $i^{\mathcal{T}}_b(\tau^\N_{A,\mu}) = \tau^{\mathcal{M}^{\mathcal{T}}_b}_{A,i_b(\mu)}$ for all $A \in \mathcal{A}$ and cardinals $\mu \geq \delta_k^\N$ of $\N$ and $\delta(\T) = \sup\{\gamma^{\M^{\mathcal{T}}_b}_{A,0} \ | \ A \in \mathcal{A}\}$ where $\delta = i^{\mathcal{T}}_b(\delta_0)$.
\end{enumerate}
We can also define an $\mathcal{A}$-guided strategy that acts on finite stacks of normal trees in a similar fashion.
\end{definition}
\indent The most important instance of the above definition used in this paper is when $\mathcal{A}$ is a self-justifying-system that seals a $\Sigma_1$ gap. A strategy guided by such an $\mathcal{A}$ has many desirable properties.
\subsection{The computation}
\noindent Now let $\mathcal{F} = \{(\M,\vec{A}) \ | \ \vec{A}$ is a finite sequence of OD sets of reals and $\M$ is $k$-suitable for some $k$ and is strongly $\vec{A}$-iterable$\}$. We say $(\M,\vec{A}) \leq_{\mathcal{F}} (\N,\vec{B})$ if $\vec{A}\subseteq \vec{B}$ and $\M$ iterates to a suitable initial segment of $\N$, say $\N^-$, via its iteration strategy that respects $\vec{A}$. We then let $\pi_{(\M,\vec{A}),(\N,\vec{B})}:H^\M_{\vec{A}} \rightarrow H^{\N^-}_{\vec{B}}$ be the unique map. That is, given any two different iteration maps $i_0,i_1:\M \rightarrow \N^-$ according to $\M$'s iteration strategy, by strong $\vec{A}$-iterability, $i_0 \rest H^\M_{\vec{A}} = i_1 \rest H^\M_{\vec{A}}$, so the map $\pi_{(\M,\vec{A}),(\N,\vec{B})}$ is well-defined. The following theorem is basically due to Woodin. We just sketch the proof and give more details in the proof of Proposition \ref{ASatAD+}.
\begin{theorem} 
\label{iterability 0}
Assume $V=L(\powerset(\mathbb{R}))+ \textsf{AD}^+ + \textsf{MC}+\Theta = \theta_0$. Given any OD set of reals A and any n $\in \omega$, there is an n-suitable M that is strongly A-iterable. The same conclusion holds for any finite sequence $\vec{A}$ of OD sets of reals.
\end{theorem}
\begin{proof} We'll prove the theorem for $n=1$. The other cases are similar. So suppose not. By \rthm{main theorem}, $V=K(\mathbb{R})$. Then $V \vDash \phi$ where $\phi = (\exists \alpha) \ (K(\mathbb{R})|\alpha \vDash ``\textsf{ZF}^- + \Theta$ exists + $(\exists A) \ (A$ is OD and there is no 1-suitable strongly $A$-iterable mouse))".
\\
\indent Let $\gamma < \undertilde{\delta}^2_1$ be least such that $K(\mathbb{R})|\gamma \vDash \phi$. Such a $\gamma$ exists by $\Sigma_1$-reflection, i.e. Theorem \ref{fundamental result of ad+}. Then it is easy to see that $\gamma$ ends a proper weak gap, say $[\overline{\gamma},\gamma]$ for some $\overline{\gamma}<\gamma$. Fix the least such $A$ as above. By \cite{Scalesendgap} and the minimality of $\gamma$, we get a self-justifying-system (sjs) \index{self-justifying-system, sjs}$\langle A_i \ | \ i < \omega \rangle$ of OD$^{K(\mathbb{R})|\gamma}$ sets of reals in $K(\mathbb{R})|\gamma$ that seals the gap\footnote{This means that for all $i$, $\neg A_i$ and a scale for $A_i$ are in $\langle A_i \ | \ i < \omega \rangle$. Furthermore, the $A_i$'s are cofinal in the Wadge hierarchy of $K(\mathbb{R})|\gamma$.}. We may and do assume $A = A_0$. Let $\Gamma = \Sigma_1^{K(\mathbb{R})|\overline{\gamma}}$ and $\Omega$ a good pointclass beyond $K(\mathbb{R})|(\gamma+1)$, i.e. $\powerset(\mathbb{R})^{K(\mathbb{R})|(\gamma+1)} \subsetneq \undertilde{\Delta}_{\Omega}$. $\Omega$ exists because $\gamma <$ $\undertilde{\delta}^2_1$. Let $N^*$ be a coarse $\Omega$-Woodin, fully iterable mouse. Such an $N^*$ exists by \cite{DMATM} or by Theorem \ref{n*x}. In fact by Theorem \ref{n*x}, one can choose $N^*$ that Suslin captures $\Omega$ and the sequence $\langle A_i \ | \ i < \omega\rangle$. Also by \cite{DMATM}, there are club-in-$\textrm{OR}^{N^*}$ many $\Gamma$-Woodin cardinals in $N^*$. It can be shown that the $L[E]$-construction done inside $N^*$ reaches a $\mathcal{P}$ such that $\mathcal{P}$ is 1-suitable with respect to $\Gamma$ (hence has canonical terms for the $A_i$'s) and $\mathcal{P} \vDash ``\delta_0$ and $\delta_1$ are Woodin cardinals" where $\delta_0$ and $\delta_1$ are the first two $\Gamma$-Woodin cardinals in $N^*$. Let $\Sigma$ be the strategy for $\mathcal{P}$ induced by that of $N^*$. By lifting up to the background strategy and using term condensation for the self-justifying-system, we get that $\Sigma$ is guided by $\langle A_i \ | \ i < \omega \rangle$, hence ($\mathcal{P}, \Sigma$) is strongly $A$-iterable. But then $K(\mathbb{R})|\gamma \vDash$"$\mathcal{P}$ is strongly $A$-iterable." This is a contradiction.
\end{proof}
The theorem implies $\mathcal{F} \neq \varnothing$. Moreover, we have that $\mathcal{F}$ is a directed system because given any $(\M,\vec{A}), (\N,\vec{B}) \in \mathcal{F}$, we can do a simultaneous comparison of $(\M,\vec{A}), (\N,\vec{B})$, and some $(\P,\vec{A}\oplus\vec{B})\in\mathcal{F}$ using their iteration strategies to obtain some $(\Q, \vec{A}\oplus \vec{B})\in \mathcal{F}$ such that $(\M,\vec{A}), (\N,\vec{B})\leq_\mathcal{F} (\Q,\vec{A}\oplus \vec{B})$. We summarize facts about $\M_\infty$ proved in \cite{CMI} and \cite{steel2012hod}. These results are due to Woodin.
\begin{lemma}
\label{wellfoundedness}
\begin{enumerate}
\item $\M_\infty$ is wellfounded.
\item $\mathcal{M}_\infty$ has $\omega$ Woodin cardinals $(\delta^{\mathcal{M}_\infty}_i)_{i<\omega}$ cofinal in its ordinals.
\item $\theta_0 = \delta^{\mathcal{M}_\infty}_0$ and $\textrm{HOD}|\theta_0 = \mathcal{M}_\infty|\delta^{\mathcal{M}_\infty}_0$.
\end{enumerate}
\end{lemma}
We'll extend this computation to the full \textrm{HOD}. Now we define a strategy $\Sigma_\infty$ for $\mathcal{M}_\infty$. For each $A \in OD \cap \powerset(\mathbb{R})$, let $\tau^{\mathcal{M}_\infty}_{A,k}$ = common value of $\pi_{(\P,A),\infty}(\tau^\P_{A,k})$ where $\pi_{(\P,A),\infty}$ is the direct limit map and $\tau^\P_{A,k}$ is the standard term of $\P$ that captures $A$ at $\delta^P_k$. $\Sigma_\infty$ will be defined (in $V$) for (finite stacks of) trees on $\mathcal{M}_\infty|\delta^{\mathcal{M}_\infty}_0$ in $\mathcal{M}_\infty$. For $k \geq n$, $\mathcal{M}_\infty \vDash ``Col(\omega,\delta^{\mathcal{M}_\infty}_n) \times Col(\omega,\delta^{\mathcal{M}_\infty}_k) \Vdash (\tau^{\mathcal{M}_\infty}_{A,n})_g =  (\tau^{\mathcal{M}_\infty}_{A,k})_h \cap \mathcal{M}_\infty[g]$" where $g$ is $Col(\omega,\delta^{\mathcal{M}_\infty}_n)$ generic and $h$ is $Col(\omega,\delta^{\mathcal{M}_\infty}_k)$ generic. This is just saying that the terms cohere with one another. 
\\
\indent Let $G$ be $Col(\omega, <\lambda^{\mathcal{M}_\infty})$ generic over $\mathcal{M}_\infty$ where $\lambda^{\mathcal{M}_\infty}$ is the sup of Woodin cardinals in $\mathcal{M}_\infty$. Then $\mathbb{R}^*_G$ is the symmetric reals and $A^*_G$ := $\cup_k(\tau^{\mathcal{M}_\infty}_{A,k})_{G|\delta^{\mathcal{M}_\infty}_k}$. 
\begin{proposition}
\label{ASatAD+}
For all $A\subseteq \mathbb{R}$, $A$ is $OD$, $L(A^*_G, \mathbb{R}^*_G) \vDash \textsf{AD}^+$.
\end{proposition}
\begin{proof}
Suppose not. Using $\Sigma_1$-reflection, there is an $N$, which is a level of $K(\mathbb{R})$ below $\undertilde{\delta}^2_1$ satisfying the statement (T) $\equiv$ ``$\textsf{AD}^+ + \textsf{ZF}^-+\textsf{DC}+ \textsf{MC} + \exists A (A \textrm{ is OD and } L(A^*_G,\mathbb{R}^*_G) \nvDash \textsf{AD}^+))$". We may assume $N$ is the first such level. Let
\begin{center}
$U=\{ (x, \M) : \M$ is a sound $x$-mouse, $\rho_\omega(\M)=\{x\}$, and has an iteration strategy in $N\}.$
\end{center}
 Since $\textsf{MC}$ holds in $N$, $U$ is a universal $(\Sigma^2_1)^N$-set. Let $A\in N$ be an OD set of reals witnessing $\phi$. We assume that $A$ has the minimal Wadge rank among the sets witnessing $\phi$. Using the results of \cite{wilson2012contributions}, we can get a $\vec{B}=\la B_i: i<\omega\ra$ which is a self-justifying-system (sjs) such that $B_0 = U$ and each $B_i\in N$. Furthermore, we may assume that each $B_i$ is OD in $N$.

Because $\textsf{MC}$ holds and $\Gamma^*=_{def}\powerset(\mathbb{R})^N\varsubsetneq \utilde{\Delta}^2_1$, there is a real $x$ such that there is a sound mouse $\M$ over $x$ such that $\rho(\M)=x$ and $\M$ doesn't have an iteration strategy in $N$. Fix then such an $(x, \M)$ and let $\Sigma$ be the strategy of $\M$. Let $\Gamma$ be a good pointclass such that $Code(\Sigma), \vec{B}, U, U^c\in \utilde{\Delta}_\Gamma$. Let $F$ be as in \rthm{n*x} and let $z$ be such that $(\N^*_z, \d_z,\Sigma_z)$ Suslin captures $Code(\Sigma), \vec{B}, U, U^c$.

We let $\Phi=(\Sigma^2_1)^N$. We have that $\Phi$ is a good pointclass. Because $\vec{B}$ is Suslin captured by $\N^*_z$, we have $(\delta_z^+)^{\N^*_z}$-complementing trees $T, S\in \N^*_z$ which capture $\vec{B}$. Let $\k$ be the least cardinal of $\N^*_z$ which, in $\N^*_z$ is $<\d_z$-strong.\\

\noindent \textbf{Claim 1.} \textit{$\N^*_z\models ``\k$ is a limit of points $\eta$ such that $Lp^{\Gamma^*}(\N^*_z|\eta)\models ``\eta$ is Woodin".}
\begin{proof}
The proof is an easy reflection argument. Let $\lambda=\d_z^+$ and let $\pi: M\rightarrow \N^*_z|\l$ be an elementary substructure such that
\begin{enumerate}
\item $T,S\in ran(\pi)$,
\item if $\cp(\pi)=\eta$ then $V_\eta^{\N^*_z}\subseteq M$, $\pi(\eta)=\d_z$ and $\eta>\kappa$.
\end{enumerate}
By elementarity, we have that $M\models ``\eta$ is Woodin". Letting $\pi^{-1}(\la T, S\ra)=\la \bar{T}, \bar{S}\ra$, we have that $(\bar{T},\bar{S})$ Suslin captures the universal $\Phi$ set over $M$ at $(\eta^+)^M$. This implies that $M$ is $\Phi$-full and in particular, $Lp^{\Gamma^*}(\N^*_z|\eta)\in M$. Therefore, $Lp^{\Gamma^*}(\N^*_z|\eta)\models ``\eta$ is Woodin". The claim then follows by a standard argument.
\end{proof}

Let now $\la\eta_i : i<\omega\ra$ be the first $\omega$ points $<\kappa$ such that for every $i<\omega$, $Lp^{\Gamma^*}(\N^*_z|\eta_i)\models ``\eta_i$ is Woodin". Let now $\la \N_i: i<\omega\ra$ be a sequence constructed according to the following rules:
\begin{enumerate}
\item $\N_0=L[\vec{E}]^{\N^*_z|\eta_0}$,
\item $\N_{i+1}=(L[\vec{E}][\N_i])^{\N^*_z|\eta_{i+1}}$.
\end{enumerate}
Let $\N_\omega=\cup_{i<\omega}\N_i$. \\

\noindent \textbf{Claim 2.} \textit{For every $i<\omega$, $\N_\omega\models ``\eta_{i}$ is Woodin" and $\N_{\omega}|(\eta_i^+)^{\N_\omega}=Lp^{\Gamma^*}(\N_i)$.}
\begin{proof}
It is enough to show that
\begin{enumerate}
\item $\N_{i+1}\models ``\eta_i$ is Woodin",
\item $\N_i=V_{\eta_i}^{\N_{i+1}}$,
\item $\N_{i+1}|(\eta_i^+)^{\N_{i+1}}=Lp^{\Gamma^*}(\N_i)$.
\end{enumerate}
To show 1-3, it is enough to show that if
$\W\trianglelefteq \N_{i+1}$ is such that  $\rho_\omega(W)\leq\eta_i$ then the fragment of $\W$'s iteration strategy which acts on trees above $\eta_i$ is in $\Gamma^*$. Fix then $i$ and $\W\trianglelefteq \N_{i+1}$ is such that  $\rho_\omega(W)\leq\eta_i$. Let $\xi$ be such that the if $\S$ is the $\xi$-th model of the full background construction producing $\N_{i+1}$ then $\mathbb{C}(\S)=\W$. Let $\pi: \W\rightarrow \S$ be the core map. It is a fine-structural map but that it irrelevant and we surpass this point. The iteration strategy of $\W$ is the $\pi$-pullback of the iteration strategy of $\S$. Let then $\nu<\eta_{i+1}$ be such that $\S$ is the $\xi$-th model of the full background construction of $\N^*_x|\nu$. To determine the complexity of the induced strategy of $\S$ it is enough to determine the strategy of $\N^*_x|\nu$ which acts on non-dropping stacks that are completely above $\eta_i$. Now, notice that by the choice of $\eta_{i+1}$, for any non-dropping tree $\T$ on $\N^*_x|\nu$ which is above $\eta_i$ and is of limit length, if $b=\Sigma(\T)$ then $\Q(b, \T)$ exists and $\Q(b, \T)$ has no overlaps, and $\Q(b, \T)\trianglelefteq Lp^{\Gamma^*}(\M(\T))$. This observation indeed shows that the fragment of the iteration strategy of $\N^*_x|\nu$ that acts on non-dropping stack that are above $\eta_i$ is in $\Gamma^*$. Hence, the strategy of $\W$ is in $\Gamma^*$.
\end{proof}

We now claim that there is $\W\trianglelefteq Lp(\N_\omega)$ such that $\rho(W)<\eta_\omega$. To see this suppose not. It follows from $\textsf{MC}$ that $Lp(\N_\omega)$ is $\Sigma^2_1$-full. We then have that $x$ is generic over $Lp(\N_\omega)$ at the extender algebra of $\N_\omega$ at $\eta_0$. Because $Lp(\N_\omega)[x]$ is $\Sigma^2_1$-full, we have that $\M\in Lp(\N_\omega)[x]$ and $Lp(\N_\omega)\models ``\M$ is $\eta_\omega$-iterable" by fullness of $Lp(\N_\omega)$. Let $\S=(L[\vec{E}][x])^{\N_\omega[x]|\eta_2}$ where the extenders used have critical point $>\eta_0$. Then working in $\N_{\omega}[x]$ we can compare $\M$ with $\S$. Using standard arguments, we get that $\S$ side doesn't move and by universality, $\M$ side has to come short (see \cite{ATHM}). This in fact means that $\M\trianglelefteq \S$. But the same argument used in the proof of Claim 2 shows that every $\K\trianglelefteq \S$ has an iteration strategy in $\Gamma^*$, contradiction!

Let now $\W\trianglelefteq Lp(\N_\omega)$ be least such that $\rho_\omega(\W)<\eta_\omega$. Let $k, l$ be such that $\rho_{l}(\W)<\eta_k$. We can now consider $\W$ as a $\W|\eta_k$-mouse and considering it such a mouse we let $\N=\mathbb{C}_{l}(\W)$. Thus, $\N$ is sound above $\eta_k$. We let $\la \gamma_i : i<\omega\ra$ be the Woodin cardinals of $\N$ and $\gamma=\sup_{i<\omega}\gg_i$.

Let $\Lambda$ be the strategy of $\N$. We claim that $\Lambda$ is $\Gamma^*$-fullness preserving above $\gamma_k$. To see this fix $\N^*$ which is a $\Lambda$-iterate of $\N$ such that the iteration embedding $i: \N\rightarrow \N^*$ exists. If $\N^*$ isn't $\Gamma^*$-full then there is a strong cutpoint $\nu$ of $\N^*$ and a $\N^*|\nu$-mouse $\W$ with iteration strategy in $\Gamma^*$ such that $\rho_\omega(\W)=\nu$ and $\W\ntrianglelefteq \N^*$. If $\N^*$ is not sound above $\nu$ then $\N^*$ wins the coiteration with $\W$; but this then implies $\W \ntriangleleft \N^*$, which contradicts our assumption. Otherwise, $\N^* \triangleleft \W$, which is also a contradiction. Hence $\Lambda$ is $\Gamma^*$-fullness preserving.
\\
\indent Now it's not hard to see that $\N$ has the form $\mathcal{J}_{\xi+1}^{\vec{E}}(\N|\gamma)$ and $\mathcal{J}_\xi^{\vec{E}}(\N|\gamma)$ satisfies ``my derived model at $\gamma$ satisfies (T)." This is basically the content of Lemma 7.5 of \cite{steel2012hod}. The argument is roughly that we can iterate $\N$ to an $\R$ such that $\R = \mathcal{J}(\Q_\infty^+)$, where $\Q_\infty^+$ is discussed in the previous subsection and the Prikry forcing is done inside $N$.

Now let $\N^*$ be the transitive collapse of the pointwise definable hull of $\N|\xi$. We can then realize $N$ as a derived model of a $\Lambda$-iterate $\R$ of $\N^*$ such that $\R$ extends a Prikry generic over $N$ (the Prikry forcing is discussed in the previous subsection and $\R$ is in fact the $\Q_\infty^+$, where $\Q_\infty^+$ is as in the previous subsection). We can then use Lemmas 7.6, 7.7, and 6.51 of \cite{steel2012hod} to show that $\M_\infty^N$ is a $\Lambda$-iterate of $\N^*$. 

\indent In $N$, let $A \subseteq \mathbb{R}$ be the least OD set such that $L(A^*_G, \mathbb{R}^*_G) \nvDash \textsf{AD}$. Then there is an iterate $\mathcal{M}$ of $\mathcal{N}^*$ having preimages of all the terms $\tau^{\M_\infty}_{A,k}$. We may assume $\mathcal{M}$ has new derived model $N$ (this is possible by the above discussion) and suitable initial segments of $\M$ are points in the HOD direct limit system of $N$. Since $N\vDash \textsf{AD}^+$, $\mathcal{M}$ thinks that its derived model satisfies that $L(A,\mathbb{R}) \vDash \textsf{AD}^+$. Now iterate $\mathcal{M}$ to $\mathcal{P}$ such that $\M_\infty$ is an initial segment of $\mathcal{P}$. By elementarity $L(A^*_G,\mathbb{R}^*_G) \vDash \textsf{AD}^+$. This is a contradiction.
\end{proof}
\begin{definition}[$\Sigma_\infty$]
\label{Sigma infinity}
Given a normal tree $\mathcal{T} \in \mathcal{M}_\infty$ and $\mathcal{T}$ is based on $\mathcal{M}_\infty|\theta_0$. $\mathcal{T}$ is by $\Sigma_\infty$ if the following hold (the definition is similar for finite stacks):
\begin{itemize}
\item If $\mathcal{T}$ is short then $\Sigma$ picks the branch guided by $\Q$-structure (as computed in $\M_\infty)$.
\item If $\mathcal{T}$ is maximal then $\Sigma_\infty(\mathcal{T})$ = the unique cofinal branch $b$ which moves $\tau^{\mathcal{M}_\infty}_{A,0}$ correctly for all $A \in OD \cap \powerset(\mathbb{R})$ i.e. for each such $A$, $i_b(\tau^{\M_\infty}_{A,0}) = \tau^{\M^{\mathcal{T}}_b}_{A^*,0}$.
\end{itemize}
\end{definition}
\begin{lemma}\label{SigmaInftyTExists} Given any such $\mathcal{T}$ as above, $\Sigma_\infty(\mathcal{T})$ exists.
\end{lemma}
\begin{proof} Suppose not. By reflection (Theorem \ref{fundamental result of ad+}), there is a (least) $\gamma < \undertilde{\delta}^2_1$ such that $N=_{def} K(\mathbb{R})|(\gamma) \vDash \phi$ where $\phi$ is the statement ``$\textsf{ZF}^-+\textsf{DC}+ \textsf{MC} +\exists \mathcal{T} (\Sigma_\infty(\mathcal{T}) \textrm{ doesn't exist})$". We have a self-justifying-system $\vec{B}$ for $\Gamma^* = \powerset(\mathbb{R})^N$. By the construction of Proposition\ref{ASatAD+}, there exists a mouse $\N$ with $\omega$ Woodin cardinals which has strategy $\Gamma$ guided by $\vec{B}$. 

By reflecting to a countable hull, it's easy to see that $\M_\infty^{N}$ is a $\Gamma$-tail of $\N$ (the reflection is just to make all relevant objects countable). Note that by Theorem \ref{iterability 0}, for every $A$, which is OD in $N$, there is a $\Gamma$-iterate of $\N$ that is strongly $A$-iterable. Let $\Sigma_\infty^N$ be the strategy of $\M_\infty^N$ given by $\Gamma$. It follows then that for any tree $\mathcal{T}$, $\Sigma_\infty^N(\mathcal{T})$ is the limit of all branches $b_{A^*}$, where $A$ is OD in $N$ and $b_{A^*}$ moves the term relation for $A^*$ correctly. This fact can be seen in $N$. This gives a contradiction.
\end{proof}
\indent It is evident that $L(\mathcal{M}_\infty, \Sigma_\infty) \subseteq \textrm{HOD}$. Next, we show $\M_\infty$ and $\Sigma_\infty$ capture all of $\textrm{HOD}$. In $L(\mathcal{M}_\infty, \Sigma_\infty)$, first construct (using $\Sigma_\infty$) a mouse $\mathcal{M}_\infty^+$ extending $\mathcal{M}_\infty$ such that o($\mathcal{M}_\infty$) is the largest cardinal of $\mathcal{M}_\infty^+$ as follows: 
\begin{enumerate}
\item Let $\mathbb{R}^*_G$ be the symmetric reals obtained from a generic $G\subseteq Col(\omega, <\lambda^{\mathcal{M}_\infty})$ over $L(\mathcal{M}_\infty)$.
\item For each $A^*_G$ (defined as above where $A \in \powerset(\mathbb{R}) \cap \textrm{OD}^{K(\mathbb{R})}$) (we know $L(\mathbb{R}^*_G,A^*_G) \vDash \textsf{AD}^+$), S-translate the $\mathbb{R}^*_G$-mice in this model to mice $\mathcal{S}$ extending $\mathcal{M}_\infty$ with the derived model of $\S$ at $\lambda^{\mathcal{M}_\infty}$ $D^+(\mathcal{S}, \lambda^{\mathcal{M}_\infty}) = L(\mathbb{R}^*_G, A^*_G)$. This is again proved by a reflection argument similar to that in Proposition \ref{ASatAD+}.
\item Let $\mathcal{M}_\infty^+ = \cup_{\mathcal{S}} \mathcal{S}$ for all such $\mathcal{S}$ as above. It's easy to see that $\mathcal{M}_\infty^+$ is independent of $G$. By a reflection argument like that in Proposition \ref{ASatAD+}, we get that mice over $\mathcal{M}_\infty$ are all compatible, no levels of $\mathcal{M}_\infty^+$ projects across $o(\mathcal{M}_\infty)$.  
\end{enumerate}
\begin{remark} \rm{$\delta_0^{\M_\infty}$ is not collapsed by $\Sigma_\infty$ because it is a cardinal in \textrm{HOD}. $\Sigma_\infty$ is used to obtain the $A^*_G$ above by moving correctly the $\tau^{\M_\infty}_{A,0}$ in genericity iterations. $L(\M_\infty)$ generally does not see the sequence $\langle \tau^{\M_\infty}_{A,k} \ | \ k \in \omega \rangle$ hence can't construct $A^*_G$; that's why we need $\Sigma_\infty$. Since $\Sigma_\infty$ collapses $\delta^{\M_\infty}_1, \delta^{\M_\infty}_2...$ by genericity iterating $\M_\infty|\delta_0^{\M_\infty}$ to make $\M_\infty|\delta_i^{\M_\infty}$ generic for $i>0$, it doesn't make sense to talk about $D(L(\M_\infty,\Sigma_\infty))$}.
\end{remark}
\begin{lemma} 
\label{keylemma}
\textrm{HOD} $\subseteq L(\M_\infty, \Sigma_\infty)$
\end{lemma}
\begin{proof}  Using \rthm{WoodinVopenka}, we know $\textrm{HOD} = L[P]$ for some $P \subseteq \Theta$. Therefore, it is enough to show P $\in L(\M_\infty, \Sigma_\infty)$. Let $\phi$ be a formula defining $P$, i.e.
\begin{equation*}
\alpha \in P \Leftrightarrow K(\mathbb{R}) \vDash \phi[\alpha].
\end{equation*}
Here we suppress the ordinal parameter. Now in $L(\M_\infty, \Sigma_\infty)$ let $\pi : \M_\infty | ((\delta_0^{\M_\infty})^{++})^{\M_\infty} \rightarrow (\M_\infty)^{D^+(\M_\infty^+,\lambda^{\M_\infty})}$ where $\pi$ is according to $\Sigma_\infty$. We should note that $\Sigma_\infty$-iterates are cofinal in the directed system $\mathcal{F}$ defined in $D(\M_\infty^+,\lambda^{\M_\infty})$ by the method of boolean comparisons (see \cite{steel2012hod} for more on this).
\\
\\
\textbf{Claim:}  $K(\mathbb{R}) \vDash \phi[\alpha] \Leftrightarrow D^+(\M_\infty^+, \lambda^{M_\infty}) \vDash \phi[\pi(\alpha)]$ (\textasteriskcentered\textasteriskcentered)
\begin{proof}  Otherwise, reflect the failure of (\textasteriskcentered\textasteriskcentered) as before to the least $K(\mathbb{R})|\gamma$ and get a self-justifying-system $\vec{B}$ of OD sets along with an $\omega$-suitable mouse $\mathcal{N}$ with $\vec{B}$-guided iteration strategy $\Gamma$. By genericity iteration above its first Woodin, we may assume $D^+(\mathcal{N}, \lambda^{\mathcal{N}}) = K(\mathbb{R})|\gamma$. Fix an $\alpha$ witnessing the failure of (\textasteriskcentered\textasteriskcentered). Let $\sigma : \mathcal{N}|((\delta^{\mathcal{N}}_0)^{++})^{\mathcal{N}} \rightarrow (\M_\infty)^{D^+(\mathcal{N},\lambda^\N)}$ be the direct limit map by $\Gamma$ (by taking a countable hull containing all relevant objects, we can assume $\sigma$ exists). We may assume there is an $\overline{\alpha}$ such that $\sigma(\overline{\alpha}) = \alpha$. Notice here that $\Sigma_\infty^{K(\mathbb{R})|\gamma}$ is a tail of $\Gamma$ as $\Sigma_\infty^{K(\mathbb{R})|\gamma}$ moves all the term relations for $OD^{{K(\mathbb{R})|\gamma}}$ sets of reals correctly and $\Gamma$ is guided by the self-justifying system $\vec{B},$ which is cofinal in $\powerset(\mathbb{R}) \cap OD^{K(\mathbb{R})|\gamma}$. It then remains to see that:
\begin{equation*}
D^+(\M_\infty^+,\lambda^{\M_\infty})\vDash\phi[\pi(\alpha)] \Leftrightarrow D^+(\N,\lambda^\N) \vDash \phi[\sigma(\overline{\alpha})] \ \ (\textasteriskcentered\textasteriskcentered \textasteriskcentered).
\end{equation*} 
To see that (\textasteriskcentered\textasteriskcentered \textasteriskcentered) holds, we need to see that the fragment of $\Gamma$ that defines $\sigma(\overline{\alpha})$ can be defined in $D^+(\N, \lambda^\N)$. This then will give the equivalence in (\textasteriskcentered\textasteriskcentered \textasteriskcentered). Because $\alpha < \delta_0^{\M_\infty^{K(\mathbb{R})|\gamma}} = \delta_0^{\M_\infty^{D^+(\N,\lambda^\N)}}$, pick an $A \in \vec{B}$ such that $\gamma_{A,0}^{D(\N,\lambda^\N)} > \alpha$. Then the fragment of $\Gamma$ that defines $\sigma(\overline{\alpha})$ is definable from A (and $\N|(\delta_0^N)$) in $D^+(\N, \lambda^\N)$, which is what we want. 

The equivalence (\textasteriskcentered\textasteriskcentered \textasteriskcentered) gives us a contradiction.
\end{proof}
The claim finishes the proof of $P \in L(\M_\infty, \Sigma_\infty)$ because the right hand side of the equivalence (\textasteriskcentered\textasteriskcentered) can be computed in $L(\M_\infty,\Sigma_\infty)$. This then implies $\textrm{HOD} = L[P]$ $\subseteq L(\M_\infty, \Sigma_\infty)$.
\end{proof}

\begin{remark}
\rm{Woodin (unpublished) has also computed the full \textrm{HOD} for models satisfying $V = L(\powerset(\mathbb{R}))+\textsf{AD}^+ +\Theta=\theta_0$. To the best of the author's knowledge, here's a very rough idea of his computation. Let $\M_\infty, \Sigma_\infty, P$ be as above. For each $\alpha < \Theta$, let $\Sigma_\alpha$ be the fragment of $\Sigma_\infty$ that moves $\alpha$ along the good branch of a maximal tree. Woodin shows that the structure $(\mathbb{R}^*_G, \langle\Sigma_\alpha \ | \ \alpha < \Theta\rangle)$ can compute the set $P$. This then gives us that $\textrm{HOD} \subseteq L(\M_\infty, \Sigma_\infty)$}.
\end{remark}

\section{The $\Theta = \theta_{\alpha+1}$ case}

Again, we assume $(\textasteriskcentered)$. Assume also that $\Theta = \theta_{\alpha+1}$ for some $\alpha$ and there is a hod pair $(\P,\Sigma)$ as in the hypothesis of Theorem \ref{main theorem} for $M$. By Theorem \ref{main theorem}, $V = K^\Sigma(\mathbb{R})$.
\\
\indent First we need to compute $V^{\textrm{HOD}}_\Theta$. Here's what is done in \cite{ATHM} regarding this computation.
\begin{theorem}[Sargsyan, see Section 4.3 in \cite{ATHM}]
\label{Grigor hod computation}
Let $\Gamma = \{A \subseteq \mathbb{R} \ | \ w(A) < \theta_\alpha\}$. Then there is a hod pair $(\P,\Sigma)$ such that
\begin{enumerate}
\item $\Sigma$ is fullness preserving and has branch condensation;
\item $\Gamma(\P,\Sigma) = \Gamma$ where $\Gamma(\P,\Sigma) = \{A \subseteq \mathbb{R} \ | \ A \leq_w \Sigma_{\Q(\beta)} \textrm{ for some } \beta < \lambda^\Q \textrm{ where } \Q  \textrm{ is a } \Sigma-iterate \textrm{ of }\P\}$;
\item $\mathcal{M}^+_\infty(\P,\Sigma)|\theta_\alpha = V^{\textrm{HOD}}_\alpha$, where $\mathcal{M}^+_\infty(\P,\Sigma)$ is the direct limit of all $\Sigma$-iterates of $\P$.
\end{enumerate}
\end{theorem}

It is clear that there is no hod pair $(\P,\Sigma)$ satisfying Theorem \ref{Grigor hod computation} with $\Gamma$ replaced by $\powerset(\mathbb{R})$ as this would imply that $\Sigma \notin V$. So to compute $V^{\textrm{HOD}}_\Theta$, we need to mimic the computation in Section 2. For a more detailed discussion regarding Definitions \ref{suitability}, \ref{B iterability}, and \ref{strong B iterability}, see Section 3.1 of \cite{ATHM}.
\begin{definition}[$n$-suitable pair]
\label{suitability}
$(\P,\Sigma)$ is an n-suitable pair if there is $\delta$ such that $(\P|(\delta^{+\omega})^\P, \Sigma)$ is a hod pair and 
\begin{enumerate}
\item $\P \vDash$ \textsf{\textsf{ZF}C} - Replacement + ``there are n Woodin cardinals, $\eta_0 < \eta_1 < ... < \eta_{n-1}$ above $\delta$";
\item $o(\P) = sup_{i<\omega}{(\eta_{n-1})^{+i}}^P$;
\item $\P$ is a $\Sigma$-mouse over $\P|\delta$;
\item for any $\P$-cardinal $\eta > \delta$, if $\eta$ is a strong cutpoint then $\P|(\eta^+)^\P = Lp^{\Sigma}(\P|\eta)$.
\end{enumerate}
\end{definition}
For $\P, \delta$ as in the above definition, let $\P^- = \P|(\delta^{+\omega})^\P$ and $\mathbb{B}(\P^-,\Sigma) = \{B \subseteq \powerset(\mathbb{R})\times\mathbb{R}\times\mathbb{R} \ | \ B \textrm{ is } OD, \textrm{ and for any } (\Q,\Lambda) \textrm{ iterate of } (\P^-, \Sigma), \textrm{ and for any } (x,y) \in B_{(\Q,\Lambda)}, x \textrm{ codes } \Q\}$. Suppose $B \in \mathbb{B}(P^-,\Sigma)$ and $\kappa < o(\P)$. Let $\tau^{\P}_{B,\kappa}$ be the canonical term in $\P$ that captures $B$ at $\kappa$ i.e. for any $g \subseteq Col(\omega,\kappa)$ generic over $\P$
\begin{equation*}
B_{(\P^-,\Sigma)} \cap \P[g] = (\tau^\P_{B,\kappa})_g.
\end{equation*}
For each $m<\omega$, let
\begin{equation*}
\gamma^{\P,\Sigma}_{B,m} = \textrm{sup}(Hull^\P(\tau^\P_{B,(\eta_{n-1}^{+m})^\P})\cap \eta_0),
\end{equation*} 
\begin{equation*}
H^{\P,\Sigma}_{B,m} = Hull^\P(\gamma^{\P,\Sigma}_{B,m}\cup\{\tau^\P_{B,(\eta_{n-1}^{+m})^\P}\}),
\end{equation*}
\begin{equation*}
\gamma^{\P,\Sigma}_B = \textrm{sup}_{m<\omega}\gamma^{\P,\Sigma}_{B,m},
\end{equation*}
and
\begin{equation*}
H^{\P,\Sigma}_B = \cup_{m<\omega}H^{\P,\Sigma}_{B,m}.
\end{equation*} 
Similar definitions can be given for $\gamma^{\P,\Sigma}_{\vec{B},m}, H^{\P,\Sigma}_{\vec{B},m}, \gamma^{\P,\Sigma}_{\vec{B}}, H^{\P,\Sigma}_{\vec{B}}$ for any finite sequence $\vec{B} \in \mathbb{B}(\P^-,\Sigma)$. One just needs to include relevant terms for each element of $\vec{B}$ in each relevant hull. Now we define the notion of $B$-iterability.
\begin{definition}[$B$-iterability]
\label{B iterability}
Let $(\P,\Sigma)$ be an $n$-suitable pair and $B \in \mathbb{B}(\P^-,\Sigma)$. We say $(\P,\Sigma)$ is $B$-iterable if for all $k<\omega$, player II has a winning quasi-strategy for the game $G^{(P,\Sigma)}_{B,k}$ defined as follows. The game consists of $k$ rounds. Each round consists of a main round and a subround. Let $(\P_0,\Sigma_0) = (\P,\Sigma)$. In the main round of the first round, player I plays countable stacks of normal nondropping trees based on $\P^-_0$ or its images and player II plays according to $\Sigma_0$ or its tails. If the branches chosen by player II does not move some term for $B$ correctly, he loses. Player I has to exit the round at a countable stage; otherwise, he loses. Suppose $(\P^*,\Sigma^*)$ is the last model after the main round is finished. In the subround, player I plays a normal tree above $(\P^*)^-$ or its images based on a window of two consecutive Woodins. Player II plays a branch that moves all terms for $B$ correctly. Otherwise, he loses. Suppose $(\P_1,\Sigma_1)$ is the last model of the subround. If II hasn't lost, the next round proceeds the same way as the previous one but for the pair $(\P_1,\Sigma_1)$. If the game lasts for k rounds, II wins.
\end{definition}
\begin{definition}[Strong $B$-iterability]
\label{strong B iterability}
Let $(\P,\Sigma)$ be an $n$-suitable pair and $B \in \mathbb{B}(\P^-,\Sigma)$. We say $(\P,\Sigma)$ is strongly $B$-iterable if $(\P,\Sigma)$ is $B$-iterable and if $r_1$ is a run of $G^{\P,\Sigma}_{B,n_1}$ and $r_1$ is a run of $G^{\P,\Sigma}_{B,n_2}$ for some $n_1,n_2 < \omega$ according to the winning quasi-strategy of $\P$ and the runs produce the same end model $\Q$ then the runs move the hull $H^{\P,\Sigma}_B$ the same way. That is if $i_1$ and $i_2$ are $B$-iteration maps accoring to $r_1$ and $r_2$ respectively then $i_1\rest H^{\P,\Sigma}_B = i_2\rest H^{\P,\Sigma}_B$.
\end{definition}
 Now we're ready to define our direct limit system. Let
\begin{eqnarray*}
\mathcal{F} = \{(P,\Sigma, \vec{B}) \ &|& \ \vec{B} \in \mathbb{B}(P^-,\Sigma)^{<\omega}, (P^-,\Sigma) \textrm{ satisfies Theorem \ref{Grigor hod computation}}, (P,\Sigma) \textrm{ is $n$-suitable } \\&& \textrm{for some $n$, }\textrm{and } (P, \Sigma) \textrm{ is strongly } \vec{B}\textrm{-iterable}\}.
\end{eqnarray*}
The ordering on $\mathcal{F}$ is defined as follows:
\begin{eqnarray*}
(\P,\Sigma, \vec{B}) \preccurlyeq (\Q, \Lambda, \vec{C}) &iff& \vec{B} \subseteq \vec{C}, \exists k\exists r(r \textrm{ is a run of } G^{\P,\Sigma}_{B,k} \textrm{ with the last model } \P^* \\ && \textrm{such that } (\P^*)^- = \Q^-, \ \Sigma_{(\P^*)^-} = \Lambda, \P^* = \Q|(\eta^{+\omega})^\Q \\ &&\textrm{ where } \Q \vDash \eta > o(Q^-) \  is \ Woodin). 
\end{eqnarray*}
Suppose $(P,\Sigma, \vec{B}) \preccurlyeq (Q, \Lambda, \vec{C})$ then there is a unique map $\pi^{(\P,\Sigma),(\Q,\Delta)}_{\vec{B}}: H^{\P,\Sigma}_{\vec{B}} \rightarrow H^{\Q,\Lambda}_{\vec{B}}$. $(\mathcal{F}, \preccurlyeq)$ is then directed. Let
\begin{equation*} 
\mathcal{M}_\infty = \textrm{direct limit of } (\mathcal{F},\preccurlyeq) \textrm{ under maps }\pi^{(\P,\Sigma),(\Q,\Delta)}_{\vec{B}}. 
\end{equation*}
Also for each $(\P,\Sigma,\vec{B}) \in \mathcal{F}$, let
\begin{equation*}
\pi^{(\P,\Sigma),\infty}_{\vec{B}}: H^{\P,\Sigma}_{\vec{B}} \rightarrow \M_\infty
\end{equation*}
be the natural map.
\\
\indent Clearly, $\mathcal{M}_\infty \subseteq \textrm{HOD}$. But first, we need to show $\mathcal{F} \neq \emptyset$. In fact, we prove a stronger statement.
\begin{theorem}
\label{iterability n}
Suppose $(\P,\Sigma)$ satisfies Theorem \ref{Grigor hod computation}. Let $B \in \mathbb{B}(\P,\Sigma)$. Then for each $1 \leq n < \omega,$ there is a $\Q$ such that $\Q^-$ is a $\Sigma$-iterate of $\P^-$, $(Q, \Sigma_{\Q^-})$ is $n$-suitable and $(\Q,\Sigma_{\Q^-}, B) \in \mathcal{F}$.
\end{theorem}
\begin{proof}
Suppose not. By $\Sigma_1$-reflection (Theorem \ref{fundamental result of ad+}), there is an transitive model $N$ coded by a Suslin, co-Suslin set of reals such that $Code(\Sigma) \in \powerset(\mathbb{R})^N$ and
\begin{eqnarray*}
N &\vDash& \textsf{ZF}^- + \textsf{AD}^+ + S\textsf{MC} + ``\Theta \textrm{ exists and is successor in the Solovay sequence }" + \\ && ``\exists B \in \mathbb{B}(\P,\Sigma)(\nexists \Q, n)((\Q,\Sigma) \textrm{ is n-suitable and }(\Q,\Sigma,B) \in \mathcal{F})".
\end{eqnarray*}
We take a minimal such $N$ and fix a $B \in \mathbb{B}(\P,\Sigma)^N$ witnessing the failure of the Theorem in $N$. Using Theorem \ref{n*x} and the assumption on $N$, there is an $x \in \mathbb{R}$ and a tuple $\langle N^*_x, \delta_x, \Sigma_x\rangle$ satisfying the conclusions of Theorem \ref{n*x} relative to $\Gamma$- a good pointclass containing $(\powerset(\mathbb{R})^N, N's\textrm{ first order theory})$. Futhermore, let's assume that $N^*_x$ Suslin captures $\langle A \ | \ A \textrm{ is projective in }\Sigma\rangle)$. Let $\Omega = \powerset(\mathbb{R})^N$. For simplicity, we show that in $N$, there is a $\Sigma$-iterate $(\R,\Sigma_\R)$ such that there is a $1$-suitable $(\S,\Sigma_\R)$ such that $(\S,\Sigma_\R,B) \in \mathcal{F}$. 
\\
\indent By the assumption on $N$, $N \vDash V = K^\Sigma(\mathbb{R})$. Now $N^*_x$ has club many $(\Sigma^2_1)^\Omega$ Woodins below $\delta_x$ by a standard argument (see \cite{TWMS}). Hence, the full background construction $L[E, \Sigma][\P]$ done in $N^*_x$ will reach a model having $\omega$ Woodins (which are the first $\omega \ (\Sigma^2_1)^\Omega$ Woodins in $N^*_x$) and projecting across the sup of its first $\omega$ Woodins. Let $\Q$ be the first model in the construction with that property. By coring down if necessary, we may assume that $\Q$ is sound. Let $\langle \delta_i^\Q \ | \ i < \omega \rangle$ be the first $\omega$ Woodins of $\Q$ above $o(\P)$. A similar self-explanatory notation will be used to denote the Woodins of any $\Lambda$-iterate of $\Q$. Hence $\rho_\omega(\Q) < sup_{i<\omega}\delta_i$. Let $\Lambda$ (which extends $\Sigma$) be the strategy of $\Q$ induced from the background universe. $\Lambda$ is $\Omega$-fullness preserving. At this point it's not clear that $\Lambda$ has branch condensation. The proof of Theorem \ref{iterability 0} doesn't generalize as it's not clear what the corresponding notion of a self-justifying-system for sets in $\mathbb{B}(\P,\Sigma)$ is.
\\
\indent We in fact show a bit more. We show that an iterate $(\R,\Lambda_\R)$ of $(\Q,\Lambda)$ has strong $B$ condensation in that if $i: \R \rightarrow \S$ is according to $\Lambda_\R$ and below $\delta_0^\Q$ and $j: \R \rightarrow \mathcal{W}$ is such that there is a $k: \mathcal{W} \rightarrow \S$ such that $i = k\circ j$ then $i(\tau^\R_{B,\delta_0^\R}) = \tau^\S_{B,\delta_0^\S}$, $\mathcal{W}$ is full, and $k^{-1}(\tau^\S_{B,\delta_0^\S}) = \tau^{\mathcal{W}}_{B,j(\delta_0^\R)}$. That we get $\mathcal{W}$ being full is easy because $\Lambda \notin \N$. So we only need to prove the other two clauses. We also get strong $B$-iterability by the Dodd-Jensen property of $\Lambda_\R$. Once we have this pair $(\R,\Lambda_\R)$, we can just let our desired $\S$ to be $\R|((\delta_0^{\R})^{+\omega})^\R$. Suppose not. Using the property of $\Q$ and the relativized (to $\Sigma$) Prikry forcing in $N$ (see \cite{steel08}), we get that for any $n$, there is an iterate $\R$ of $\Q$ (above $\delta_0^\Q$) extending a Prikry generic and having $N$ as the (new) derived model (computed at the sup of the first $\omega$ Woodins above $o(\P)$). Furthermore, this property holds for any $\Lambda$ iterate of $\Q$. Without going further into details of the techniques used in \cite{steel08}, we remark that if $\R$ is an $\mathbb{R}$-genericity iterate of $\Q$, then the new derived model of $\R$ is $N$. In other words, once we know one such $\mathbb{R}$-genericity iterate of $\Q$ realizes $N$ as its derived model then all $\mathbb{R}$-genericity iterates of $\Q$ do. 
Let $(\phi, s)$ define $B$ over $N$, i.e.
\begin{equation*}
(\mathcal{R},\Psi, x, y) \in B \ iff \ N \vDash \phi[((\mathcal{R},\Psi, x, y)), s].
\end{equation*}   
The following argument mirrors that of Lemma 3.2.15 in \cite{ATHM} though it's not clear to the author who this argument is orginially due to. The process below is described in Figure \ref{diagram}. From now to the end of the proof, all stacks on $\Q$ or its iterates thereof are below the $\delta_0^\Q$ or its image. By our assumption, there is $\langle \vec{\mathcal{T}}_i, \vec{\mathcal{S}}_i, \mathcal{Q}_i,\mathcal{R}_i, \pi_i, \sigma_i, j_i \ | \ i < \omega\rangle \in N$ such that
\begin{enumerate}
\item $\mathcal{Q}_0 = \Q$; $\vec{\mathcal{T}}_0$ is a stack on $\Q$ according to $\Lambda$ with last model $\mathcal{Q}_1$; $\pi_0 = i^{\vec{\mathcal{T}}_0}$; $\vec{\mathcal{S}}_0$ is a stack on $\Q$ with last model $\mathcal{R}_0$; $\sigma_0 = i^{\vec{\mathcal{S}}_0}$; and $j_0: \mathcal{R}_0 \rightarrow \mathcal{Q}_1$.
\item $\vec{\mathcal{T}}_i$ is a stack on $\mathcal{Q}_i$ according to $\Lambda$ with last model $\mathcal{Q}_{i+1}$; $\pi_i = i^{\vec{\mathcal{T}}_i}$; $\vec{\mathcal{S}}_i$ is a stack on $\mathcal{Q}_i$ with last model $\mathcal{R}_i$; $\sigma_i = i^{\vec{\mathcal{S}}_i}$; $j_0: \mathcal{R}_i \rightarrow \mathcal{Q}_{i+1}$.
\item for all $k$, $\pi_k = j_k\circ\sigma_k$.
\item for all $k$, $\pi_k(\tau^{\Q_k}_{B,\delta_0^{\Q_k}}) \neq \tau^{\Q_{k+1}}_{B,\delta_0^{\Q_{k+1}}}$ or $j_k(\tau^{\mathcal{R}_k}_{B,\delta_0^{\R_k}}) \neq \tau^{\mathcal{Q}_{k+1}}_{B,\delta_0^{\Q_{k+1}}}$.
\end{enumerate}
  
Let $\mathcal{Q}_\omega$ be the direct limit of the $\mathcal{Q}_i$'s under maps $\pi_i$'s. First we rename the $\langle\Q_i,\R_i,\pi_i,\sigma_i,j_i$$\ | \ i<\omega\rangle$ into $\langle\Q_i^0,\R_i^0,\pi_i^0,\sigma_i^0,j_i^0 \ | \ i<\omega\rangle$. We then assume that $N$ is countable (by working with a countable elementary substructure of $N$) and fix (in $V$) $\langle x_i \ | \ i<\omega\rangle$- a generic enumeration of $\mathbb{R}$. Using our assumption on $\Q$, we get $\langle \mathcal{Q}^n_i,\mathcal{R}^n_i, \pi^n_i, \sigma^n_i, j^n_i, \tau^n_1,k^n_i \ | \ n, i \leq \omega\rangle$ such that
\begin{enumerate}
\item $\mathcal{Q}^\omega_i$ is the direct limit of the $\Q^n_i$'s under maps $\tau^n_i$'s for all $i\leq \omega$.
\item $\mathcal{R}^\omega_i$ is the direct limit of the $\mathcal{R}^n_i$'s under maps $k^n_i$'s for all $i<\omega$.
\item $\Q^n_\omega$ is the direct limit of the $\Q^n_i$'s under maps $\pi^n_i$'s.
\item for all $n\leq \omega$, $i<\omega$, $\pi^n_i: \mathcal{Q}^n_i \rightarrow \mathcal{Q}^n_{i+1}$; $\sigma^n_i: \mathcal{Q}^n_i \rightarrow \mathcal{R}^n_i$; $j^n_i: \mathcal{R}^n_i \rightarrow \mathcal{Q}^n_{i+1}$ and $\pi^n_i = j^n_i\circ\sigma^n_i$.
\item Derived model of the $\mathcal{Q}^n_i$'s, $\mathcal{R}^n_i$'s is $N$.
\end{enumerate}
Then we start by iterating $\mathcal{Q}^0_0$ above $\delta_0^{\Q^0_0}$ to $\Q^1_0$ to make $x_0$-generic at $\delta_1^{\Q^1_0}$. During this process, we lift the genericity iteration tree to all $\R^0_n$ for $n<\omega$ and $\Q^0_n$ for $n \leq \omega$. We pick branches for the tree on $\Q^0_0$ by picking branches for the lift-up tree on $\Q^0_\omega$ using $\Lambda_{\Q^0_\omega}$. Let $\tau^0_0: \Q^0_0\rightarrow \Q^1_0$ be the iteration map and $\W$ be the end model of the lift-up tree on $\Q^0_\omega$. We then iterate the end model of the lifted tree on $\R^0_0$ to $\R^1_0$ to make $x_0$ generic at $\delta_1^{\R^1_0}$ with branches being picked by lifting the iteration tree onto $\W$ and using the branches according to $\Lambda_\W$. Let $k^0_0:\R^0_0 \rightarrow \R^1_0$ be the iteration embedding, $\sigma^1_0: \Q^1_0\rightarrow R^1_0$ be the natural map, and $\mathcal{X}$ be the end model of the lifted tree on the $\W$ side. We then iterate the end model of the lifted stack on $\Q^0_1$ to $\Q^1_1$ to make $x_0$ generic at $\delta_1^{\Q^1_1}$ with branches being picked by lifting the tree to $\mathcal{X}$ and using branches picked by $\Lambda_{\mathcal{X}}$. Let $\tau^0_1:\Q^0_1 \rightarrow \Q^1_1$ be the iteration embedding, $j^1_0: \R^1_0 \rightarrow \Q^1_1$ be the natural map, and $\pi^1_0 = j^1_0\circ\sigma^1_0$. Continue this process of making $x_0$ generic for the later models $\R^0_n$'s and $\Q^0_n$'s for $n < \omega$. We then let $\Q^1_\omega$ be the direct limit of the $\Q^1_n$ under maps $\pi^1_n$'s. We then start at $\Q^1_0$ and repeat the above process to make $x_1$ generic appropriate iterates of $\delta_2^{\Q^1_0}$ etc. This whole process define models and maps $\langle \mathcal{Q}^n_i,\mathcal{R}^n_i, \pi^n_i, \sigma^n_i, j^n_i, \tau^n_1,k^n_i \ | \ n, i \leq \omega\rangle$ as described above. See \rfig{diagram}.
\\
\indent Note that by our construction, for all $n < \omega$, the maps $\pi^0_n$'s and $\tau^n_\omega$'s are via $\Lambda$ or its appropriate tails; furthermore, $\Q^\omega_\omega$ is wellfounded and full (with respect to mice in $N$). This in turns implies that the direct limits $\Q^\omega_n$'s and $\R^\omega_n$'s are wellfounded and full. We must then have that for some $k$, for all $n \geq k$, $\pi^\omega_n(s) = s$. This implies that for all $n \geq k$
\begin{equation*}
\pi^\omega_n(\tau_{B,\delta_0^{\Q^\omega_n}}^{\mathcal{Q}^\omega_n}) = \tau_{B,\delta_0^{\Q^\omega_{n+1}}}^{\mathcal{Q}^\omega_{n+1}}. 
\end{equation*}
We can also assume that for all $n\geq k$, $\sigma^\omega_n(s) = s, j^\omega_n(s) = s$. Hence
\begin{equation*}
\sigma^\omega_n(\tau_{B,\delta_0^{\Q^\omega_n}}^{\mathcal{Q}^\omega_n}) = \tau_{B,\delta_0^{\R^\omega_{n}}}^{\mathcal{R}^\omega_{n}}. ;
\end{equation*}
\begin{equation*}
j^\omega_n(\tau_{B,\delta_0^{\R^\omega_n}}^{\mathcal{R}^\omega_n}) = \tau_{B,\delta_0^{\Q^\omega_{n+1}}}^{\mathcal{Q}^\omega_{n+1}}. ;
\end{equation*}
This is a contradiction, hence we're done.
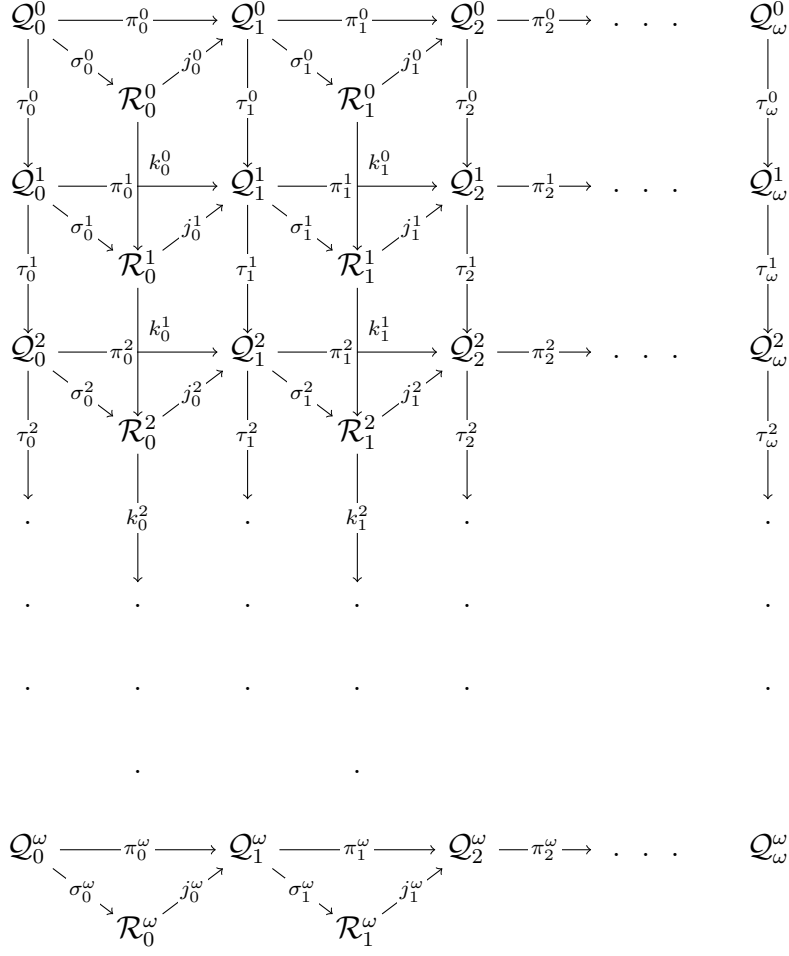
\begin{figure}
\centering
\begin{tikzpicture}[description/.style={fill=white,inner sep=1pt}]
\matrix (m) [matrix of math nodes, row sep=1.5em,
column sep=1.5em, text height=1ex, text depth=0.15ex]
{ \Q^0_0 & & \Q^0_1 & & \Q^0_2 & & \ .\ \ . \ \ . & \Q^0_\omega\\
  & \R^0_0 & & \R^0_1 & & & & \\
\Q^1_0 & & \Q^1_1 & & \Q^1_2 & & \ .\ \ . \ \ . & \Q^1_\omega\\
  & \R^1_0 & & \R^1_1 & & & & \\
\Q^2_0 & & \Q^2_1 & & \Q^2_2 & & \ .\ \ . \ \ . & \Q^2_\omega\\
  & \R^2_0 & & \R^2_1 & & & & \\
. &   & . &   & . &   & & . \\
. & . & . & . & . &   & & . \\
. & . & . & . & . &   & & . \\
  & . &   & . &   &   & &   \\
\Q^\omega_0 & & \Q^\omega_1 & & \Q^\omega_2 & & \ .\ \ . \ \ . & \Q^\omega_\omega\\
  & \R^\omega_0 & & \R^\omega_1 & & & & \\
};

\path[->,font=\scriptsize]
(m-1-1) edge node[description] {$ \pi^0_0 $} (m-1-3)
edge node[description] {$ \sigma^0_0 $} (m-2-2)
(m-2-2) edge node[description] {$ j^0_0 $} (m-1-3)
(m-1-3) edge node[description] {$ \pi^0_1 $} (m-1-5)
edge node[description] {$ \sigma^0_1 $} (m-2-4)
(m-2-4) edge node[description] {$ j^0_1$} (m-1-5)
(m-1-5) edge node[description] {$ \pi^0_2$} (m-1-7)
(m-3-1) edge node[left][description] {$ \pi^1_0 $} (m-3-3)
edge node[description] {$ \sigma^1_0 $} (m-4-2)
(m-4-2) edge node[description] {$ j^1_0 $} (m-3-3)
(m-3-3) edge node[left][description] {$ \pi^1_1 $} (m-3-5)
edge node[description] {$ \sigma^1_1 $} (m-4-4)
(m-4-4) edge node[description] {$ j^1_1$} (m-3-5)
(m-3-5) edge node[description] {$ \pi^1_2$} (m-3-7)
(m-1-1)edge node[description]{$ \tau^0_0 $}(m-3-1)
(m-2-2)edge node[above right]{$ k^0_0 $}(m-4-2)
(m-1-3)edge node[description]{$ \tau^0_1$}(m-3-3)
(m-2-4)edge node[above right]{$ k^0_1 $}(m-4-4)
(m-1-5)edge node[description]{$ \tau^0_2$}(m-3-5)
(m-1-8)edge node[description]{$\tau^0_\omega$}(m-3-8)
(m-5-1) edge node[left][description] {$ \pi^2_0 $} (m-5-3)
edge node[description] {$ \sigma^2_0 $} (m-6-2)
(m-6-2) edge node[description] {$ j^2_0 $} (m-5-3)
(m-5-3) edge node[left][description] {$ \pi^2_1 $} (m-5-5)
edge node[description] {$ \sigma^2_1 $} (m-6-4)
(m-6-4) edge node[description] {$ j^2_1$} (m-5-5)
(m-5-5) edge node[description] {$ \pi^2_2$} (m-5-7)
(m-3-1)edge node[description]{$ \tau^1_0 $}(m-5-1)
(m-4-2)edge node[above right]{$ k^1_0 $}(m-6-2)
(m-3-3)edge node[description]{$ \tau^1_1$}(m-5-3)
(m-4-4)edge node[above right]{$ k^1_1 $}(m-6-4)
(m-3-5)edge node[description]{$ \tau^1_2$}(m-5-5)
(m-3-8)edge node[description]{$ \tau^1_\omega $}(m-5-8)
(m-5-1)edge node[description]{$ \tau^2_0$}(m-7-1)
(m-6-2)edge node[description]{$ k^2_0$}(m-8-2)
(m-5-3)edge node[description]{$ \tau^2_1$}(m-7-3)
(m-6-4)edge node[description]{$ k^2_1$}(m-8-4)
(m-5-5)edge node[description]{$ \tau^2_2$}(m-7-5)
(m-5-8)edge node[description]{$ \tau^2_\omega$}(m-7-8)
(m-11-1) edge node[description] {$ \pi^\omega_0 $} (m-11-3)
edge node[description] {$ \sigma^\omega_0 $} (m-12-2)
(m-12-2) edge node[description] {$ j^\omega_0 $} (m-11-3)
(m-11-3) edge node[description] {$ \pi^\omega_1 $} (m-11-5)
edge node[description] {$ \sigma^\omega_1 $} (m-12-4)
(m-12-4) edge node[description] {$ j^\omega_1$} (m-11-5)
(m-11-5) edge node[description] {$ \pi^\omega_2 $}(m-11-7)
;
\end{tikzpicture}
\caption{The process in Theorem \ref{iterability n}}
\label{diagram}
\end{figure}
\end{proof}
\begin{remark}
\rm{The proof of Theorem \ref{iterability n} also shows that if $(\P,\Sigma)$ is $n$-suitable and $(\P,\Sigma,B) \in \mathcal{F}$ and $C \in \mathbb{B}(\P^-,\Sigma)$ then there is a $B$-iterate $\Q$ of $\P$ such that $(\Q,\Sigma_{\Q^-},B\oplus C)\in \mathcal{F}$; in fact, $\Q$ has strong $B\oplus C$-condensation as defined in the proof of Theorem \ref{iterability n}.}
\end{remark}
It is easy to see that ${\mathcal{M}_\infty}|{\theta_\alpha} = V^{\textrm{HOD}}_{\theta_\alpha}$. Let $\langle\eta_i \ | \ i<\omega\rangle$ be the increasing enumeration of Woodin cardinals in $\mathcal{M}_\infty$ larger than $\theta_\alpha$. Theorem \ref{iterability n} is used to show that $\mathcal{M}_\infty$ is large enough in that
\begin{lemma}
\label{Minftylarge}
\begin{enumerate}
\item $\mathcal{M}_\infty$ is well-founded.
\item $\mathcal{M}_\infty|\eta_0 = V^{\textrm{HOD}}_\Theta$. In particular, $\eta_0 = \Theta$.
\end{enumerate}
\end{lemma}
\begin{proof}
We prove (1) and (2) simultaneously. For a similar argument, see Lemma 3.3.2 in \cite{ATHM}. Toward a contradiction, suppose not. By $\Sigma_1$-reflection (Theorem \ref{fundamental result of ad+}), there is a transitive model $N$ coded by a Suslin, co-Suslin set of reals such that $Code(\Sigma) \in \powerset(\mathbb{R})^N$ and
\begin{eqnarray*}
N &\vDash& \textsf{ZF}^- + \textsf{DC} + S\textsf{MC} + ``\Theta \textrm{ exists and is successor in the Solovay sequence }" + \\ && ``\textrm{(1) and (2) do not both hold}".
\end{eqnarray*}
We take a minimal such $N$ and let $\Omega = \powerset(\mathbb{R})^N$. We get $N \vDash V = K^{\Sigma}(\mathbb{R})$ and a $(\Q,\Lambda)$ with the property that for all $B \in \mathbb{B}(P,\Sigma)^N$, there is a $\Lambda$ iterate $\R$ of $\Q$ that strongly respects $B$. $(\Q, \Lambda)$ also has the property that any $\Lambda$ iterate $\R$ of $\Q$ can be further iterated by $\Lambda_\R$ to $\S$ such that $N$ is the derived model of $\S$.
\\
\indent Fix $\langle \alpha_i \ | \ i<\omega\rangle$ a \textrm{cof}inal in $\Theta^\Omega$ sequence of ordinals. Such a sequence exists since $\Omega = Env((\Sigma^2_1)^N)$. For each $n$, let
\begin{eqnarray*}
D_n = \{(\R,\Psi, x, y) \ &|& \ (\R,\Psi) \textrm{ is a hod pair equivalent to } (\P,\Sigma), \ x \textrm{ codes } \R, \\&& y \in \textrm{ the least } OD^N_\Psi \textrm{ set of reals with Wadge rank } \geq \alpha_n\}
\end{eqnarray*}
Clearly, for all $n$, $D_n \in \mathbb{B}(\P,\Sigma)^N$. Without loss of generality, we may assume $\Lambda$ strongly respects all the $D_n$'s and the derived model of $\Q$ is $N$. Let $\vec{D} = \langle D_n \ | \ n < \omega\rangle$. Before proving the next claim, let us introduce the following notion. First let for a set $A (A \subseteq \mathbb{R} \textrm{ or } A \in \mathbb{B}(\P,\Sigma))$, $\tau_{A,m}^{Q,0}$ be the canonical capturing term for $A$ in $\Q$ at $(\delta_0^{+m})^\Q$. Set
\begin{eqnarray*}
&&\gamma^{\Q,0}_{D_i,m} = sup\{H_1^\Q(P \cup \{\tau^{\Q,0}_{D_i,m}\}) \cap \delta_0\};
\\ && \gamma^{\Q,0}_{D_i} = sup_{m<\omega}\gamma^{\Q,0}_{D_i,m}.
\end{eqnarray*}
\textbf{Claim 1.} \textit{ For any $\Lambda$-iterate $(\S, \Upsilon)$ of $\Q$. Suppose $i: \Q \rightarrow \S$ is the itaration map. Then 
\begin{equation*}
i(\delta_0) = sup_{i<\omega}{\gamma_{D_i}^{\S,0}}.
\end{equation*}}
\begin{proof}
Working in $N$, let $\langle A_i \ | \ i < \omega\rangle$ be a sequence of $OD_\Sigma$ sets such that $A_0$ is a universal $\Sigma^2_1(\Sigma)$ set; $A_1 = \mathbb{R}\backslash A_0$; the $\langle A_i \ | \ i \geq 2\rangle$ is a $\Pi^2_1(\Sigma)$-semiscale on $A_1$. Suppose $\phi_i$ and $s_i \in \textrm{OR}^{<\omega}$ are such that
\begin{equation*}
x \in A_i \ iff \ N \vDash \phi_i[\Sigma, s_i, x] 
\end{equation*}
Now for each $i$, let
\begin{eqnarray*}
A_i^* = \{(\R,\Psi, x, y) \ &|& \ (\R,\Psi) \textrm{ is a hod pair equivalent to } (\P,\Sigma), \ x \textrm{ codes } \R, \\&& N \vDash \phi_i[\Psi, s_i, y]\}
\end{eqnarray*}
Aside from the assumption about $(\Q,\Lambda)$ above ,we also assume $\Lambda$ is guided by $\langle A_i \ | \ i < \omega\rangle$ for stacks above $\P$ and below $\delta_0$. This is possible by relativizing to $\Sigma$ the proof of a similar fact in the case $\Theta = \theta_0$. This means
\begin{equation*}
\delta_0 = sup_{i<\omega}\gamma^{\Q,0}_{A^*_i}.
\end{equation*}
This fact in turns implies
\begin{equation*}
\delta_0 = sup_{i<\omega}\gamma^{\Q,0}_{D_i}.
\end{equation*}
To see this, fix an $A^*_i$. We'll show that there is a $j$ such that $\gamma^{\Q,0}_{D_j} \ge \gamma^{\Q,0}_{A^*_i}$. Well, fix a real coding $\P$ and let $j$ be such that 
\begin{equation*}
w(A_i)=w((A^*_i)_{(\P,\Sigma,x)}) \le w((D_j)_{(\P,\Sigma,x)}).
\end{equation*}
Let $z$ be a real witnessing the reduction. Then there is a map $i: \Q \rightarrow \R$ such that
\begin{enumerate}
\item $i$ is according to $\Lambda$ and the iteration is above $\Q^- = \P$;
\item $z$ is generic for the extender algebra $\mathbb{A}$  of $\R$ at $\delta^\R$.
\end{enumerate}
Note that $i(\tau^\Q_{A^*_i}) = \tau^R_{A^*_i}$, $i(\tau^\Q_{D_j}) = \tau^\R_{D_j}$, and $\R[z] \vDash \tau_{A^*_i} \le_w \tau_{D_j}$ via $z$. Hence $\tau^\R_{A^*_i} \in X = \{\tau \in \R^{\mathbb{A}} \ | \ (\exists p\in\mathbb{A})(p \Vdash_R \tau \le_w \tau_{D_j} \textrm{ via } \dot{z})\}$ and $|X|^\R < \delta^\R$ (by the fact that the extender algebra $\mathbb{A}$ is $\delta^\R$-cc). But $X$ is definable over $\R$ from $\tau^R_{D_j}$, hence $|X|^\R < \gamma^{R,0}_{D_j}$. Since $\tau^R_{A^*_i}\in X$, $\gamma^{\R,0}_{A^*_i} \le \gamma^{\R,0}_{D_j}$ which in turns implies  $\gamma^{\Q,0}_{A^*_i} \le \gamma^{\Q,0}_{D_j}$.
\\
\indent Now to finish the claim, let $(\S, \Upsilon)$ be a $\Lambda$ iterate of $\Q$. Suppose $i: \Q \rightarrow \S$ is the iteration map. Let $\R = i(\P)$ and $\Sigma_\Q$ be the tail of $\Sigma$ under the iteration. We claim that
\begin{equation*}
i(\delta_0) = sup_{i<\omega}{\gamma_{D_i}^{S,0}}. (\textasteriskcentered)
\end{equation*}
This is easily seen to finish the proof of Claim 1. To see (\textasteriskcentered), we repeat the proof of the previous part applied to $(\S,\Upsilon)$ and $\langle B_i \ | \ i<\omega\rangle$ where $B_0$ is a universal $\Sigma^2_1(\Sigma_\Q)$; $B_1 = \mathbb{R}\backslash B_0$; $\langle B_i \ | \ i\geq 2\rangle$ is a $\Pi^2_1(\Sigma_\Q)$-semiscale on $B_1$. We may assume $(\S,\Upsilon)$ is guided by $\langle B_i \ | \ i<\omega\rangle$ for stacks above $R$ and below $i(\delta_0)$. Now we are in the position to apply the exact same argument as above and conclude that (\textasteriskcentered) holds. Hence we're done.
\end{proof}	
The proof of claim 1 shows $\Lambda$ has branch condensation, hence the direct limit $\mathcal{M}_\infty(\Q,\Lambda)$ is defined and is wellfounded. This implies that in $N$, $\mathcal{M}_\infty$ is wellfounded. Let $\langle\delta_i \ | \ i<\omega\rangle$ be the first $\omega$ Woodins of $\Q$ above $\Q^-$ and $i^{\Q,\Lambda}_{\Q,\infty}: \Q \rightarrow \mathcal{M}_\infty(\Q,\Lambda)$ be the iteration embedding according to $\Lambda$ and $\langle\eta_n \ | \ n < \omega\rangle$ = $\langle i^{\Q,\Lambda}_{\Q,\infty}(\delta_i) \ | \ i<\omega\rangle$. For $(\R,\Lambda_\R)$ and iterate of $(\Q,\Lambda)$, let $i^{\R,\Lambda_\R}_{\R,\infty}$ have the obvious meaning and $i^{\Q,\Lambda}_{\Q,\R}$ be the iteration map according to $\Lambda$. Note that in $N$, $\M_\infty(\Q,\Lambda)|\eta_n = \M_\infty|\eta_n$ for all $n$.
\\
\\
\textbf{Claim 2.} \textit{ $\eta_0 = \Theta^\Omega$.}
\begin{proof}
Working in $N$, we first claim that 
\begin{equation*}
\mathcal{M}_\infty(\Q,\Lambda)|\eta_0 = V^{\textrm{HOD}}_{\eta_0}. \ \ \ (\textasteriskcentered)
\end{equation*}
To show (\textasteriskcentered), it is enough to show that if $A \subseteq \alpha < \eta_0$ and $A$ is $OD$ then $A \in \mathcal{M}_\infty(\Q,\Lambda)$. To see this, let $i$ be such that $\gamma^{\mathcal{M}_\infty(\Q,\Lambda),0}_{D_i} > \alpha$ (such an $i$ exists by the proof of Claim 1). Let
\begin{eqnarray*}
C = \{(\R,\Psi, x, y) \ &|& \ (\R,\Psi) \textrm{ is a hod pair equivalent to } (\P,\Sigma), \ x \textrm{ codes } \R,\ y \textrm{ codes } (N,\gamma) \\ &&\textrm{ such that } (\N,\Psi) \textrm{ is 1-suitable, } \N \textrm{ is strongly } D_i \textrm{ iterable via a } \\ && \textrm{ quasi-strategy }  \Phi  \textrm{ extending } \Psi,\ \gamma < \gamma^{\mathcal{N},0}_{D_i},\ \pi^{(\N,\Psi),\infty}_{D_i}(\gamma) \in A\}.
\end{eqnarray*}
By replacing $\Q$ by an iterate we may assume $(\Q,\Lambda)$ is $C$-iterable. Let $\tau^{\Q}_C = \tau^{\Q}_{C,(\delta_0^{+\omega})^\Q}$ and $\tau_{C} = i^{(\Q,\Lambda)}_{\Q,\infty}(\tau^{\Q}_C)$. The following equivalence is easily shown by a standard computation:
\begin{eqnarray*}
\xi \in A \ iff \ \mathcal{M}_\infty(\Q,\Lambda) \vDash \Vdash_{Col(\omega,\eta_0^{+\omega})} && ``\textrm{if } x \textrm{ codes } i^{\Q,\Lambda}_{\Q,\infty}(\P), y \textrm{ codes } (\mathcal{M}_\infty(\Q,\Lambda)|\eta_0^{+\omega},\xi) \\ && \textrm{ then } (x,y) \in \tau_C".
\end{eqnarray*}
For the reader's convenience, we'll show why the above equivalence holds. First suppose $\xi \in A$. Let $(\S,\Xi) \in I(\Q,\Lambda)$ be such that there is a $\gamma < \gamma_{D_i}^{\S,0}$ and $i^{\S,\Xi}_{\S,\infty}(\gamma) = \xi$. Then we have (letting $\nu = i^{\Q,\Lambda}_{\Q,S}(\delta_0)$)
\begin{eqnarray*}
\S \vDash \Vdash_{Col(\omega,\nu^{+\omega})} ``\textrm{if } x \textrm{ codes } i^{\Q,\Lambda}_{\Q,\S}(\P), y \textrm{ codes } (S|\nu^{+\omega},\gamma)\textrm{ then } (x,y) \in i^{\Q,\Lambda}_{\Q,\S}(\tau^{\Q,\Lambda}_C)".
\end{eqnarray*}
By applying $i^{\S,\Xi}_{\S,\infty}$ to this ,we get
\begin{eqnarray*}
\mathcal{M}_\infty(\Q,\Lambda) \vDash \Vdash_{Col(\omega,\eta_0^{+\omega})} ``\textrm{if } x \textrm{ codes } i^{\Q,\Lambda}_{\Q,\infty}(\P), y \textrm{ codes } (\mathcal{M}_\infty(\Q,\Lambda)|\eta_0^{+\omega},\xi) \textrm{ then } (x,y) \in \tau_C".
\end{eqnarray*}
Now to show $(\Leftarrow)$, let $(\S,\Xi) \in I(\Q,\Lambda)$ be such that for some $\gamma < \gamma_{D_i}^{\S,0}$, $\xi = i^{\S,\Xi}_{\S,\infty}(\gamma)$. Let $\nu = i^{\Q,\Lambda}_{\Q,\S}(\delta_0)$, we have
\begin{eqnarray*}
\S \vDash \Vdash_{Col(\omega,\nu^{+\omega})} ``\textrm{if } x \textrm{ codes } i^{\Q,\Lambda}_{\Q,\S}(\P), y \textrm{ codes } (\S|\nu^{+\omega},\gamma)\textrm{ then } (x,y) \in i^{\Q,\Lambda}_{\Q,\S}(\tau^{\Q,\Lambda}_C)".
\end{eqnarray*}
This means there is a quasi-strategy $\Psi$ on $\S(0)$ ($\S(0) = \S|({\nu^{+\omega}})^\S$) such that $(\S(0),i^{\Q,\Lambda}_{\Q,\S}(\Sigma))$ is $1$-suitable, $\Psi$ extends $i^{\Q,\Lambda}_{\Q,\S}(\Sigma))$, and $\Psi$ is $D_i$-iterable. We need to see that $\pi^{(\S(0),i^{\Q,\Lambda}_{\Q,\S}(\Sigma)),\infty}_{D_i}(\gamma) = \xi$. But this is true by the choice of $D_i$, $\xi = i^{\S,\Xi}_{S,\infty}(\gamma)$, and the fact that $\Psi$ agrees with $\Xi$ on how ordinals below $\gamma^{\S,0}_{D_i}$ are mapped.
\\
\indent The equivalence above shows $A \in \mathcal{M}_\infty(Q,\Lambda)$, hence completes the proof of (\textasteriskcentered). (\textasteriskcentered) in turns shows that $\eta_0$ is a cardinal in $\textrm{HOD}$ and $\eta_0 \leq \Theta$ (otherwise, $\textrm{HOD}|\eta_0 = \mathcal{M}_\infty(Q,\Lambda)|\eta_0 \vDash \Theta$ is not Woodin while $\textrm{HOD} \vDash \Theta$ is Woodin).
\\
\indent Next, we show expectedly that
\begin{equation*}
\eta_0 = \Theta. \ \ \ (\textasteriskcentered\textasteriskcentered)
\end{equation*}
Suppose toward a contradiction that $\eta_0 < \Theta$. Let $\Q(0) = \Q|(\delta_0^{+\omega})^\Q$, $\Lambda_0 = \Lambda|\Q(0)$, and $\mathcal{M}_\infty(\Q,\Lambda)(0) = \mathcal{M}_\infty(\Q,\Lambda)|(\eta_0^{+\omega})^{\mathcal{M}_\infty(\Q,\Lambda)}$. Let $\pi = i\rest\Q(0)$; so $\pi$ is according to $\Lambda_0$. By the Coding Lemma and our assumption that $\eta_0 < \Theta$, $\pi, \mathcal{M}_\infty(Q,\Lambda)(0) \in N$. From this, we can show $\Lambda_0 \in N$ by the following computation: $\Lambda_0(\vec{\mathcal{T}}) = b$ if and only if
\begin{enumerate}
\item the part of $\vec{\mathcal{T}}$ based on $P$ is according to $\Sigma$;
\item if $i^{\vec{\mathcal{T}}}_b$ exists then there is a $\sigma: \mathcal{M}^{\vec{\mathcal{T}}}_b \rightarrow\mathcal{M}_\infty(Q,\Lambda)(0)$ such that $\pi = \sigma\circ i^{\vec{\mathcal{T}}}_b$;
\item $\vec{\mathcal{T}}^{\smallfrown}\mathcal{M}^{\vec{\mathcal{T}}}_b$ is $Q$-structure guided.  
\end{enumerate}
By branch condensation of $\Lambda_0$, (1),(2), and (3) indeed define $\Lambda_0$ in $N$. This means $\Lambda_0$ is $OD^N$ from $\Sigma$ (and some real $x$); hence $\Lambda_0 \in N$. So suppose $\gamma = w(Code(\Lambda_0)) < \Theta^\Omega$. In N, let
\begin{eqnarray*}
B = \{(\R,\Psi, x, y) \ &|& \ (\R,\Psi) \textrm{ is a hod pair equivalent to } (\P,\Sigma), \ x \textrm{ codes } \R,\ y \in A_\R \\ &&\textrm{ where } A_\R \textrm{ is the least } OD(Code(\Psi)) \textrm{ set such that } w(A_\R) > \gamma\}
\end{eqnarray*}
Then $B \in \mathbb{B}(\P,\Sigma)^N$. We may assume $\Lambda_0$ respects $B$. It is then easy to see that whenever $(\R,\Lambda_\R) \in I(Q(0),\Lambda_0)$ (also let $\S \triangleleft \R$ be the iterate of $\P$), $w(Code(\Lambda_\R)) \geq w(A_\R)$ because $\Lambda_\R$ can compute membership of $A_\R$ by performing genericity iterations (above $\S$) to make reals generic. This means $w(Code(\Lambda_\R)) > \gamma = w(Code(\Lambda_0))$. This contradicts the fact that $w(Code(\Lambda_\R)) = w(Code(\Lambda_0))$.
\end{proof}
The proof of Claim 1 and Claim 2 shows that the fragment of $\Lambda$ on stacks above $P$ and below $\delta_0$ is guided by $\langle D_i \ | \ i<\omega\rangle$.
\end{proof}
\indent Now we define a strategy $\Sigma_\infty$ for $\mathcal{M}_\infty$ extending the strategy $\Sigma_\infty^-$ of $\mathcal{M}_\infty^- = V^{\textrm{HOD}}_{\theta_\alpha}$. Let $(\P,\Sigma,A) \in \mathcal{F}$ and suppose $\P$ is $n$-suitable with $\langle\delta_i \ | \ i<n\rangle$ being the sequence of Woodins of $\P$ above $\P^-$, let $\tau^{\mathcal{M}_\infty}_{A,k}$ = common value of $\pi^{\P,\Sigma}_{\vec{B},\infty}(\tau^\P_{A,\delta_k})$. $\Sigma_\infty$ will be defined (in V) for trees on $\mathcal{M}_\infty|\eta_0$ in $\mathcal{M}_\infty$. For $k \geq n$, $\mathcal{M}_\infty \vDash "Col(\omega,\eta_n) \times Col(\omega,\eta_k) \Vdash (\tau^{\mathcal{M}_\infty}_{A,n})_g =  (\tau^{M_\infty}_{A,k})_h \cap \mathcal{M}_\infty[g]$" where $g$ is $Col(\omega,\eta_n)$ generic and h is $Col(\omega,\eta_k)$ generic and $(\tau^{\mathcal{M}_\infty}_{A,n})_g$ is understood to be $A_{(\mathcal{M}_\infty^-,\Sigma_\infty^-)}\cap \mathcal{M}_\infty[g]$. This is just saying that the terms cohere with one another. 
\\
\indent Let $\lambda^{\mathcal{M}_\infty} = sup_{i<\omega}\eta_i$. Let $G$ be $Col(\omega, \lambda^{\mathcal{M}_\infty})$ generic over $\mathcal{M}_\infty$. Then $\mathbb{R}^*_G$ is the symmetric reals and $A^*_G$ := $\cup_k(\tau^{\mathcal{M}_\infty}_{A,k})_{G|\eta_k}$. 
\begin{proposition}
For all $A\in \mathbb{B}(\mathcal{M}_\infty^-,\Sigma_\infty^-)$, $L(A^*_G, \mathbb{R}^*_G) \vDash \textsf{AD}^+$
\end{proposition}
\begin{proof}
We briefly sketch the proof of this since the techniques involved have been fully spelled out in Section 2. If not, reflect the situation down to a model $N$ coded by a Suslin co-Suslin set. Next get a ``next mouse" $\mathcal{N}$ with $\omega$ Woodin cardinals that iterates out to (possibly a longer mouse than) $\mathcal{M}_\infty^N$. We can and do assume that an iterate of $\mathcal{N}$ has derived model is $K^{\Sigma}(\mathbb{R})$ of $N$, where $(\P,\Sigma)$ is a hod pair giving us $\textrm{HOD}|\theta_\alpha$; the proof of this fact uses a relative-to-$\Sigma$ Prikry forcing (see \cite{steel08}) and S-constructions (see \cite{ATHM}). From now on, we work inside the reflected universe $N$.
\\
\indent Let $A \subseteq \mathbb{B}(\mathcal{M}_\infty^-,\Sigma_\infty^-)$ be the least $OD$ set such that $L(A^*_G, \mathbb{R}^*_G) \nvDash \textsf{AD}^+$. Then there is an iterate $\mathcal{M}$ of $\mathcal{N}$ having preimages of all the terms $\tau^{M_\infty}_{A,k}$. We may assume $\mathcal{M}$ has derived model $K^{\Sigma}(\mathbb{R})$. Since we have $\textsf{AD}^+$, $\mathcal{M}$ thinks that its derived model (in this case is $K^{\Sigma}(\mathbb{R})$) satisfies that $L(A_{(\P,\Sigma)},\mathbb{R}) \vDash \textsf{AD}^+$, where we reuse $(\P,\Sigma)$ for an equivalent (but possibly different) hod pair from the original one. Now iterate $\mathcal{M}$ to $\mathcal{Q}$ such that $M_\infty$ is a proper initial segment of $\mathcal{Q}$. By elementarity $L(A^*_G,\mathbb{R}^*_G) \vDash \textsf{AD}^+$. This is a contradiction.
\end{proof}
\begin{definition}
\label{Sigma infinity n}
Given a normal tree $\mathcal{T} \in \mathcal{M}_\infty$ and $\mathcal{T}$ is based on $\mathcal{M}_\infty|\theta_0$. $\mathcal{T}$ is by $\Sigma_\infty$ if the following hold (the definition is similar for finite stacks):
\begin{itemize}
\item If $\mathcal{T}$ is short then $\Sigma$ picks the branch guided by $Q$-structure (as computed in $\M_\infty)$.
\item If $\mathcal{T}$ is maximal then $\Sigma_\infty(\mathcal{T})$ = the unique \textrm{cof}inal branch $b$ which moves $\tau^{\mathcal{M}_\infty}_{A,0}$ correctly for all $A \in OD$ such that there is some $(\P,\Sigma,A)\in\mathcal{F}$ i.e. for each such $A$, $i_b(\tau^{\M_\infty}_{A,0}) = \tau^{\M^{\mathcal{T}}_b}_{A^*,0}$.
\end{itemize}
\end{definition}
\begin{lemma} Given any such $\mathcal{T}$ as above, $\Sigma_\infty(\mathcal{T})$ exists.
\end{lemma}
\begin{proof} Suppose not. Again reflect the failure to a model $N$ coded by a Suslin co-Suslin set. We may assume $N \vDash V = K^{\Sigma}(\mathbb{R})$ where $(\P,\Sigma)$ is a hod pair giving us $\textrm{HOD}|\theta_\alpha$. Just as in the previous proposition, we then get a next mouse $\mathcal{N}$ that iterates out to (possibly a longer mouse than) $\mathcal{M}^N_\infty$. This mouse $\mathcal{N}$ has strategy $\Lambda$ with property that for all $A \in \mathbb{B}(P,\Sigma)$, there is a $\Lambda$-iterate $(\mathcal{M}, \Lambda_{\mathcal{M}})$ of $\mathcal{N}$ such that $\Lambda_{\mathcal{M}}$ strongly respects $A$. This easily gives us a contradiction.  
\end{proof}
\indent It is evident that $L(\mathcal{M}_\infty,\Sigma_\infty) \subseteq \textrm{HOD}$. Next, we show $\M_\infty$ and $\Sigma_\infty$ capture all unbounded subsets of $\Theta$ in $\textrm{HOD}$. In $L(\M_\infty, \Sigma_\infty)$, first construct (using $\Sigma_\infty$) a mouse $\M_\infty^+$ extending $\M_\infty$ such that o($\M_\infty$) is the largest cardinal of $\mathcal{M}_\infty^+$ as follows: 
\begin{enumerate}
\item Let $\mathbb{R}^*_G$ be the symmetric reals obtained from a generic $G$ over $\mathcal{M}_\infty$ of Col($\omega, <\lambda^{\mathcal{M}_\infty}$).
\item For each $A^*_G$ (defined as above) (we know $L(\mathbb{R}^*_G,A^*_G) \vDash \textsf{AD}^+$), pull back the hybrid mice over $\mathbb{R}^*_G$ in this model to hybrid mice $\S$ extending $\mathcal{M}_\infty$ with $D^+(S, \lambda^{M_\infty}) = L(\mathbb{R}^*_G, A^*_G)$.
\item Let $\mathcal{M}_\infty^+ = \cup_\S \S$ for all such $\S$ above. $\mathcal{M}_\infty^+$ is independent of $G$. By a reflection argument (and Prikry-like forcing) as above, the translated mice over $\mathcal{M}_\infty$ are all compatible, no levels of $\mathcal{M}_\infty^+$ projects across $o(\mathcal{M}_\infty)$, and $\M^+_\infty$ contains as its initial segments all translation of $\mathbb{R}^*_G$-mice in $D^+(\mathcal{M}_\infty^+, \lambda^{\mathcal{M}_\infty})$. This is just saying that $\mathcal{M}_\infty^+$ contains enough mice to compute $\textrm{HOD}$.  
\end{enumerate}
\begin{remark} \rm{$\Theta$ is not collapsed by $\Sigma_\infty$ as it is a cardinal in \textrm{HOD}. $\Sigma_\infty$ is used to obtain the $A^*_G$ above by moving correctly the $\tau^{M_\infty}_{A,0}$ in genericity iterations. $L(\mathcal{M}_\infty)$ does not see the sequence $\langle \tau^{\mathcal{M}_\infty}_{A,k} | k \in \omega \rangle$ hence can't construct $A^*_G$. Also since $\Sigma_\infty$ collapses $\delta^{M_\infty}_1, \delta^{\mathcal{M}_\infty}_2...$, it doesn't make sense to talk about $D(L(\mathcal{M}_\infty,\Sigma_\infty))$}.
\end{remark}
\begin{lemma} 
\label{keylemma1}
\textrm{HOD} $\subseteq L(\mathcal{M}_\infty, \Sigma_\infty)$
\end{lemma}
\begin{proof}  Using \rthm{WoodinVopenka}, we know $\textrm{HOD} = L[P]$ for some $P \subseteq \Theta$. Therefore, it is enough to show $P \in L(\mathcal{M}_\infty, \Sigma_\infty)$. Let $\phi$ be a formula defining $P$, i.e.
\begin{equation*}
\alpha \in P \Leftrightarrow V \vDash \phi[\alpha].
\end{equation*}
We suppress the ordinal parameter here. Now in $L(\mathcal{M}_\infty, \Sigma_\infty)$ let $\pi : \mathcal{M}_\infty | (\eta_0^{++})^{\mathcal{M}_\infty} \rightarrow (\mathcal{M}_\infty)^{D^+(\mathcal{M}_\infty^+,\lambda^{\mathcal{M}_\infty})}$ where $\pi$ is according to $\Sigma_\infty$.
\\
\\
$\textbf{Claim:}$  $\alpha \in P \Leftrightarrow D(\mathcal{M}_\infty^+, \lambda^{M_\infty}) \vDash \phi[\pi(\alpha)]$. \ \  (\textasteriskcentered)
\begin{proof}  Otherwise, reflect the failure of (\textasteriskcentered) as before to get a model $N$ coded by a Suslin co-Suslin set, a hod pair $(\P,\Sigma)$ giving us $\textrm{HOD}|\theta_\alpha$ such that  
\begin{eqnarray*}
N \vDash \textsf{ZF} + \textsf{DC} + \textsf{AD}^+ +  V = K^{\Sigma}(\mathbb{R}) + (\exists \alpha)(\phi[\alpha] \nLeftrightarrow D^+(\mathcal{M}_\infty^+,\Sigma_\infty) \vDash \phi[\pi(\alpha)]).
\end{eqnarray*}
Fix such an $\alpha$. As before, let $\mathcal{N}$ be the next mouse (i.e. $\N$ has $\omega$ Woodins $\langle\delta_i \ | \ i<\omega\rangle$ on top of $\P$) with $\rho(\N)<sup_{i}\delta_i$) with strategy $\Lambda$ extending $\Sigma$ and $\Lambda$ has branch condensation and is $\Omega$-fullness preserving, where $\Omega = (\Sigma^2_1)^N$. We may assume $\Lambda$ is guided by $\vec{D}$ where $\vec{D} = \langle D_n \ | \ n<\omega\rangle$ is defined as in Lemma \ref{Minftylarge}. As before, we may assume $\N$ can realize $N$ as its new derived model. Let $\sigma : N|((\delta^N_0)^{++})^N \rightarrow (M_\infty)^{D^+(N,\lambda^N)}$ be the direct limit map by $\Lambda$. We may assume $\sigma(\overline{\alpha}) = \alpha$ for some $\overline{\alpha}$. Working in $N$, it then remains to see that:
\begin{equation*}
D^+(\mathcal{M}_\infty^+,\lambda^{\mathcal{M}_\infty})\vDash\phi[\pi(\alpha)] \Leftrightarrow D^+(\N,\lambda^\N) \vDash \phi[\sigma(\overline{\alpha})] \ \ (**).
\end{equation*} 
To see that (\textasteriskcentered \textasteriskcentered) holds, we need to see that the fragment of $\Lambda$ that defines $\sigma(\overline{\alpha})$ can be defined in $D^+(\N, \lambda^\N)$. This then will give the equivalence in (\textasteriskcentered \textasteriskcentered). Because $\alpha < \eta_0$,  $\overline{\alpha} < \delta_0$, pick an $n$ such that such that $\gamma_{D_n,0}^{\N,0} > \overline{\alpha}$. Then the fragment of $\Lambda$ that defines $\sigma(\overline{\alpha})$ is definable from $D_n$ (and $\N|(\delta_0^\N)$) in $D^+(\N, \lambda^\N)$, which is what we want. 

The equivalence (\textasteriskcentered \textasteriskcentered) gives us a contradiction.
\end{proof}
The claim finishes the proof of $P \in L(\mathcal{M}_\infty, \Sigma_\infty)$. This then implies $\textrm{HOD} = L[P] \subseteq L(\mathcal{M}_\infty, \Sigma_\infty)$.
\end{proof}
\indent Lemma \ref{keylemma1} implies $\textrm{HOD} = L(\mathcal{M}_\infty, \Sigma_\infty)$, hence completes our computation. 

As mentioned above, aside from assuming $(\textasteriskcentered)$, we also assume $\Sigma$ is such that $K^\Sigma(\mathbb{R})$ is defined. That obviously leaves open whether the HOD computation can be carried out with simply assuming (\textasteriskcentered).

\section{The Limit Case}
There are two cases: the easier case is when $\textrm{HOD} \vDash ``cof(\Theta) \textrm{ is not measurable}"$, and the harder case is when $\textrm{HOD} \vDash ``cof(\Theta) \textrm{ is measurable}"$. 
\\
\indent Here's the direct limit system that gives us $V^{\textrm{HOD}}_\Theta$.
\begin{equation*}
\mathcal{F} = \{(\mathcal{Q},\Lambda) \ | \ (\mathcal{Q},\Lambda) \textrm{ is a hod pair; }\Lambda \textrm{ is fullness preserving and has branch condensation}\}.
\end{equation*} 
The order on $\mathcal{F}$ is given by
\begin{equation*}
(\mathcal{Q},\Lambda) \leq^{\mathcal{F}} (\mathcal{R},\Psi)\ \Leftrightarrow \ \mathcal{Q} \textrm{ iterates to a hod initial segment of } \mathcal{R}.
\end{equation*}
By Theorem \ref{branch condensation's consequences}, $\leq^{\mathcal{F}}$ is directed and we can form the direct limit of $\mathcal{F}$ under the natural embeddings coming from the comparison process. Let $\mathcal{M}_\infty$ be the direct limit. By the computation in \cite{ATHM}, 
\begin{equation*}
|\mathcal{M}_\infty| = V^{\textrm{HOD}}_\Theta.
\end{equation*}
$\M_\infty$ as a structure also has a predicate for its extender sequence and a predicate for a sequence of strategies.

We quote a theorem from \cite{ATHM} which will be used in the upcoming computation. For unexplained notations, see \cite{ATHM}.
\begin{theorem}[Sargsyan, Theorem 4.2.23 in \cite{ATHM}]
\label{well behaved hod mice}
Suppose $(\mathcal{P},\Sigma)$ is a hod pair such that $\Sigma$ has branch condensation and is fullness preserving. There is then $\mathcal{Q}$ a $\Sigma$-iterate of $\mathcal{P}$ such that whenever $\mathcal{R}$ is a $\Sigma_{\mathcal{Q}}$-iterate of $\mathcal{Q}$, $\alpha < \lambda^{\mathcal{R}}$, and $B \in (\mathbb{B}(\mathcal{R}(\alpha), \Sigma_{\mathcal{R}(\alpha)}))^{L(\Gamma(\mathcal{R}(\alpha+1),\Sigma_{\mathcal{R}(\alpha+1)}))}$
\begin{enumerate}
\item $\Sigma_{\mathcal{R}(\alpha+1)}$ is super fullness preserving and is strongly guided by some
\begin{equation*}
\vec{B} = \langle B_i \ | \ i<\omega\rangle \subseteq (\mathbb{B}(\mathcal{R}(\alpha), \Sigma_{\mathcal{R}(\alpha)}))^{L(\Gamma(\mathcal{R}(\alpha+1),\Sigma_{\mathcal{R}(\alpha+1)}))};
\end{equation*}
\item there is a $(\mathcal{S},\Sigma_\S) \in I(\mathcal{R}(\alpha+1), \Sigma_{\mathcal{R}(\alpha+1)})$ such that $\Sigma_{\mathcal{S}}$ has strong $B$-condensation.  
\end{enumerate}
\end{theorem}
We deal with the easy case first.

\subsection{Nonmeasurable Cofinality}
The following theorem is the full $\textrm{HOD}$ computation in this case.
\begin{theorem}
\label{easylimit}
$\textrm{HOD} = L(\M_\infty)$
\end{theorem}
\begin{proof}
\indent To prove the theorem, suppose the equality is false. Then by Theorem \ref{WoodinVopenka}, there is an $A\subseteq \Theta$ such that $A \in \textrm{HOD} \backslash L(\M_\infty)$ (the fact that $L(\M_\infty) \subseteq \textrm{HOD}$ is obvious). By $\Sigma_1$-reflection (i.e. Theorem \ref{fundamental result of ad+}), there is a transitive $N$ coded by a Suslin co-Suslin set such that
\begin{eqnarray*}
N &\vDash& \textsf{ZF}^- + \textsf{AD}^+ + V = L(\powerset(\mathbb{R}))+ \textsf{SMC} + \Theta \textrm{ exists and is limit in the Solovay sequence} \\ && + \textrm{HOD} \vDash ``cof(\Theta) \textrm{ is not measurable }"+``\exists B \subseteq \Theta(B \in \textrm{HOD} \backslash L(\M_\infty))".
\end{eqnarray*}
Take $N$ to be the minimal such and let $B$ witness the failure of the theorem in $N$. Let $\phi$ define $B$ (for simplicity, we suppress the ordinal parameter) i.e.
\begin{equation*}
\alpha \in B \ \Leftrightarrow \ N\vDash \phi[\alpha]
\end{equation*}
Let $\Omega = \powerset(\mathbb{R})^N$. There is a pair $(\P,\Sigma)$ such that:
\begin{enumerate}
\item $\P = L_\beta(\cup_{\gamma < \lambda^P}P_\gamma)$ for some $\lambda^P$;
\item for all $\gamma < \lambda^\P$, $\P_\beta$ is a hod mouse whose strategy $\Sigma_\gamma \in \Omega$ is $\Omega$-fullness preserving, has branch condensation, and $\lambda^{\P_\beta} = \beta$;
\item if $\gamma < \eta < \lambda^\P$, $\P_\gamma \unlhd_{hod} \P_\eta$;
\item $\beta$ is least such that $\rho_\omega(L_\beta(\cup_{\gamma < \lambda^\P}\P_\gamma)) < o(\cup_{\gamma < \lambda^\P}\P_\gamma))$;
\item $\P \vDash cof(\lambda^\P)$ is not measurable; 
\item $\Sigma$ has branch condensation and extends $\varoplus_{\gamma<\lambda^\P}{\Sigma_\gamma}$;
\end{enumerate}
Such a $(\P,\Sigma)$ can be obtained by performing a $\Omega$-hod pair construction (see Definition \ref{gamma hod pair construction}) inside some $N^*_x$ capturing a good pointclass beyond $\Omega$. We may and do assume that $(\cup_{\gamma < \lambda^{\mathcal{P}}}\P_\gamma, \varoplus_{\gamma<\lambda^\P}{\Sigma_\gamma})$ satisfies Theorem \ref{well behaved hod mice} applied in $N$. This implies that the direct limit $\mathcal{M}_\infty^+$ of all $\Sigma$-iterates of $\P$ is a subset of $\textrm{HOD}^N$. Let $j: \P \rightarrow \mathcal{M}_\infty^+$ be the natural map. Then in $N$, $\mathcal{M}_\infty^+|j(\lambda^\P) = \mathcal{M}_\infty$.
\\
\indent Now pick a sequence $\langle \gamma_i \ | \ i < \omega\rangle$ cofinal in $\lambda^\P$ such that $\delta_{\lambda^{\P_{\gamma_i}}}$ is Woodin in $\P$, an enumeration $\langle x_i \ | \ i < \omega\rangle$ of $\mathbb{R}$ and do a genericity iteration of $\P$ to successively make each $x_i$ generic at appropriate image of $\delta_{\lambda^{\P_{\gamma_i}}}$. Let $\Q$ be the end model of this process and $i: \P \rightarrow \Q$ be the iteration embedding. Then by assumption (5) above, we have that $N$ is the derived model of $\Q$ at $i(\lambda^P)$.
\\
\indent In $N$, let $D$ be the derived model of $\mathcal{M}^+_\infty$ at $\Theta$ and 
\begin{equation*}
\pi_\infty: \mathcal{M}_\infty \rightarrow (\mathcal{M}_\infty)^D
\end{equation*}
be the direct limit embedding given by the join of the strategies of $\mathcal{M}_\infty$'s hod initial segments. Then by the same argument as that given in Lemma \ref{keylemma}, we have 
\begin{equation*}
\alpha \in B \ \Leftrightarrow \ D \vDash \phi[\pi_\infty(\alpha)].
\end{equation*}
The proof of Lemma \ref{keylemma} also gives that $B \in (L(\mathcal{M}_\infty))^N$, which contradicts our assumption. Hence we're done.
\end{proof}

\begin{remark}
It's not clear that in the statement of Theorem \ref{easylimit}, ``$\M_\infty$" can be replaced by ``$V^{\textrm{HOD}}_\Theta$".
\end{remark}

\subsection{Measurable Cofinality}
Suppose $\textrm{HOD} \vDash cof(\Theta)$ is measurable. We know by \cite{ATHM} that $V^{\textrm{HOD}}_\Theta$ is $|\mathcal{N}_\infty|$ where $\mathcal{N}_\infty$ is the direct limit (under the natural maps) of $\mathcal{F}$, where $\mathcal{F}$ is introduced at the beginning of this section. Let 
\begin{equation*}
\M_\infty = Ult_0(\textrm{HOD},\mu)|\Theta,
\end{equation*}
where $\mu$ is the order zero measure on $cof^{\textrm{HOD}}(\Theta)$. Let $f:\textrm{cof}^{\textrm{HOD}}(\Theta)=_{def}\alpha \rightarrow \Theta$ be a continuous and cofinal function in HOD. For notational simplicity, for each $\beta < \alpha$, let $\Lambda_\beta$ be the strategy of $\M_\infty(f(\beta))$ and $\Sigma_\beta$ be the strategy of $\N_\infty(f(\beta))$. Let
\begin{equation*}
\M_\infty^+ = Ult_0(\textrm{HOD},\mu)|(\Theta^+)^{Ult_0(\textrm{HOD},\mu)},
\end{equation*}
and
\begin{equation*}
\N_\infty^+ = \cup\{\M \ | \ \N_\infty \trianglelefteq \M, \ \rho(\mathcal{M})=\Theta, \mathcal{M} \textrm{ is a hybrid mouse satisfying property (\textasteriskcentered)}\}.
\end{equation*}
Here a mouse $\M$ satisfies property (\textasteriskcentered) if whenever $\pi: \M^* \rightarrow \M$ is elementary, $\M^*$ is countable, transitive, and $\pi(\Theta^*) = \Theta$, then $\M^*$ is a $\oplus_{\xi<\Theta^*}\Sigma_\xi^*$-mouse for stacks above $\Theta^*$, where $\Sigma_\xi^*$ is the strategy for the hod mouse $\M^*(\xi)$ obtained by the following process: let $(\P,\Sigma)\in \mathcal{F}$ and $i:\P\rightarrow \M_\infty$ be the direct limit embedding such that the range of $i$ contains the range of $\pi\rest \M^*(\xi)$; $\Sigma_\xi^*$ is then defined to be the $\pi\circ i^{-1}$-pullback of $\Sigma$. It's easy to see that the strategy $\Sigma_\xi^*$ as defined doesn't depend on the choice of $(\P,\Sigma)$. This is because if $(\P_0,\Sigma_0,i_0)$ and $(\P_1,\Sigma_1,i_1)$ are two possible choices to define $\Sigma_\xi^*$, we can coiterate $(\P_0,\Sigma_0)$ against $(\P_1,\Sigma_1)$ to a pair $(\R,\Lambda)$ and let $i_i:\P_i\rightarrow \R$ be the iteration maps and let $i_2:\R\rightarrow \M_\infty$ be the direct limit embedding. Then $\Sigma_0 = \Lambda^{i_0}$ and $\Sigma_1 = \Lambda^{i_1}$; hence the $\pi\circ i_0^{-1}$-pullback of $\Sigma_0$ is the same as the $\pi\circ i_1^{-1}$-pullback of $\Sigma_1$ because both are the same as the $\pi\circ i_2^{-1}$-pullback of $\Lambda$.

 We give two characterizations of $\textrm{HOD}$ here: one in terms of $\M_\infty^+$ and the other in terms of $\N_\infty^+$. The first one is easier to see.
\begin{theorem}
\label{SingularMeasurable}
\begin{enumerate}
\item $\textrm{HOD} = L(\N_\infty,\M_\infty^+)$.
\item $\textrm{HOD} = L(\N_\infty^+)$.
\end{enumerate}
\end{theorem}
\begin{proof}
To prove (1), first let $j_\mu: \textrm{HOD} \rightarrow Ult_0(\textrm{HOD},\mu)$ be the canonical ultrapower map. Let $A \in \textrm{HOD}$, $A \subseteq \Theta$. By the computation of \textrm{HOD} below $\Theta$, we know that for each limit $\beta < \alpha$,
\begin{equation*}
A \cap \theta_{f(\beta)} \in |\N_\infty(f(\beta))|.
\end{equation*}

This means
\begin{center}
$j_\mu(A)\cap \Theta \in \M_\infty^+$.
\end{center}
We then have
\begin{eqnarray*}
\gamma \in A \Leftrightarrow j_\mu(\gamma) \in j_\mu(A)\cap\Theta.
\end{eqnarray*}
Since $j_\mu|\Theta$ agrees with the canonical ultrapower map $k: \N_\infty \rightarrow Ult_0(\N_\infty,\mu)$ on all ordinals less than $\Theta$, the above equivalence shows that $A \in L(\N_\infty,\M_\infty^+)$. This proves (1).
\\
\indent Suppose the statement of (2) is false. There is an $A \subseteq \Theta$ such that $A \in \textrm{HOD} \backslash \mathcal{N}_\infty^+$. By $\Sigma_1$-reflection (i.e. Theorem \ref{fundamental result of ad+}), there is a transitive $N$ coded by a Suslin co-Suslin set such that
\begin{eqnarray*}
N &\vDash& \textsf{ZF}^- + \textsf{DC} + V = L(\powerset(\mathbb{R}))+ \textsf{SMC} + ``\Theta \textrm{ exists and is limit in the Solovay sequence }" \\ && + ``\textrm{HOD} \vDash cof(\Theta) = \alpha \textrm{ is measurable as witnessed by }f" \\ &&+``\exists A \subseteq \Theta(A \in \textrm{HOD} \backslash \mathcal{N}_\infty^+)".
\end{eqnarray*}
Take $N$ to be the minimal such and let $A$ witness the failure of (2) in $N$. Let $\mu$, $j_\mu$, $\M_\infty$, $\M_\infty^+$, $\N_\infty$, $\N_\infty^+$ be as above but relativized to $N$. Working in $N$, there is a sequence $\langle\mathcal{M}_\beta \ | \ \beta< \alpha, \beta \textrm{ is limit}\rangle \in \textrm{HOD}$ such that for each limit $\beta < \alpha$, 
$\mathcal{M}_\beta$ is the least hod initial segment of $\N_\infty|\theta_{f(\beta)}$ such that $A\cap \theta_{f(\beta)}$ is definable over $\M_\beta$.

Let $\Omega = \powerset(\mathbb{R})^N$. Fix an $N^*_x$ capturing a good pointclass beyond $\Omega$. Now, we again do the $\Omega$-hod pair construction in $N^*_x$ to obtain a pair $(\Q,\Lambda)$ such that
\begin{enumerate}
\item there is a limit ordinal $\lambda^\Q$ such that for all $\gamma < \lambda^\Q$, $\Q_\beta$ is a hod mouse with $\lambda^{\Q_\beta} = \beta$ and whose strategy $\Psi_\gamma \in \Omega$ is $\Omega$-fullness preserving, has branch condensation;
\item if $\gamma < \eta < \lambda^\Q$, $\Q_\gamma \unlhd_{hod} \Q_\eta$; 
\item $\Q$ is the first sound mouse from the $L[E, \varoplus_{\gamma<\lambda^\Q}\Psi_\gamma][\cup_{\gamma < \lambda^\Q} \Q_\gamma]$-construction done in $N^*_x$ that has projectum $\leq o(\cup_{\gamma<\lambda^\Q}\Q_\gamma)$ and extends $Lp^{\Omega,\varoplus_{\gamma<\lambda^\Q}\Psi_\gamma}(\cup_{\gamma<\lambda^\Q}\Q_\gamma)$
\footnote{If $\M \lhd Lp^{\Omega,\varoplus_{\gamma<\lambda^\Q}\Psi_\gamma}(\cup_{\gamma<\lambda^\Q}\Q_\gamma)$ and $\M$ extends $\cup_{\gamma<\lambda^\Q}\Q_\gamma$ then $\M$ is a mouse in $N$ in the sense that $N$ knows how to iterate $\M$ for stacks above $o(\cup_{\gamma<\lambda^\Q}\Q_\gamma)$.} and $\Lambda$ be the induced strategy of $\Q$.
\end{enumerate}
From the construction of $\Q$ and the properties of $N$, it's easy to verify the following:
\begin{enumerate}
\item Let $\delta_{\lambda^\Q} = o(\cup_{\gamma<\lambda^\Q}\Q_\gamma)$ and $\eta = o(Lp^{\Omega,\varoplus_{\gamma<\lambda^\Q}\Psi_\gamma}(\cup_{\gamma<\lambda^\Q}\Q_\gamma))$. Then $\eta=(\delta_{\lambda^\Q}^+)^\Q$.
\item $\Lambda\notin \Omega$.
\item $\Q\vDash \delta_{\lambda^\Q}$ has measurable cofinality. 
\end{enumerate}
Let $\mathcal{M}_\infty(\Q,\Lambda)$ be the direct limit (under natural embeddings) of $\Lambda$-iterates of $\Q$. 
\begin{lemma}
\label{direct limit exists}
$\mathcal{M}_\infty(\Q,\Lambda)$ exists. 
\end{lemma}
\begin{proof}
\indent First note that $\Lambda$ is $\Omega$-fullness preserving. To see this, suppose not. Let $k: \Q \rightarrow \R$ be according to $\Lambda$ witnessing this. It's easy to see that the tail $\Lambda_\R$ of $\Lambda$ acting on $\R|k(\eta)$ is not in $\Omega$ (otherwise, $\Lambda_\R^k = \Lambda$ by hull condensation and hence $\Lambda \in \Omega$. Contradiction.)  However, $\varoplus_{\gamma<\lambda^\R}\Psi_{\R(\gamma)} \in \Omega$ since the iterate of $N^*_x$ by the lift-up of $k$ thinks that the fragment of its strategy inducing $\varoplus_{\gamma<\lambda^\R}\Psi_{\R(\gamma)}$ is in $\Omega$. Now suppose $\M$ is a $\varoplus_{\gamma<\lambda^\R}\Psi_{\R(\gamma)}$-mouse projecting to $\delta_{\lambda^\R}$ with strategy $\Xi$ in $\Omega$ and $\M \ntrianglelefteq \R$ (again, $\Xi$ acts on trees above $\delta_{\lambda^\R}$ and moves the predicates for $\varoplus_{\gamma<\lambda^\R}\Psi_{\R(\gamma)}$ correctly). We can compare $\M$ and $\R$ (the comparison is above $\delta_{\lambda^\R}$). Let $\overline{\M}$ be the last model on the $\M$ side and $\overline{\R}$ on the $\R$ side. Then $\overline{\R} \lhd \overline{\M}$. Let $\pi: \R \rightarrow \overline{\R}$ be the iteration map from the comparison process and $\Sigma$ be the $\pi\circ k$-pullback of the strategy of $\overline{\R}$. Hence $\Sigma \in \Omega$ since $\Xi \in \Omega$. $\Sigma$ acts on trees above $\delta_{\lambda^\Q}$ and moves the predicate for $\varoplus_{\gamma<\lambda^\Q}\Psi_{\gamma}$ correctly by by our assumption on $\Xi$ and branch condensation of $\varoplus_{\gamma<\lambda^\Q}\Psi_{\gamma}$. These properties of $\Sigma$ imply that $\Q\lhd Lp^{\Omega,\varoplus_{\gamma<\lambda^\Q}\Psi_\gamma}(\cup_{\gamma<\lambda^\Q}\Q_\gamma)$. Contradiction. For the case that there are $\alpha<\lambda^\R$, $\delta_\alpha^\R \leq \eta < \delta^\R_{\eta+1}$, and $\eta$ is a strong cutpoint of $\R$, and $\M$ is a sound $\Psi_{\R(\alpha)}$-mouse projecting to $\eta$ with iteration strategy in $\Omega$, the proof is the same as that of Theorem 3.7.6 in \cite{ATHM}.
\\
\indent Now we show $\Lambda$ has branch condensation (see Figure \ref{proof of branch condensation}). The proof of this comes from private conversations between the author and John Steel. We'd like to thank him for this. For notational simplicity, we write $\Lambda^-$ for $\varoplus_{\gamma<\lambda^\Q}\Psi_{\gamma}$. Hence, $\Lambda \notin \Omega$ and $\Lambda^-\in\Omega$. Suppose $\Lambda$ does not have branch condensation. We have a minimal counterexample as follows: there are an iteration $i: \Q \rightarrow \R$ by $\Lambda$, a normal tree $\mathcal{U}$ on $\R$ in the window $[\xi, \gamma)$ where $\xi < \gamma$ are two consecutive Woodins in $\R$ such that sup$i''\delta_{\lambda^\Q} \le \xi$, two distinct cofinal branches of $\mathcal{U}$: $b$ and $c = \Lambda_\R(\mathcal{U})$, an iteration map $j: \Q \rightarrow \S$ by $\Lambda$, and a map $\sigma: \M^{\mathcal{U}}_b \rightarrow \S$ such that $j = \sigma \circ i^{\mathcal{U}}_b \circ i$. We may also assume that if $\overline{\R}$ is the first model along the main branch of the stack from $\Q$ to $\R$ giving rise to $i$ and $i_{\overline{\R},\R}: \overline{\R}\rightarrow \R$ be the natural map such that $i_{\overline{\R},\R}(\overline{\xi}) = \xi$ and $i_{\overline{\R},\R}(\overline{\gamma}) = \gamma$, then the extenders used to get from $\Q$ to $\overline{\R}$ have generators below $\overline{\xi}$. This gives us sup$(Hull^\R(\xi\cup\{p\})\cap \gamma) = \gamma$ where $p$ is the standard parameter of $\R$. Let $\Phi = \Lambda_\S^{\sigma}$ and $\Phi^-=\oplus_{\xi<\lambda^{\M^\U_b}}\Phi_{\M^\U_b(\xi)}$. It's easy to see that $\Phi^-\in\Omega$. By the same proof as in the previous paragraph, $\Phi$ is $\Omega$-fullness preserving. This of course implies that $\M^\U_b$ is $\Omega$-full and $\Phi\notin \Omega$.

Now we compare $\M^\U_b$ and $\M^\U_c$. First we line up the strategies of $\M^\U_b|\delta(\U)$ and $\M^\U_c|\delta(\U)$ by iterating them into the ($\Omega$-full) hod pair construction of some $N^*_y$ (where $y$ codes $(x, \M^\U_c,$
\\$\M^\U_b)$). This can be done because the strategies of $\M^\U_b|\delta(\U)$ and of $\M^\U_c|\delta(\U)$ have branch condensation by Theorems 2.7.6 and 2.7.7 of \cite{ATHM}\footnote{We note here that suppose $(\P,\Sigma)$ is a hod pair and $\P\vDash \delta^\P$ has measurable cofinality. Then knowing that all ``lower level" strategies of all iterates of $(\P,\Sigma)$ has branch condensation does not tell us that $\Sigma$ itself has branch condensation.}. This process produces a single normal tree $\W$. Let $a = \Phi(\W)$ and $d = \Lambda_{\M^\U_c}(\W)$. Let $X = Hull^\R(\xi\cup\{p\})\cap \gamma$. Note that $(i^\W_a\circ i^\U_b)$''X $\subseteq \delta(\W)$ and $i^\W_d\circ i^\U_c$''X $\subseteq \delta(\W)$. Now continue lining up $\M^\W_a$ and $\M^\W_d$ above $\delta(\W)$ (using the same process as above). We get $\pi: \M^\W_a \rightarrow \K$ and $\tau: \M^\W_d \rightarrow \K$ (we indeed end up with the same model $\K$ by our assumption on the pair $(\Lambda,\Lambda^-)$). But then 
\begin{equation*}
(\pi\circ i^\W_a \circ i^\U_b)\textrm{''X} = (\tau\circ i^\W_d \circ i^\U_c)\textrm{''X}.
\end{equation*}
But by the fact that $(i^\W_a\circ i^\U_b)$''X $\subseteq \delta(\W)$ and $i^\W_d\circ i^\U_c$''X $\subseteq \delta(\W)$ and $\pi$ agrees with $\tau$ above $\delta(\W)$, we get
\begin{equation*}
(i^\W_a \circ i^\U_b)\textrm{''X} = (i^\W_d \circ i^\U_c)\textrm{''X}.
\end{equation*}
This gives $ran(i^\W_a) \cap ran(i^\W_d)$ is cofinal in $\delta(\W)$, which implies $a = d$. This in turns easily implies $b = c$. Contradiction.
\begin{figure}
\centering
\begin{tikzpicture}[node distance=3cm, auto]
  \node (A) {$\Q$};
  \node (B) [right of=A] {$\R$};
  \node (C) [right of=B] {$\M^{\U}_c$};
  \node (D) [right of=C] {$\M^{\W}_d$};
  \node (E) [right of=D] {$\K$};
  \node (F) [node distance=2cm, below of=C] {$\M^{\mathcal{U}}_b$};
  \node (G) [right of=F] {$\M^\W_a$};
  \node (H) [node distance=2cm, below of=F] {$\S$};
  \draw[->] (A) to node {$i$} (B);
  \draw[->] (B) to node  {$\U,c$} (C);
  \draw[->] (C) to node  {$\W,d$} (D);
  \draw[->] (D) to node {$\tau$} (E);
  \draw[->] (B) to node [swap] {$\U,b$} (F);
  \draw[->] (F) to node {$\W,a$} (G);
  \draw[->] (G) to node {$\pi$} (E);
  \draw[->] (A) to node {$j$} (H);
  \draw[->] (F) to node {$\sigma$} (H);
\end{tikzpicture}
\caption{The proof of branch condensation of $\Lambda$ in Lemma \ref{direct limit exists}}
\label{proof of branch condensation}
\end{figure}
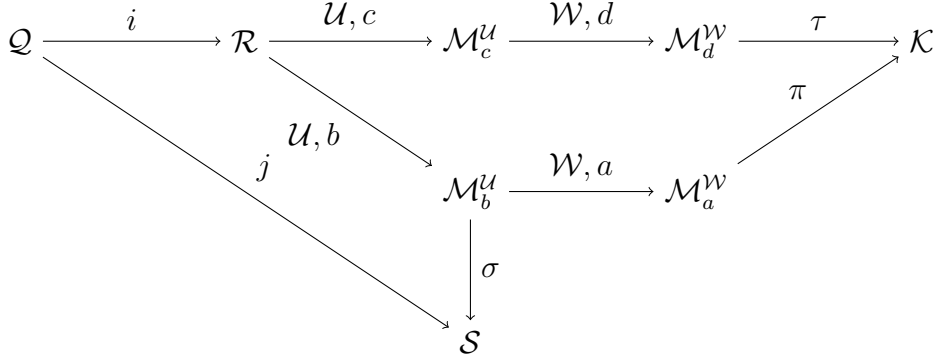
\indent Finally, let $\R$ and $\S$ be $\Lambda$-iterates of $\Q$ and let $\Lambda_\R$ and $\Lambda_\S$ be the tails of $\Lambda$ on $\R$ and $\S$ respectively. We want to show that $\R$ and $\S$ can be further iterated (using $\Lambda_\R$ and $\Lambda_\S$ respectively) to the same model. To see this, we compare $\R$ and $\S$ against the $\Omega$-full hod pair construction of some $N^{*}_y$ (for some $y$ coding $(x, \R, \S)$). Then during the comparison, only $\R$ and $\S$ move (to say $\R^*$ and $\S^*$). It's easy to see that $\R^* = \S^*$ and their strategies are the same (as the induced strategy of $N^*_y$ on its appropriate background construction).
\end{proof}
\indent By the properties of $(\Q,\Psi)$ and $\Lambda$, we get that $\rho(\mathcal{M}_\infty(\Q,\Lambda)) \leq \Theta$ and $(\textrm{HOD}|\Theta)^N = \M_\infty(\Q,\Lambda)|\Theta$. Let $k$ be the least such that $\rho_{k+1}(\Q) \leq \delta_{\lambda^\Q}$.
\\
\\
\textit{Claim. $\mathcal{M}_\infty(\Q,\Lambda) \notin N$}
\begin{proof}
Suppose not. Let $i: \Q \rightarrow \M_\infty(\Q,\Lambda)$ be the direct limit map according to $\Lambda$. By an absoluteness argument (i.e. using the absoluteness of the illfoundedness of the tree built in $N[g]$ for $g\subseteq Col(\omega, |\M_\infty(\Q,\Lambda)|)$ generic over $N$ of approximations of a embedding from $\Q$ into $\mathcal{M}_\infty(\Q,\Xi)$ extending the iteration embedding according to $\varoplus_{\beta<\lambda^\Q}\Psi_\beta$ on $\Q|\delta_{\lambda^\Q}$), we get a map $\pi$ such that
\begin{enumerate}
\item $\pi\in N$
\item $\pi: \Q \rightarrow \mathcal{M}_\infty(\Q,\Lambda)$;
\item for each $\beta < \lambda^\Q$, $\pi|Q(\beta)$ is according to $\Psi_\beta$.
\item $\pi(p) = i(p)$ where $p = p_k(\Q)$.
\end{enumerate}
This implies that $\pi = i \in N$ since $\Q$ is $\delta_{\lambda^\Q}$-sound and $\rho(\Q) \leq \delta^{\lambda^\Q}$. But this map determines $\Lambda$ in $N$ as follows: let $\mathcal{T}\in N$ be countable and be according to $\Lambda$, $N$ can build a tree searching for a cofinal branch $b$ of $\mathcal{T}$ along with an embedding $\sigma:\M^{\mathcal{T}}_b\rightarrow \M_\infty(\Q,\Lambda)$ such that $\pi = \sigma\circ i^{\mathcal{T}}_b$. Using the fact that $\Lambda$ has branch condensation, we easily get that $\Lambda \in N$. But this is a contradiction.
\end{proof}
Returning to the proof of (2), let $j =_{def} j_\mu: \textrm{HOD} \rightarrow Ult_0(\textrm{HOD},\mu)$ and $\W = j(\langle \M_\beta \ | \ \beta < \alpha, \beta \rm{ \ is\ limit}\rangle)(\alpha)$. Let $i: \M_\infty(\Q,\Lambda) \rightarrow Ult_k(\M_\infty(\Q,\Lambda), \mu)$ be the canonical map. Note that $A \notin \M_\infty(\Q,\Lambda)$. To see this, assume not, let $\R \lhd \M_\infty(\Q,\Lambda)$ be the first level $\S$ of $\M_\infty(\Q,\Lambda)$ such that $A$ is definable over $\S$. 

We claim that $\R \in N$. Recall that $\W$ is the first level of $\M_\infty^+$ such that $j(A)\cap \Theta$ is definable over $\W$. Now let 
\begin{center}
$k: \R \rightarrow Ult_0(\R,\mu) =_{def} \R^*$
\end{center}
be the $\Sigma_0$-ultrapower map. By the definition of $\W$ and $\R^*$ and the fact that they are both countably iterable, we get that $\W = \R^* \in N$. Let $p$ be the standard parameters for $\R$. In $N$, we can compute $Th_0^\R(\Theta\cup p)$ as follows: for a formula $\psi$ in the language of hod premice and $s\in \Theta^{<\omega}$,
\begin{center}
$(\psi,s)\in Th_0^\R(\Theta\cup p)\Leftrightarrow (\psi, j(s)) \in Th_0^{\R^*}(\Theta\cup k(s))$.
\end{center}
Since $Th_0^{\R^*}(\Theta\cup k(s)) = Th_0^{\W}(\Theta\cup k(s))\in N$, $j|\Theta\in N$, and $k(s)\in \W\in N$, we get $Th_0^\R(\Theta\cup p)\in N$. This shows $\R\in N$.

To get a contradiction, we show $\R \lhd \N_\infty^+$ by showing $\R$ is satisfies property (\textasteriskcentered) in $N$. Let $\K$ be a countable mouse embeddable into $\R$ by a map $k \in N$. Then we can compare $\K$ and $\Q$ against the $\Omega$-full hod pair construction of some $N^*_y$ just like in the argument on the previous page; hence we may assume $\K \lhd \Q$ ($\Q \unlhd \K$ can't happen because then $\Lambda\in N$). The minimality assumption on $\Q$ easily implies $\K \lhd Lp^{\Omega, \varoplus_{\gamma<\lambda^{\Q}}\Psi_\gamma}(\Q|\delta_{\lambda^{\Q}})$. But then $N$ can iterate $\K$ for stacks on $\K$ above $\delta_{\lambda^\Q} = \delta_{\lambda^{\K}}$, which is what we want to show. The fact that $\R \lhd \N^+_\infty$ contradicts $A \notin \N_\infty^+$.
\\
\indent Next, we note that $Ult_0(\textrm{HOD},\mu)|\Theta = Ult_k(\M_\infty(\Q,\Lambda),\mu)|\Theta$ and $i|\Theta = j|\Theta$. Let $\R = Th^{\M_\infty(\Q,\Lambda)}(\Theta \cup \{p\})$ where $p = p_k(\M_\infty(\Q,\Lambda))$ and $\S = Th^{Ult_k(\M_\infty(\Q,\Lambda), \mu)}(\Theta \cup \{i(p)\})$. We have that $\M_\alpha$ and $\S$ are sound hybrid mice in the same hierarchy, hence by countable iterability, we can conclude either $\M_\alpha \vartriangleleft \S$ or $\S \unlhd \M_\alpha$.
\\
\indent If $\M_\alpha \vartriangleleft \S$, then $\M_\alpha \in Ult_k(\M_\infty(\Q,\Lambda), \mu)$. This implies $A \in \M_\infty(\Q,\Lambda)$ by a computation similar to that in the proof of (1), i.e.
\begin{equation*}
\beta \in A \Leftrightarrow \M_\infty(\Q,\Lambda) \vDash (i|\Theta)(\beta) \in \M_\alpha.
\end{equation*}
This is a contradiction to the fact that $A \notin \M_\infty(\Q,\Lambda)$. Now suppose $\S \unlhd \M_\alpha$. This then implies $\S \in Ult_0(\textrm{HOD},\mu)$, which in turns implies $\M_\infty(\Q,\Lambda) \in \textrm{HOD}$ by the following computation: for any formula $\phi$ and $s \in \Theta^{<\omega}$,
\begin{equation*}
(\phi, s) \in \R \Leftrightarrow \textrm{HOD} \vDash (\phi, (j|\Theta)(s)) \in \S.
\end{equation*}
This is a contradiction to the claim. This completes the proof of (2).
\end{proof}
Theorem \ref{SingularMeasurable} completes our analysis of \textrm{HOD} for determinacy models of the form ``$V=L(\powerset(\mathbb{R}))$ below $``\textsf{AD}_{\mathbb{R}}+\Theta$ is regular."

\section{Questions and open problems}
\noindent \textbf{Question 1.} Assume (\textasteriskcentered) and $\Theta = \theta_{\alpha+1}$. Can one carry out the HOD analysis similar to that of Section 3?
\\
\\
The following question is also natural.
\\
\\
\noindent \textbf{Question 2.} Assume $\textsf{AD}^+ + V=L(\powerset(\mathbb{R}))$. Does \textrm{HOD} satisfy G\textsf{CH}?
\\
\\
\noindent More generally (and vaguely), we can ask whether \textrm{HOD} is a fine-structural model. As shown in \cite{ATHM} and in this paper, under (\textasteriskcentered), \textrm{HOD} is indeed a fine-structural (hybrid) model. The next natural determinacy theory to aim to understand \textrm{HOD} for seems to be the theory ``$\textsf{AD}^+ + \Theta=\theta_{\alpha+1} + \theta_\alpha$ is the largest Suslin cardinal." It's not known whether this theory is consistent (relative to large cardinals). Recent work suggests that this theory is consistent relative to a Woodin limit of Woodin cardinals. 

\bibliographystyle{plain}
\bibliography{Rmicebib}
\end{document}